\input ./style/arxiv-general.cfg
\documentclass[aap,MSNbibl,nameyear,seceqn,dvips]{arximspdf}
\makeatletter
   \@ifpackageloaded{graphicx}{}{\usepackage{graphicx}}
\makeatother
\usepackage{multirow}
\usepackage{colortbl}

\doi{10.1214/14-AAP1070}
\volume{25}
\issue{6}
\pubyear{2015}
\firstpage{3162}
\lastpage{3208}
\docsubty{FLA}

\makeatletter
\newcommand{\iint}{\int\!\!\!\int}
\newcommand{\iiint}{\int\!\!\!\int\!\!\!\int}
\def\xrightarrow{\longrightarrow}

\newcommand{\eqref}[1]{(\ref{#1})}
\newtheorem{proposition}{Proposition}[section]
\newtheorem{corollary}[proposition]{Corollary}
\newtheorem{lemma}[proposition]{Lemma}
\newtheorem{theorem}[proposition]{Theorem}
\newproclaim{assumption}[proposition]{Assumption}
\newproclaim{conjecture}[proposition]{Conjecture}
\newproclaim{definition}[proposition]{Definition}
\newproclaim{remark}[proposition]{Remark}
\newproclaim{example}[proposition]{Example}

\makeatother

\begin{document}
\begin{frontmatter}

\title{Scaling limits of spatial compartment models for
chemical reaction networks}
\runtitle{Scaling limits of spatial compartment models}

\begin{aug}
\author[A]{\fnms{Peter}~\snm{Pfaffelhuber}\ead
[label=e1]{p.p@stochastik.uni-freiburg.de}}
\and
\author[B]{\fnms{Lea}~\snm{Popovic}\corref{}\ead[label=e2]{lpopovic@mathstat.concordia.ca}\thanksref{T1}}
\runauthor{P. Pfaffelhuber and L. Popovic}
\thankstext{T1}{Supported by the University Faculty award and Discovery
grant of NSERC (Natural Sciences and Engineering Council of Canada).}
\affiliation{Albert-Ludwigs-Universit\"at Freiburg and Concordia University}
\address[A]{Abteilung f\"ur Mathematische Stochastik\\
Albert-Ludwigs-Universit\"at Freiburg\\
Eckerstr. 1\\
79104 Freiburg\\
Germany\\
\printead{e1}}
\address[B]{Department of Mathematics and Statistics\\
Concordia University\\
Montreal QC H3G 1M8\\
Canada\\
\printead{e2}}
\end{aug}

%
\received{\smonth{2} \syear{2013}}
%
\revised{\smonth{7} \syear{2014}}

%
\begin{abstract}
We study the effects of fast spatial movement of molecules on the
dynamics of chemical species in a spatially heterogeneous chemical
reaction network using a compartment model.
The reaction networks we consider are either
single- or multi-scale. When reaction dynamics is on a single-scale, fast
spatial movement has a simple effect of averaging reactions
over the distribution of all the species. When reaction
dynamics is on multiple scales, we show that spatial
movement of molecules has different effects depending on whether the
movement of each type of species is faster or slower than the
effective reaction dynamics on this molecular type.
We obtain results for both when the system is without and with
conserved quantities, which are linear combinations of species
evolving only on the slower time scale.
\end{abstract}

%

%
\end{frontmatter}

\section{Introduction}\label{sec1}
When chemical species react, they are present in some (open or closed)
system with a spatial dimension. Most models of chemical reaction
systems describe the evolution of the concentration of chemical
species and ignore both stochastic and spatial effects inherent in the
system. This can be justified by law of large number results when
both: the number of species across all molecular types is large, and
when the movement of molecules within the system is much faster than
the chemical reactions themselves. In applications where these
assumptions hold, the system is spatially homogeneous, and the use of
the deterministic law of mass action kinetics is approximately
appropriate [\citet{Kurtz1970}].

However, in biological cells, low numbers of certain key chemical
species involved in the reaction systems result in appreciable noise
in gene expression and many regulatory functions of the cell, and lead
to cell--cell variability and different cell fate decisions
[\citet{McadamsArkin}, \citet{ElowitzSwain}]. In order to
understand the key effects of intrinsic noise in chemical reaction
networks on the overall dynamics, one needs to derive new
approximations of stochastic models describing the evolution of
molecular counts of chemical species. Furthermore, the cell is not a
spatially homogeneous environment, and it has been repeatedly
demonstrated that spatial concentration of certain molecules plays an
important role in many cellular processes [\citet{HowardRutenberg},
\citet{Takahashi}]. Deterministic spatial models
(reaction--diffusion PDEs) are insufficient for this purpose, since
local fluctuations of small molecular counts can have propagative
effects, even when the overall total number of molecules in the cell
of the relevant species is high.

Numerical simulations of biochemical reactions are indispensable since
the stochastic spatial dynamics of most systems of interest is
analytically intractable. A number of different simulation methods
have been developed for this purpose, ranging from exact methods
(accounting for each stochastic event) to approximate methods
(replacing exact stochastics for some aspects of the system with
approximate statistical distributions); see, for example,
\citet{FangeElf}, \citet{PetzoldKhammash}, \citet{JeschkeUhrmacher}. For
effective computations, a mesoscopic form of the full stochastic
spatial reaction model is necessary. These are compartment models in
which a heterogeneous system is divided into homogeneous subsystems in
each of which a set of chemical reactions are performed. Molecular
species are distributed across compartments and their diffusion is
modeled by moves between neighboring compartments
[\citet{BurrageBurrage}].

An important mathematical feature of models of biochemical reactions
lies in the essential multi-scale nature of the reaction processes.
In some cases, all chemical species are present in a comparable amount
and change their concentrations on the same temporal scale. We call
this a \emph{single-scale} reaction network. However, if at least one
chemical species changes its abundance (relative to its abundance) on
a much faster time scale, we call this a \emph{multi-scale} reaction
network. The fast change of the concentrations of some chemical
species then has an impact on the dynamics of the slow species. When
one adds spatial movement of species into the system, then the species
with fast movement can have an averaging effect on the dynamics as
well. The overall dynamics depends on how these two averaging factors
interact, and subsequently determines the evolution of the slow
species on the final time scale of interest.

In this paper, we analyze the effects of movement of molecules on the
dynamics of the molecular counts of chemical species. We consider a
finite number of compartments in which different reactions can happen,
with species moving between compartments, and derive results for the
evolution of the sum total of molecules in all compartments. We
consider the following for chemical reactions within compartments: (i)
single-scale chemical reactions, (ii) multi-scale chemical reactions
without conserved quantities and (iii) multi-scale chemical reactions
with conserved quantities. We focus on the derivation of simplified
models obtained as limits of rescaled versions of the original model
[in the spirit of \citet{KangKurtzPopovic2012},
\citet{BallKurtzPopovicRempala2006}, \citet{Zeiser2012}]. We
stress that our results are for mesoscopic models of spatial systems,
as opposed to models in which the number of compartments increases and
the size of compartments shrinks [\citet{Blount},
\citet{Kouritzin}, \citet{Kotelenez}].

Our goal is to find reduced models that capture all relevant stochastic
features of the original model, while focusing only on the quantities that
are easy to measure (sum total of molecular numbers for each species)
in the system.
Our results for this reduced dynamics make stochastic simulations
almost trivial,
while at the same time differentiating (being able to detect) between
different cases
of system heterogeneity and between different relative time scales of
species movement.

\subsection{Outline of results}
 After introducing a model for chemical reactions in
Section~\ref{ss:MarkovChain}, we provide results in nonspatial
systems which we need later for spatial results: asymptotics for a
single-scale reaction network (Lemma~\ref{lem:crnlimit-ss}),
asymptotics for two-scale systems without
(Lemma~\ref{lem:crnlimit-ts}) and with conserved
quantities on the fast time scale (Lemma~\ref{lem:crnlimit-cons}), as
well as extensions to three-scale systems
(Lemma~\ref{lem:crnlimit-ts3}). Examples of each type of system are
provided as well (Examples~\ref{ex:self}, \ref{ex:production},
\ref{ex:michmen}). In Section~\ref{S:3}, we give results for
spatial compartment models: for single-scale spatial systems
(Theorem~\ref{T1}), for two-scale spatial systems
(Theorem~\ref{thm:T2}) in the absence of conserved quantities, and in
the presence of conserved quantities (Theorem~\ref{thm:T3}). We also
place the earlier examples in a spatial setting
(Examples~\ref{ex:selfSp}, \ref{ex:productionSp},
\ref{ex:michmenSp}). We conclude the paper with a discussion
of possible implications and extensions.

\begin{remark}[(Notation)]
For some Polish space $E$, we denote the set of continuous (bounded
and continuous, continuous with compact support), real-valued
functions by $\mathcal C(E)$ [$\mathcal C_b(E)$, $\mathcal
C_c(E)$]. In general, we write $\underline x:= (x_k)_k$ for vectors
and $\underline{\underline x} := (x_{ik})_{i,k}$ for matrices. In
addition, $\underline x_{i\cdot}$ is the $i$th line and $\underline
x_{\cdot k}$ is the $k$th column of $\underline{\underline x}$. We
denote by $\mathcal D(I; E)$ the space of c\`adl\`ag functions
$I\subseteq\mathbb R \to E$, which is equipped, as usual, with the
Skorohod topology, metrized by the Skorohod metric, $d_{Sk}$. For
sets $F, F' \subseteq E$, we write $F-F' := \{f\in F\dvtx f\notin F'\}$.
\end{remark}

\section{Chemical reactions in a single compartment}
\label{S:resultsSingle}
Before we present our main results on spatial systems in the next
section, we provide basic results on reaction networks in a
single compartment, which are special cases of theorems given in
\citet{KangKurtz2013}. Our results for the spatial case both rely on them
and are proved using similar techniques.

We consider a set $\mathcal I$ of different chemical
species, involved in $\mathcal K$ different reactions of the form
%
\begin{equation}
\label{eq:CRN} (\nu_{ik})_{i\in\mathcal I} \to \bigl(
\nu'_{ik} \bigr)_{i\in\mathcal I}
\end{equation}
with $\underline{\underline{\nu}} = (\nu_{ik})_{i\in\mathcal I,
k\in\mathcal K}, \underline{\underline{\nu'}}=
(\nu'_{ik})_{i\in\mathcal I, k\in\mathcal K}\in\mathbb
Z_+^{\mathcal
I\times\mathcal K}$ and $\nu_{ik}=l$ if $l$ molecules of the
chemical species $i$ take part in reaction $k$ and $\nu'_{ik}=l$ if
reaction $k$ produces $l$ molecules of species $i$. In the chemical
reaction literature, $
\underline{\underline\zeta}=\underline{\underline{\nu
'}}-\underline
{\underline{\nu}}$
is called the \textit{stoichiometric matrix} of the system, and
$\sum_{i\in\mathcal I} \nu_{ik}$ the \textit{order} of reaction $k$. In
addition, we set $\underline\zeta_{\cdot k} :=
(\zeta_{ik})_{i\in\mathcal I}$.

\subsection{The Markov chain model and the rescaled system}
\label{ss:MarkovChain}
Denoting by $X_i(t)$ the number of molecules of species $i$ at time
$t$, we assume that $(\underline X(t))_{t\geq0}$ with $\underline
X(t) = (X_i(t))_{i\in\mathcal I}$ is solution of
%
\begin{equation}
\label{eq:crn00} %
X_{i}(t)  =
X_{i}(0) + \sum_{k\in\mathcal K}
\zeta_{ik} Y_{k} \biggl(\int_0^t
\Lambda_{k}^{\mathrm{CR}} \bigl(\underline X(u) \bigr) \,du \biggr),
\end{equation}
where the $Y_k$'s are independent (rate 1) Poisson processes and
$\Lambda_k^{\mathrm{CR}}(\underline X(u))$ is the reaction rate of
reaction $k$ at time $u$, $k\in\mathcal K$. We will assume throughout
the following.

\begin{assumption}[(Dynamics of unscaled single compartment reaction
network)]
The \label{ass:21} reaction network dynamics satisfies the following
conditions:
\begin{longlist}[(ii)]
\item[(i)] The reaction rate $\underline x\mapsto
\Lambda_k^{\mathrm{CR}}(\underline x)$ is a nonnegative locally
Lipshitz, locally bounded function and $\Lambda_k^{\mathrm{CR}}\neq
0$, $k\in\mathcal K$.
\item[(ii)] Given $(Y_k)_{k\in\mathcal K}$, the
time-change equation \eqref{eq:crn00} has a unique solution.
\end{longlist}
\end{assumption}

 The most important chemical reaction kinetics is given by
mass action, that is,
%
\begin{equation}
\label{eq:mak1} \Lambda_{k}^{\mathrm{CR}}(\underline x) =
\kappa'_{k} \prod_{i\in\mathcal I}
\nu_{ik}!\pmatrix{x_{i}
\cr
\nu_{ik}}
\end{equation}
for constants $\kappa_{k}$. In other words, the rate of reaction $k$
is proportional to the number of possible combinations of reacting
molecules.
Solutions to \eqref{eq:crn00} can be guaranteed by using, for example,
\citet{EthierKurtz1986}, Theorem~6.2.8; see also their Remark~6.2.9(b). Note, however, that Assumption~\ref{ass:21}(i) does not
suffice to guarantee a global solution to \eqref{eq:crn00}, since it
has to be certain---usually by imposing some growth condition---that
the solution does not become infinite in finite time.

Chemical reaction networks in many applications involve chemical
species with vastly differing numbers of molecules and reactions with
rate constants that also vary over several orders of magnitude
\citet{BallKurtzPopovicRempala2006}, Examples. This wide variation in
number and rate yield phenomena that evolve on very different
time scales. Recognizing that the variation in time scales is due both
to variation in species number and to variation in rate constants, we
normalize species numbers and rate constants by powers of a parameter
$N$ which we assume to be large, and consider a sequence of models,
parametrized by $N\in\mathbb N$. Rescaled versions of the original
model, under certain assumptions, have a limit as $N\to\infty$. We
will use stochastic equations of the form (\ref{eq:crn00}) driven by
independent Poisson processes to show convergence, exploiting the law
of large numbers and martingale properties of the Poisson
processes. We rely heavily on the stochastic averaging methods that
date back to Khasminskii, for which we follow the formalism in terms
of martingale problems from \citet{Kurtz1992}.

We rescale the system as follows: consider the solution $(\underline
X^N(t))_{t\geq0}$ of \eqref{eq:crn00} with the chemical reaction
rates $\Lambda^{\mathrm{CR}}_{k}$ replaced by
$\Lambda^{\mathrm{CR},N}_{k}$. For real-valued
\begin{equation}
\label{eq:scalingParms} \underline\alpha= (\alpha_i)_{i\in\mathcal I},\qquad
\underline\beta= (\beta_k)_{k\in\mathcal K},\qquad \gamma
\end{equation}
with $\alpha_i\geq0, i\in\mathcal I$, we denote the $(\underline
\alpha, \underline\beta, \gamma)$-rescaled system by
%
\begin{equation}\quad
\label{eq:resc0} V_{i}^N(t) := N^{-\alpha_i}
X_{i}^N \bigl(N^\gamma t \bigr),\qquad  
 \lambda^{\mathrm{CR},N}_{k}(\underline v):=N^{-\beta_k}\Lambda
^{\mathrm{CR},N}_{k} \bigl( \bigl(N^{\alpha_i}v_i
\bigr)_{i\in\mathcal I} \bigr), 
\end{equation}
where $(\underline\alpha, \underline\beta, \gamma)$ is chosen so
that: $V^N_{i}=\mathcal O(1), i\in\mathcal I$, for all time (a.s.
does not go infinity in finite time, but also does not have a.s. zero
limit for all time), and $\lambda^{\mathrm{CR},N}_{k}=\mathcal O(1),
k\in\mathcal K$ (for all values of $\underline v$ when it is not equal
to zero). We will restrict to the case $\gamma=0$ which can always be
achieved when considering $\beta_k' = \beta_k+\gamma, k\in\mathcal
K$. From \eqref{eq:crn00}, we see that the $(\underline\alpha,
\underline\beta, \gamma)$-rescaled system $\underline
V^N:=(\underline
V^N(t))_{t\geq0}$ is a solution to the system of stochastic equations
%
\begin{equation}
\label{eq:crn-ss} %
\begin{aligned} V^N_{i}(t) & =
V^N_{i}(0) + \sum_{k\in\mathcal K}
N^{-\alpha_i} \zeta_{ik}Y_{k} \biggl(N^{\beta_k + \gamma}
\int_0^t \lambda_{k}^{\mathrm{CR},N}
\bigl(\underline V^N(u) \bigr) \,du \biggr), \end{aligned} %
\end{equation}
which has a unique solution thanks to Assumption~\ref{ass:21}. In
vector notation, we use the diagonal matrix $N^{-\underline\alpha}$
with $i$th diagonal entry $N^{-\alpha_i}$ and write
%
\begin{equation}
\label{eq:crn-ss2} \underline V^N(t) = \underline V^N(0) +
\sum_{k\in\mathcal K} N^{-\underline\alpha}\underline
\zeta_{\cdot k} Y_{k} \biggl(N^{\beta_k + \gamma} \int
_0^t \lambda_{k}^{\mathrm{CR},N}
\bigl(\underline V^N(u) \bigr) \,du \biggr).
\end{equation}

The reaction rates satisfy the following.

\begin{assumption}[(Dynamics of scaled single compartment reaction
network)]
There \label{ass:22} exist locally Lipshitz functions
$\lambda^{\mathrm{CR}}_{k}\dvtx \mathbb R_+^{\mathcal I}\to\mathbb R_+$,
$k\in\mathcal K$ with
%
\begin{equation}
\label{eq:resc10} %
\begin{aligned} N^{-\beta_k}\Lambda^{\mathrm{CR},N}_{k}
\bigl( \bigl(N^{\alpha_i}v_i \bigr)_{i\in\mathcal I} \bigr)
\mathop{\xrightarrow}^{N\to\infty} \lambda^{\mathrm{CR}}_{k}(\underline v)
\end{aligned} %
\end{equation}
uniformly on compacts. [Without loss of generality, we will assume that
convergence in \eqref{eq:resc10} is actually an identity;
our results easily generalize by the assumed uniform convergence on compacts.]
\end{assumption}

 In the special case of mass action
kinetics (\ref{eq:mak1}), if $\alpha_i=1$ for all $i\in\mathcal I$ and
$\kappa_{k} = \kappa_{k}'
N^{-(\sum_i\nu_{ik})+1}$
with $\beta_k = 1$ and some $\kappa_{k}'>0$ for all $k\in\mathcal
K$, then
\begin{eqnarray*}
N^{-\beta_k} \Lambda_{k}^{\mathrm{CR},N} \bigl(
\bigl(N^{\alpha_i}v_{i} \bigr)_{i\in\mathcal
I} \bigr)
\mathop{\xrightarrow}^{N\to\infty} \kappa_{k}' \prod
_{i\in\mathcal I} v_i^{\nu_{ik}}.
\end{eqnarray*}
The polynomial on the right-hand side is known in the literature for
deterministic chemical reaction systems as the \textit{mass action
kinetic} rate.


\subsection{Single scale systems}
\label{Sec:23}
For $i\in\mathcal I$, the set of reactions which change the number of species
$i$ is
\[
\mathcal K_i:=\{k\in\mathcal K\dvtx \zeta_{ik}
\neq0\}
\]
(a reaction of the form $A+B\to A+C$ does not change
the number of species $A$). A chemical reaction network is a
\textit{single scale system} if $(\underline\alpha,
\underline\beta, \gamma)$ from \eqref{eq:scalingParms} satisfy
%
\begin{equation}
\label{eq:single-scale} \max_{k\in\mathcal K_i}\beta_k +\gamma=
\alpha_i,\qquad i\in\mathcal I.
\end{equation}
For $i\in\mathcal I$, let $\mathcal K^*_i\subseteq\mathcal K_i$ be the
set of reactions such that $\beta_k +\gamma= \alpha_i$, and let
$\mathcal K^*=\bigcup_{i\in\mathcal I}\mathcal K^*_i$. Define
$\underline{\underline\zeta}^*$ by
\[
\zeta^*_{ik}=\lim_{N\to\infty}N^{-\alpha_i}N^{\beta_k+\gamma
}
\zeta_{ik}.
\]
Then $\underline{\underline\zeta}^*$ is the matrix whose
$i\in\mathcal I, k\in\mathcal K_i^*$ entries are $\zeta_{ik}$ and its
$i\in\mathcal I, k\in\mathcal K_i- \mathcal K_i^*$ entries are
zero. Let $\mathcal I_\circ$ be the subset of species with
$\alpha_i=0$, called the \textit{discrete species}, and let $\mathcal
K^*_\circ=\bigcup_{i\in\mathcal I_\circ} \mathcal K^*_i$, called
the \emph{slow reactions}. Let $\mathcal I_\bullet$ be the subset of
species with $\alpha_i>0$, called the \emph{continuous species}, and
let $\mathcal K^*_\bullet=\bigcup_{i\in\mathcal I_\bullet} \mathcal
K^*_i$, called the \emph{fast reactions}. Then $\mathcal K^*=\mathcal
K^*_\circ\cup\mathcal K^*_\bullet$. Note that by definition
$\mathcal I_\circ$ and $\mathcal I_\bullet$ are disjoint, and by
definition of $\mathcal K^*_i$ (and as reaction rates come with a
single scaling $N^{\beta_k+\gamma}$), $\mathcal K^*_\circ$ and
$\mathcal K^*_\bullet$ are also disjoint. In the limit of the
rescaled system, the species indexed by $\mathcal I_\circ$ are
$\mathbb
Z_+$-valued (hence the name discrete species), while the species
indexed by $\mathcal I_\bullet$ are $\mathbb R_+$-valued (continuous
species). See Table~\ref{tab:1} for an overview of these
definitions. We next assume the following.

\definecolor{gray}{gray}{0.75}
\begin{table}
\tabcolsep=0pt
\caption{An overview of different sets and
possibilities in the case $\gamma=0$.
The set $\mathcal I$ is split into discrete ($\mathcal I_\circ$)
and continuous ($\mathcal I_\bullet$) chemical species,
while the set $\mathcal K^{\ast}$ is split
into slow ($\mathcal K^\ast_\circ$) and fast ($\mathcal K^\ast
_\bullet
$) reactions.
The gray boxes give the reactions which still appear in
the limit dynamics.
A special feature of single-scale systems is that
discrete species are exactly changed through slow reactions, and
continuous species are changed
by fast reactions. In particular, discrete species are not changed
by fast reactions. This is different in
multi-scale networks; see Table \protect\ref{tab:2}}\label{tab:1}
\begin{tabular*}{\textwidth}{@{\extracolsep{\fill}}lcc@{}}
\hline
& \multicolumn{1}{c}{\textbf{Slow reactions}} & \multicolumn{1}{c}{\textbf{Fast reactions}} \\
& \multicolumn{1}{c}{$\bolds{k\in\mathcal K^\ast_\circ, \beta_k=0}$} &
\multicolumn{1}{c}{$\bolds{k\in\mathcal K^\ast_\bullet, \beta_k>0}$} \\
\hline
\multicolumn{1}{@{}l}{\multirow{2}{90pt}[8pt]{Discrete species,
$\alpha_i=0, i \in\mathcal I_\circ$}}
& \multicolumn{1}{>{\columncolor[gray]{0.8}[-2pt][0pt]}c}{$\zeta_{ik}^\ast
\cases{\neq0,
&\quad $k\in\mathcal K_i^\ast$,\vspace*{2pt}\cr
=0, &\quad $\mbox{else}$}$}
&
$\zeta_{ik} = \zeta_{ik}^\ast= 0$\\[12pt]
\multicolumn{1}{@{}l}{\multirow{2}{90pt}[8pt]
{Continuous species, $\alpha_i>0, i \in\mathcal
I_\bullet$}}&
\multicolumn{1}{c}{\multirow{2}{130pt}{\centering $\zeta_{ik}=0$ or $\beta_k+\gamma-\alpha_i<0$
$\Longrightarrow\zeta^\ast_{ik}=0$}}&
\multicolumn{1}{>{\columncolor[gray]{0.8}[0pt][0pt]}c}{$\zeta_{ik}^\ast
\cases{\neq0, &\quad
$k\in\mathcal K_i^\ast$,\vspace*{2pt}\cr=0, &\quad $\mbox{else}$}
$}\vspace*{6pt}\\
\hline
\end{tabular*}
\end{table}
%

%
\begin{table}
\tabcolsep=0pt
\caption{As in the single-scale case, the set $\mathcal I$ is split into discrete
and continuous chemical species. In addition, discrete and continuous species
are either changed on the fast or slow time scale. (This means that
$\mathcal I^f_\circ, \mathcal I^f_\bullet, \mathcal I^s_\circ,
\mathcal I^s_\bullet$ are disjoint sets.) The set of reactions is
split into
several categories, which can overlap. Here, $k\in\mathcal K_\circ^f$
is a reaction which changes a discrete species on the fast time scale, etc.
Note that such a reaction can as well change a continuous species on the
slow time scale. The separation of fast and slow time scales is
determined by (\protect\ref{eq:two-scale})
with $\varepsilon=1$.
As in Table \protect\ref{tab:1}, we mark the cells which finally
determine the dynamics of the limiting object}
\label{tab:2}
\begin{tabular*}{\textwidth}{@{\extracolsep{4in minus 4in}}lccc@{}}
\multicolumn{4}{c}{\textbf{Fast time scale} $\bolds{N  \,dt}$}\\
\hline\\
&\multicolumn{1}{c}{\multirow{3}{75pt}[7pt]{\centering \textbf{Reactions of discrete
species on
fast scale} $\bolds{k\in\mathcal K^f_\circ, \beta_k=1}$}}
 &
\multicolumn{1}{c}{\multirow{3}{75pt}[7pt]{\centering \textbf{Reactions of continuous species on fast scale}
$\bolds{k\in\mathcal K^f_\bullet}$\textbf{,}
$\bolds{\beta_k>1}$}}
&
\multicolumn{1}{c@{}}{\multirow{3}{75pt}[2pt]{\centering \textbf{Other reactions}
$\bolds{k\in\mathcal K - \mathcal K^f_\circ- \mathcal K^f_\bullet}$\textbf{,} $\bolds{\beta
_k\geq0}$}}\\\\\\
\hline
\multicolumn{1}{@{}l}{\multirow{2}{90pt}[8pt]{Discrete species,
$\alpha_i=0, i = \mathcal I^f_\circ$}}
& \multicolumn{1}{>{\columncolor[gray]{0.8}[-2pt][0pt]}c}{$\zeta_{ik}^f
\cases{\neq0,
&\quad $k\in\mathcal K_i^f,$\vspace*{2pt}\cr
=0, &\quad $\mbox{else}$}$} &
$\zeta_{ik} = \zeta_{ik}^f = 0$
& $\zeta_{ik}^f = 0$
\\[12pt]
\multicolumn{1}{@{}l}{\multirow{2}{90pt}[8pt]{Continuous species,
$\alpha_i>0, i = \mathcal I_\bullet^f$}}
& \multicolumn{1}{c@{}}{\multirow{3}{80pt}[8pt]{\centering $\zeta_{ik}=0$ or
$\beta_k-\alpha_i-1<0$
$\Longrightarrow\zeta^f_{ik}=0$}}&
\multicolumn{1}{>{\columncolor[gray]{0.8}[-2pt][0pt]}c}{$\zeta_{ik}^f
\cases{\neq0, &\quad
$k\in\mathcal K_i^f,$\vspace*{2pt}\cr=0, &\quad $\mbox{else}$}
$}
& \multicolumn{1}{c@{}}{\multirow{3}{80pt}[8pt]{\centering $\zeta_{ik}=0$ or
$\beta_k-\alpha_i-1>0$
$\Longrightarrow\zeta^f_{ik}=0$}}
\\\\
\hline
\end{tabular*}
\begin{tabular*}{\textwidth}{@{\extracolsep{\fill}}lccc@{}}
\multicolumn{3}{c}{\textbf{Slow time scale} $\bolds{dt}$}\\
\hline\\
&
\multicolumn{1}{c}{\multirow{2}{130pt}[7pt]{\centering \textbf{Reactions of discrete species on slow scale
$\bolds{k\in\mathcal K^s_\circ, \beta_k=0}$}}} &
\multicolumn{1}{c}{\multirow{2}{140pt}[7pt]{\centering \textbf{Reactions of continuous species on slow scale
$\bolds{k\in\mathcal K^s_\bullet,
\beta_k>0}$}}} \\
\hline
\multicolumn{1}{@{}l}{\multirow{2}{90pt}[8pt]{Discrete species, $\alpha_i=0, i = \mathcal I^s_\circ$}}
&
\multicolumn{1}{>{\columncolor[gray]{0.8}[-2pt][0pt]}c}{$\zeta_{ik}^s
\cases{\neq0,
& \quad $k\in\mathcal K_i^s$\vspace*{2pt}\cr=0, &\quad $\mbox{else}$}
$} &
$\zeta_{ik} = \zeta_{ik}^s = 0$
\\[12pt]
\multicolumn{1}{@{}l}{\multirow{2}{90pt}[8pt]{Continuous species, $\alpha_i>0, i = \mathcal I_\bullet^s$}}
& \multicolumn{1}{c}{\multirow{3}{80pt}[8pt]{\centering$\zeta_{ik}=0$ or $\beta_k-\alpha_i<0$
$\Longrightarrow\zeta^s_{ik}=0$}}&
\multicolumn{1}{>{\columncolor[gray]{0.8}[-2pt][0pt]}c}{$\zeta_{ik}^s
\cases{\neq0, &\quad
$k\in\mathcal K_i^s,$\vspace*{2pt}\cr
=0, &\quad $\mbox{else}$}$}
\\\\
\hline
\end{tabular*}
%
\end{table}

\begin{assumption}[(Dynamics of the reaction network)]
For \label{ass:25} Poisson processes $(Y_k)_{k\in\mathcal
K^\ast_\circ}$, the time-change equation
%
\begin{eqnarray}
\label{eq:crnlimit-ss} \underline V(t) &= &\underline V(0) + \sum
_{k\in\mathcal
K_\circ^\ast} \underline\zeta_{\cdot k}^\ast
Y_k \biggl(\int_0^t
\lambda_k^{\mathrm{CR}} \bigl(\underline V(u) \bigr)\,du \biggr)
\nonumber
\\[-8pt]
\\[-8pt]
\nonumber
&&{} +
\sum_{k\in\mathcal
K^*_\bullet} \underline\zeta^*_{\cdot k} \int
_0^t \lambda_{k}^{\mathrm{CR}}
\bigl(\underline V(u) \bigr) \,du
\end{eqnarray}
has a unique solution $\underline V := (\underline V(t))_{t\geq0}$.
\end{assumption}

Actually, the last display is shorthand notation for
\begin{eqnarray*}
V_i(t) & =& V_i(0) + \sum_{k\in\mathcal K_\circ^\ast}
\zeta_{ik}^\ast Y_k \biggl(\int
_0^t \lambda_k^{\mathrm{CR}}
\bigl(\underline V(u) \bigr)\,du \biggr),\qquad i\in\mathcal I_\circ ,
\\
V_i(t) & =& V_i(0) + \sum_{k\in\mathcal K^*_\bullet}
\zeta^*_{ik} \int_0^t
\lambda_{k}^{\mathrm{CR}} \bigl(\underline V(u) \bigr) \,du,\qquad i\in
\mathcal I_\bullet.
\end{eqnarray*}
%

\begin{lemma}[(Convergence of single-scale reaction networks)]\label{lem:crnlimit-ss}
Let $\underline V^N$ be the vector
process of rescaled species amounts for the reaction network which
is the unique solution to \eqref{eq:crn-ss}. Assume
$(\underline\alpha, \underline\beta, \gamma)$ from \eqref{eq:resc0}
satisfy the single scale system assumptions
 \eqref{eq:single-scale}, and Assumptions \ref{ass:22}
and \ref{ass:25} for the rescaled reaction network are
satisfied. Then, if $V^N(0)\Longrightarrow V(0)$, the process of
rescaled amounts $\underline V^N$ converges weakly to the solution
$\underline V$ of \eqref{eq:crnlimit-ss} in the Skorohod topology.
\end{lemma}

 The proof of Lemma~\ref{lem:crnlimit-ss} is an extension of
the classical theorem for convergence
of Markov chains to solutions of ODEs; see \citeauthor{Kurtz1970}
(\citeyear{Kurtz1970,Kurtz1981}), or \citet{EthierKurtz1986}. It is essentially shown in
\citet{KangKurtz2013}, Theorem~4.1. For other recent related results,
see \citet{Zeiser2012}.
\begin{example}[(Self-regulating gene)]\label{ex:self} We give a
simple example of a single-scale reaction network which leads to a
piecewise deterministic solution [similar to
\citet{Zeiser2012}, Section~5]. 
Consider a self-regulating gene modeled by the set of reactions
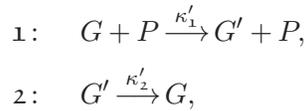
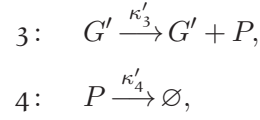
\begin{eqnarray*}
\mathfrak1\dvtx & &\quad G + P \mathop{\xrightarrow}^{\kappa_{\mathfrak1}'}
G'+P,
\\
\mathfrak2\dvtx & & \quad G' \mathop{\xrightarrow}^{\kappa_{\mathfrak2}'}
G,
\\
\mathfrak3\dvtx & &\quad  G' \mathop{\xrightarrow}^{\kappa_{\mathfrak3}'}
G'+P,
\\
\mathfrak4\dvtx & &\quad P \mathop{\xrightarrow}^{\kappa_{\mathfrak4}'}
\varnothing,
\end{eqnarray*}
where $G$ is the inactivated gene, $G'$ is the activated gene (hence
$G,G'$ sums to~$1$ and is conserved by the reactions), and $P$ is
the protein expressed by the gene. Here, $ \mathfrak1$ describes
activation of the gene by the protein, $\mathfrak2$ is spontaneous
deactivation of the gene, $\mathfrak3$ is production of the protein
by the activated gene and $\mathfrak4$ is degradation of the
protein. Let $\underline x = (x_G, x_{G'},
x_P)=(1-x_{G'},x_{G'},x_P)$ and let the reaction rates be
\begin{eqnarray*}
\Lambda_{\mathfrak1}^{\mathrm{CR}}(\underline x) & =& \kappa_{\mathfrak1}'
x_G x_P =\kappa_{\mathfrak1}'
(1-x_{G'}) x_P,\\ \Lambda_{\mathfrak
2}^{\mathrm{CR}}(
\underline x)& =& \kappa_{\mathfrak2}' x_{G'},
\\
\Lambda_{\mathfrak3}^{\mathrm{CR}}(\underline x) & =& \kappa_{\mathfrak3}'
x_{G'}, \qquad\Lambda_{\mathfrak
4}^{\mathrm{CR}}(\underline x) =
\kappa_{\mathfrak4}' x_{P}
\end{eqnarray*}
with scaling $\alpha_G = \alpha_{G'}=0, \alpha_P=1$, that is,
$\mathcal I_\circ= \{G, G'\}$ and $\mathcal I_\bullet= \{P\}$, as
well as
\begin{eqnarray*}
\beta_{\mathfrak1}& =& 0, \qquad\beta_{\mathfrak2} = 0,\qquad
\beta_{\mathfrak3} = 1,\qquad
\beta_{\mathfrak4} = 1,
\\
\kappa_{\mathfrak1}'& =& N^{-1}\kappa_{\mathfrak1},\qquad
\kappa_{\mathfrak2}' = \kappa_{\mathfrak2},\qquad
\kappa_{\mathfrak3}' = N \kappa_{\mathfrak3},\qquad
\kappa_{\mathfrak4}' = \kappa_{\mathfrak4}.
\end{eqnarray*}
Then $v_G = x_G=1-v_{G'},
v_{G'}=x_{G'}, v_P = N^{-1}x_P$, and see \eqref{eq:resc10},
\begin{eqnarray*}
\lambda_{\mathfrak1}^{\mathrm{CR}}(\underline v) & =& \kappa_1
v_G v_P= \kappa_1 (1-v_{G'})
v_P,\qquad \lambda_{\mathfrak2}^{\mathrm{CR}}(\underline v) =
\kappa_2 v_{G'},
\\
\lambda_{\mathfrak3}^{\mathrm{CR}}(\underline v) & = &\kappa_3
v_{G'}, \qquad\lambda_{\mathfrak
4}^{\mathrm{CR}}(\underline v) =
\kappa_4 v_{P}.
\end{eqnarray*}
Here, $\mathcal K_\circ^\ast= \mathcal K_G = \mathcal K_{G'} =
\{\mathfrak1, \mathfrak2\}$ and $\mathcal K_\bullet^\ast=
\mathcal K_P = \{\mathfrak3, \mathfrak4\}$. In this example, the
matrices $\underline{\underline\zeta}$ and
$\underline{\underline\zeta}^\ast$ are given by
\begin{eqnarray*}
\underline{\underline\zeta} = \underline{\underline
\zeta}^\ast = %
\matrix{ G
\cr
G'
\cr
P }
\pmatrix{ -1 & 1 & 0 & 0
\cr
1 & -1 & 0 & 0
\cr
0 & 0 & 1 & -1 }.
\end{eqnarray*}
Moreover, according to Lemma~\ref{lem:crnlimit-ss}, the limit
$(\underline V(t))_{t\geq0}$ of $(V^N(t))_{t\geq0}$ solves
\begin{eqnarray*}
V_{G'}(t) & =& V_{G'}(0) + Y_{\mathfrak1} \biggl(
\kappa_{\mathfrak1} \int_0^t
\bigl(1-V_{G'}(u) \bigr)V_P(u) \,du \biggr) -
Y_{\mathfrak
2} \biggl(\kappa_{\mathfrak2} \int_0^t
V_{G'}(u)\,du \biggr),
\\
V_P(t) & = &V_P(0) + \kappa_{\mathfrak3} \int
_0^t V_{G'}(u)\,du -
\kappa_{\mathfrak4} \int_0^t
V_{P}(u)\,du.
\end{eqnarray*}\vadjust{\goodbreak}
\end{example}

\subsection{Multi-scale systems}
\label{Sec:24}

\subsubsection*{Two-scale systems}
\label{S:two-scale}
We say that the chemical network \eqref{eq:CRN} is a two scale system
if $(\underline\alpha, \underline\beta, \gamma)$ from
\eqref{eq:scalingParms} are such that: there is a partition of
$\mathcal I$ into (disjoint) $\mathcal I^f$, called the
\emph{fast species}, and $\mathcal I^s$, called the
\emph{slow species}, such that, for some $\varepsilon>0$,
%
\begin{eqnarray}
\label{eq:two-scale} %
\max_{k\in\mathcal K_i}
\beta_k +\gamma& = &\alpha_i + \varepsilon, \qquad i\in\mathcal
I^f,
\nonumber
\\[-8pt]
\\[-8pt]
\nonumber
\max_{k\in\mathcal K_i} \beta_k +\gamma& =&
\alpha_i,\qquad i\in\mathcal I^s.
\end{eqnarray}
Without loss of generality, we assume that $\gamma=0$,
and that our choice of $N$ is such that $\varepsilon=1$ in
\eqref{eq:two-scale}, so the relative change of fast species happens
at rate $\mathcal O(N)$ and the relative change of slow species
happens at rate $\mathcal O(1)$.

We first consider what happens on the \textit{faster time scale}
$N  \,dt$. For each $i\in\mathcal I^f$, let $\mathcal
K^f_i\subseteq\mathcal K_i$ be the set of reactions with $\beta_k =
\alpha_i+1$. Define
%
\begin{equation}
\label{eq:Kf} \mathcal K^f= \bigl\{k\in\mathcal K\dvtx \exists i\in
\mathcal I^f\dvtx k\in\mathcal K_i, \beta_k=
\alpha_i+1 \bigr\},
\end{equation}
and a matrix
${\underline{\underline\zeta}}^f$ with $|\mathcal I^f|$ rows and
$|\mathcal K^f|$ columns defined by
%
\begin{equation}
\label{eq:fast-zeta} \zeta^f_{ik}=\lim_{N\to\infty}N^{-(\alpha_i+1)}N^{\beta_k}
\zeta_{ik},\qquad i\in\mathcal I^f, k\in\mathcal
K_i^f.
\end{equation}
This matrix identifies a subnetwork of reactions and their effective
change on the faster time scale $N  \,dt$. Let $\mathcal
I^f_{\circ}\subseteq\mathcal I^f$ be the subset of fast species for
which $\alpha_i=0$, and let $\mathcal
K^f_{\circ}=\bigcup_{i\in\mathcal I^f_{\circ}} \mathcal K^f_i$ be the
subset of reactions changing these fast discrete species on this time
scale. Let $\mathcal I^f_{\bullet}\subseteq\mathcal I^f$ be the subset
of fast species for which $\alpha_i>0$, which are continuous species
on the fast time scale, and let $\mathcal
K^f_{\bullet}=\bigcup_{i\in\mathcal I^f_{\bullet}} \mathcal K^f_i$ be
the subset of reactions changing continuous species on this time
scale. Since $\mathcal I^f_{\circ}$ and $\mathcal I^f_{\bullet}$ are
disjoint in $\mathcal I^f=\mathcal I^f_{\circ}\cup\mathcal
I^f_{\bullet}$, and since $\beta_k$ is unique for each reaction in
$\mathcal K^f=\mathcal K^f_{\circ}\cup\mathcal K^f_{\bullet}$, it
follows that $\mathcal K^f_{\circ}$ and $\mathcal K^f_{\bullet}$ are
disjoint as well.

We next consider what happens on the \textit{slower time scale}
$dt$. For each $i\in\mathcal I^s$, let
%
\begin{equation}
\label{eq:Ks} \mathcal K^s= \bigl\{k\in\mathcal K\dvtx \exists i\in
\mathcal I^s, \beta_k=\alpha_i \bigr\},
\end{equation}
be the set of reactions such that $\beta_k = \alpha_i$, and
${\underline{\underline\zeta}}^s = (\zeta^s_{ik})_{i\in\mathcal I^s,
k\in\mathcal K^s}$ defined by
%
\begin{equation}
\label{eq:slow-zeta} \zeta^s_{ik}=\lim_{N\to\infty}N^{-\alpha_i}N^{\beta_k}
\zeta _{ik}, \qquad i\in\mathcal I^s, k\in\mathcal
K_i^s.
\end{equation}
This matrix identifies the subnetwork of reactions and their effective
change on the slower time scale $dt$. Let $\mathcal
I^s_{\circ}\subseteq\mathcal I^s$ be the subset of (discrete) slow
species for which $\alpha_i=0$, and $\mathcal
K^s_{\circ}=\bigcup_{i\in\mathcal I^s_{\circ}} \mathcal K^s_i$ the
subset of reactions changing discrete species on this time scale. Let
$\mathcal I^s_{\bullet}\subseteq\mathcal I$ be the subset of
(continuous) slow species for which $\alpha_i>0$, and $\mathcal
K^s_{\bullet}=\bigcup_{i\in\mathcal I^s_\bullet} \mathcal K^s_i$ the
subset of reactions changing continuous species on this time scale. As
before, $\mathcal I^s_{\circ}$ and $\mathcal I^s_{\bullet}$ being
disjoint in $\mathcal I^s=\mathcal I^s_{\circ}\cup\mathcal
I^s_{\bullet}$ implies that $\mathcal K^s_{\circ}$ and $\mathcal
K^s_{\bullet}$ are disjoint in $\mathcal K^s=\mathcal
K^s_{\circ}\cup\mathcal K^s_{\bullet}$ as well.

Note, however, that there is no reason for $\mathcal K^s$ to be
disjoint from $\mathcal K^f$. In fact, there may be reactions in
$\mathcal K^s$ with parameter $\beta_k$ that make an effective change
on the time scale $dt$ in a slow species of high enough abundance
$\alpha_i=\beta_k$, that also effectively change some fast species on
the time scale $N  \,dt$, that is, for some $j\in\mathcal I^f$ with
$\beta_k = \alpha_j+1$. The important factor for limiting results is
that we identify contributions from reactions on each of the two
scales independently, and make assumptions on their stability.


Our initial division of species into fast and slow may include some
\textit{conserved quantities}, that is, linear combinations of fast
species that remain unchanged on the faster time scale $N  \,dt$. Let
$\mathcal N(({\underline{\underline\zeta}}^f)^{\textsc t} )$ be the
null space
of $(\underline{\underline\zeta}^f)^{\textsc t}$. If its dimension is
\mbox{$>$0} it
is formed by all the linear combinations of species conserved by the
limiting fast subnetwork, meaning that they see no effective change on
the time scale
$N \,dt$. 
In spatial systems, the fast species are counts of species in a single
compartment (which evolves due to both, movement and chemical
reactions) while conserved quantities are the sum total of the
coordinates in all compartments (which evolves only according to
chemical reactions). For now---unless stated otherwise---we assume
that the basis for the species is such that $\mathrm{dim}(\mathcal
N(({\underline{\underline\zeta}}^f)^{\textsc t} ))=0$.

Define the \textit{fast process} $\underline V^N_f:=(\underline
V^N_f(t))_{t\geq0}$ as $\underline V^N_f(t) := (V^N_i(t))_{i\in
\mathcal
I^f}$ and the \textit{slow process} $\underline V^N_s := (\underline
V^N_s(t))_{t\geq0}$ as $\underline V^N_s(t) := (V^N_i(t))_{i\in
\mathcal
I^s}$. We give necessary assumptions on the dynamics of
$\underline V^N_f$ on the time scale $N  \,dt$ conditional on
$\underline V^N_s(t)\equiv\underline v_s$ being constant, on the
dynamics of $\underline V^N_s$, and on the overall behavior of
$\underline V^N$ in order to obtain a proper limiting dynamics of slow
species, $\underline V_s^N$.

\begin{assumption}[(Dynamics of a two-scale reaction network)] Recall
$\lambda_k^{\mathrm{CR}}$ from \eqref{eq:resc10}. The two-scale
reaction network (\ref{eq:two-scale}) with effective change
$\underline{\underline\zeta}^f$ as in~\eqref{eq:fast-zeta} on time
scale $N  \,dt$ and $\underline{\underline\zeta}^s$ as in
\eqref{eq:slow-zeta} on time scale $dt$ satisfies the following
conditions:
\label{ass:crnerg}
\begin{enumerate}[(iii)]
\item[(i)] For each $\underline v_s\in\mathbb{R}_+^{|\mathcal I^s|}$
there exists a well-defined process $\underline V_{f|\underline
v_s}$, giving the dynamics of the fast species given the vector
of slow species, that is, the solution of
%
\begin{eqnarray}
\label{eq:crnlimit-fast} %
 \underline V_{f|\underline v_s}(t)& =&
\underline V_{f|\underline v_s}(0) + \sum_{k\in\mathcal K^f_\circ}
\underline\zeta^f_{\cdot k} Y_{k} \biggl( \int
_0^t \lambda_{k}^{\mathrm{CR}}
\bigl(\underline V_{f|\underline
v_s}(u),\underline v_s \bigr) \,du \biggr)
\nonumber
\\[-8pt]
\\[-8pt]
\nonumber
&&{} +\sum_{k\in\mathcal K^f_\bullet} \underline\zeta^f_{\cdot k}
\int_0^t \lambda_{k}^{\mathrm{CR}}
\bigl(\underline V_{f|\underline
v_s}(u),\underline v_s \bigr) \,du
\end{eqnarray}
with a unique stationary probability measure $\mu_{\underline
v_s}(d\underline z)$ on $\mathbb{R}_+^{|\mathcal I^f|}$, such
that
%
\begin{equation}
\label{eq:averates} \bar\lambda_{k}^{\mathrm{CR}}(\underline
v_s) =\int_{\mathbb
{R}_+^{|\mathcal I^f|}} \lambda_{k}^{\mathrm{CR}}(
\underline z,\underline v_s) \mu_{\underline v_s}(d\underline z)<
\infty,\qquad k\in\mathcal K^s.
\end{equation}
\item[(ii)] There exists a well-defined process $\underline V_s$
that is the solution of
%
\begin{eqnarray}
\label{eq:crnlimit-ts} %
 \underline V_s(t) & =&
\underline V_s(0) + \sum_{k\in\mathcal
K^s_\circ} \underline
\zeta^s_{\cdot k} Y_{k} \biggl( \int
_0^t \bar\lambda_{k}^{\mathrm{CR}}
\bigl(\underline V_s(u) \bigr) \,du \biggr)
\nonumber
\\[-8pt]
\\[-8pt]
\nonumber
&&{} +\sum_{k\in\mathcal
K^s_\bullet} \underline\zeta^s_{\cdot k}
\int_0^t \bar\lambda_{k}^{\mathrm{CR}}
\bigl(\underline V_s(u) \bigr) \,du
\end{eqnarray}
with $\bar\lambda_k^{\mathrm{CR}}$ given by \eqref{eq:averates}.
\item[(iii)] There exists a locally bounded function $\psi\dvtx\mathbb
R_+^{|\mathcal I|} \to\mathbb R$, $\psi\ge1$ such that
$\psi(x)\to\infty$ as $x\to\infty$, and 
\begin{enumerate}[(iii-a)]
\item[(iii-a)] for each $t>0$
\[
\sup_{N}\mathbb E \biggl[\int_0^t
\psi \bigl(\underline V^N(u) \bigr)\,du \biggr]<\infty;
\]
\item[(iii-b)] for all $k\in\mathcal K$
\[
\lim_{K\to\infty}\sup_{|x|>K}
\frac{\lambda_{k}^{\mathrm
{CR}}(x)}{\psi(x)}=0.
\]
\end{enumerate}
\end{enumerate}
\end{assumption}

\begin{lemma}[(Convergence of two-scale reaction networks)]\label{lem:crnlimit-ts}
Let $\underline V^N$ be the vector
process of rescaled species amounts
for the reaction network which is the unique solution to
\eqref{eq:crn-ss} [or \eqref{eq:crn-ss2}]. Assume that
$(\underline\alpha, \underline\beta, \gamma=0)$ satisfy the two-scale
system assumptions
\eqref{eq:two-scale} for some $\mathcal I^f, \mathcal I^s$ and
$\varepsilon=1$, and the Assumptions \ref{ass:crnerg} are
satisfied. Then, if $\underline V_s^N(0)\mathop{\Longrightarrow}\limits^{N\to\infty}
\underline V_s(0)$, the process of rescaled amounts of the slow species
$\underline V^N_s (\cdot)$ converges weakly to the solution
$\underline V_s(\cdot)$ of \eqref{eq:crnlimit-ts} with rates given
by \eqref{eq:averates} in the Skorokhod topology.
\end{lemma}

The proof of Lemma~\ref{lem:crnlimit-ts} is given in \citet{KangKurtz2013}, Theorem~5.1.

\begin{example}[(Production from a fluctuating source)]
We \label{ex:production} present here an example of a reaction network
with two time scales with no conserved species on the fast time scale.
In our example, two species $A$ and $B$ react and produce species~$C$.
Source $B$ fluctuates as it is quickly transported into the
system and
degrades very fast. We have the set of reactions
%
\begin{eqnarray*}
\mathfrak1\dvtx\quad A + B \mathop{\xrightarrow}^{\kappa_{\mathfrak1}'} C,\qquad
\mathfrak{2}\dvtx\quad \varnothing\mathop{\xrightarrow}^{\kappa_{\mathfrak{2}}'} B,\qquad
\mathfrak3\dvtx\quad B \mathop{\xrightarrow}^{\kappa_{\mathfrak3'}} \varnothing.
\end{eqnarray*}
Here, the sum of the numbers of molecules $A$ and $C$ is constant
(but both will turn out to be slow species), so we only need
to consider the dynamics of the $A$ molecules. We denote molecules
numbers by $x_A$ and $x_B$, respectively, set $\underline x = (x_A,
x_B)$ and consider the reaction rates as given by mass action
kinetics,
%
\begin{eqnarray}
\label{eq:second14} \Lambda_{\mathfrak1}^{\mathrm{CR}}(\underline x)  =
\kappa_{\mathfrak1}' x_A x_B,\qquad
\Lambda_{\mathfrak{2}}^{\mathrm{CR}}(\underline x) = \kappa_{\mathfrak{2}}',\qquad
\Lambda_{\mathfrak
3}^{\mathrm{CR}}(\underline x)  = \kappa_{\mathfrak3}'
x_{B}.
\end{eqnarray}
For the scaled system, we use $\alpha_A = \alpha_C=1$,
$\alpha_B=0$. So, setting the rescaled species counts
$v_A = N^{-1}x_A, v_B = x_B$
and
%
\begin{eqnarray}
\label{eq:second14b} \beta_{\mathfrak1} &=& 0,\qquad \beta_{\mathfrak{2}} = 1,\qquad
\beta_{\mathfrak3} = 1,
\\
\label{eq:second15} \kappa_{\mathfrak1} &=&
\kappa_{\mathfrak1}',\qquad \kappa_{\mathfrak{2}} = N^{-1}
\kappa_{\mathfrak{2}}', \qquad\kappa_{\mathfrak3} = N^{-1}
\kappa_{\mathfrak3}',
\end{eqnarray}
we write
%
\begin{equation}
\label{eq:second15b} \lambda_{\mathfrak1}^{\mathrm{CR}}(\underline v) =
\kappa_1 v_A v_B,\qquad \lambda_{\mathfrak{2}}^{\mathrm{CR}}(
\underline v) = \kappa_2,\qquad \lambda_{\mathfrak3}^{\mathrm{CR}}(
\underline v) = \kappa_3 v_{B}.
\end{equation}
Now, the process $\underline V^N = (V_A^N, V_B^N)$ is given
by \eqref{eq:crn-ss} as
%
\begin{eqnarray}
\label{eq:1gene312b} %
V_A^N(t) & =&
V_A^N(0) - N^{-1} Y_{\mathfrak1} \biggl( N
\int_0^t \kappa_{\mathfrak1}
V^N_A(u) V^N_B(u) \,du \biggr),\nonumber
\\
V_B^N(t) & =& V_B^N(0) -
Y_{\mathfrak1} \biggl( N \int_0^t
\kappa_{\mathfrak1} V^N_A(u) V^N_B(u)
\,du \biggr) + Y_{\mathfrak{2}} ( N \kappa_{\mathfrak2} t )
\\
& &{}- Y_{\mathfrak{3}} \biggl( N \int_0^t
\kappa_{\mathfrak{3}} V^N_{B}(u) \,du \biggr).\nonumber
\end{eqnarray}
From this representation, it should be clear that $V_B$ is fast
while $V_A$ is a slow species. For $\gamma=0$, $\varepsilon=1$, we
have $\mathcal K^s = \{\mathfrak1\}$, $\mathcal K^f = \{\mathfrak
1, \mathfrak{2}, \mathfrak3\}$ (in particular $\mathcal K^s \cap
\mathcal K^f\neq\varnothing$), $\mathcal K_{A} = \{\mathfrak1\}$,
$\mathcal K_B = \{\mathfrak1, \mathfrak{2}, \mathfrak3\}$ and
$\mathcal I_f = \mathcal I_f^\circ= \{B\}$, $\mathcal I_s =\mathcal
I_s^\bullet= \{A\}$. The matrices describing the reaction dynamics
on both scales are
\[
\underline{\underline\zeta} = %
\matrix{ A
\cr
B } %
\pmatrix{ -1 & 0 & 0
\cr
-1 & 1 & -1},\qquad \underline{\underline\zeta}^f =
\matrix{ B} %
\pmatrix{ -1 & 1 & -1},\qquad \underline{\underline
\zeta}^s = %
\matrix{ A } %
\pmatrix{ -1 },
\]
where the three columns in $\underline{\underline\zeta}$ and
$\underline{\underline\zeta}^f$ give reactions $\mathfrak1,
\mathfrak{2}$ and $\mathfrak3$, and $\underline{\underline\zeta}^s$
is a $1\times1$-matrix since there is only one reaction where $A$
is involved. Note that $\mathcal
N((\underline{\underline\zeta}^f)^{t}) = \{0\}$, indicating that
there are no conserved quantities on the fast time scale. In order
to study the dynamics of the slow species, $\underline V_s :=V_A$,
we apply Lemma~\ref{lem:crnlimit-ts} and have to check
Assumption~\ref{ass:crnerg}. Here, for Poisson processes
$Y_{\mathfrak1}, Y_{\mathfrak{2}}$ and $Y_{\mathfrak{3}}$, and
fixed $\underline V_s = V_A= v_A$, from \eqref{eq:crnlimit-fast},
\begin{eqnarray*}
V_{B|v_A}(t) - V_{B|v_A}(0) & =& - Y_{\mathfrak1} \biggl( \int
_0^t \kappa_{\mathfrak1} v_A
V_{B|v_A}(u) \,du \biggr) + Y_{\mathfrak{2}} ( \kappa_2 t )
\\
&&{} - Y_{\mathfrak{3}} \biggl( \int_0^t
\kappa_{\mathfrak{3}} V_{B|v_A}(u) \,du \biggr)
\\
&\stackrel d =& Y_{\mathfrak{2}} ( \kappa_2 t ) -
Y_{\mathfrak1 + \mathfrak3} \biggl( \int_0^t (
\kappa_{\mathfrak1} v_A + \kappa_{\mathfrak{3}})
V_{B|v_A}(u) \,du \biggr)
\end{eqnarray*}
for some Poisson process $Y_{\mathfrak{1}+ \mathfrak3}$
which is independent of $Y_{\mathfrak2}$. Note that
$V_{B|v_A}(\cdot)$ is a birth--death process with constant birth
rate $\kappa_{\mathfrak2}$ and linear death rates, proportional to
$\kappa_{\mathfrak1} v_A + \kappa_{\mathfrak{3}}$. It is
well known that in equilibrium, $V_{B|v_A}\stackrel d = X$ with
\[
X\sim\operatorname{Poi} \biggl(\frac{\kappa_{\mathfrak
2}}{\kappa_{\mathfrak{3}} + \kappa_{\mathfrak1} v_A } \biggr),
\]
which gives the desired $\mu_{v_A}(dv_{B})$. Hence,
\eqref{eq:averates} gives
\[
\bar\lambda^{\mathrm{CR}}_{\mathfrak1}(v_A)  = \mathbf E[
\kappa_{\mathfrak1} v_A X] = \frac{\kappa_{\mathfrak1}
\kappa_{\mathfrak2} v_A}{\kappa_{\mathfrak{3}} +
\kappa_{\mathfrak1} v_A}.
\]
Finally, Lemma~\ref{lem:crnlimit-ts} implies that in the limit
$N\to\infty$, we obtain the dynamics
\[
V_A(t)  = V_A(0) - \int
_0^t \bar\lambda^{\mathrm{CR}}_{\mathfrak
1}
\bigl(\underline V_s(u) \bigr)\,du = V_A(0) - \int
_0^t \frac{\kappa_{\mathfrak1} \kappa_{\mathfrak2}
V_A(u)}{\kappa_{\mathfrak{3}} + \kappa_{\mathfrak1} V_A(u)} \,du.
\]
\end{example}

\subsubsection*{Three time scales}
Chemical reaction networks with more than two time scales also appear
in the literature; see \citet{Weinan2005} for a simulation algorithm
for such systems. One example is the heat shock response in
\textit{Escherichia coli}, introduced by \citet{Srivastava2001} and studied
in detail by \citet{Kang2012}. Here, we state an extension of
Lemma~\ref{lem:crnlimit-ts} to reaction networks with more than two
time scales [see \citet{KangKurtzPopovic2012}]. Namely, suppose that
for some $\gamma\in\mathbb R$ the parameters $\underline\alpha,
\underline\beta$ in \eqref{eq:resc0} and \eqref{eq:resc10} are such
that: there is a partition of $\mathcal I$ into disjoint sets
$\mathcal I^f, \mathcal I^m, \mathcal I^s$ such that, for some
$\varepsilon_2>\varepsilon_1>0$,
%
\begin{eqnarray}
\label{eq:three-scale} %
 \max_{k\in\mathcal K_i}
\beta_k +\gamma& =& \alpha_i + \varepsilon_2,\qquad
i\in\mathcal I^f,\nonumber
\\
\max_{k\in\mathcal K_i}\beta_k +\gamma& =&
\alpha_i + \varepsilon_1, \qquad i\in\mathcal I^m,
\\
\max_{k\in\mathcal K_i} \beta_k +\gamma& =&
\alpha_i,\qquad  i\in\mathcal I^s. \nonumber
\end{eqnarray}
We will assume, as before, that $\gamma=0$, and that our choice of $N$
is such that $\varepsilon_2=1$ in \eqref{eq:three-scale}, so the
relative change of fastest species $\mathcal I^f$ happens at rate
$\mathcal O(N)$, the relative change of the \emph{middle species}
$\mathcal I^m$ happens at rate $\mathcal O(N^{\varepsilon_1})$,
$0<\varepsilon_1<1$, and the relative change of slow species $\mathcal
I^s$ happens at rate $\mathcal O(1)$.

Again, we need to consider what happens on each single time scale
separately. In addition to earlier definitions, for each $i\in\mathcal
I^m$ we let $\mathcal K^m_i\subseteq\mathcal K$ be the set of
reactions with $\beta_k=\alpha_i+\varepsilon_1$, $\mathcal K^m=
\bigcup_{i\in\mathcal I^m}\mathcal K_i^m $, and a matrix
${\underline{\underline\zeta}}^m$ with $|\mathcal I^m|$ rows and
$|\mathcal K^m|$ columns defined by
%
\begin{equation}
\label{eq:middle-zeta} \zeta^m_{ik}=\lim_{N\to\infty}
N^{-(\alpha_i + \varepsilon_1)}N^{\beta_k}\zeta_{ik},\qquad i\in \mathcal
I^m, k\in\mathcal K_i^m,
\end{equation}
which identifies a subnetwork of reactions and their effective change
on the middle time scale $N^{\varepsilon_1}  \,dt$, and we let
$\mathcal I^m_{\circ}\subset\mathcal I^m$ be the subset of middle
species for which $\alpha_i=0$, $\mathcal
I^m_{\bullet}\subseteq\mathcal I^m$ be the subset of fast species for
which $\alpha_i>0$, and finally $\mathcal
K^m_{\circ}=\bigcup_{i\in\mathcal I^m_{\circ}} \mathcal K^m_i$, and
$\mathcal K^m_{\bullet}=\bigcup_{i\in\mathcal I^m_{\bullet}} \mathcal
K^m_i$. We now need an additional set of assumptions on the dynamics
of $\underline V^N_f$ on the time scale $N  \,dt$ conditional on
$(\underline V^N_m(t),\underline V^N_s(t))=(\underline v_m,\underline
v_s)$ being constant, and on the dynamics of $\underline V^N_m$ on the
time scale $N^{\varepsilon_1}  \,dt$ conditional on $\underline
V^N_s(t)=\underline v_s$ being constant.

\begin{assumption}[(Dynamics of a three-scale reaction network)]\label{ass:crnerg3} The
three time scale reaction network (\ref{eq:three-scale}) with
effective change (\ref{eq:fast-zeta}) on time scale $N  \,dt$,
(\ref{eq:middle-zeta}) on time scale $N^{\varepsilon_1}  \,dt$ and
(\ref{eq:slow-zeta}) on time scale $dt$ satisfies the following
conditions:
\begin{longlist}[(i-a)]
\item[(i-a)] For each $(\underline v_m,\underline
v_s)\in\mathbb{R}_+^{|\mathcal I^m|+|\mathcal I^s|}$ there exists
a well-defined process that is the solution of
\begin{eqnarray*}
\underline V_{f|(\underline v_m,\underline v_s)}(t) & = &\underline V_{f|(\underline v_m,\underline v_s)}(0)
+ \sum_{k\in\mathcal K^f_\circ} \underline\zeta^f_{\cdot
k}
Y_{k} \biggl( \int_0^t
\lambda_{k}^{\mathrm{CR}} \bigl(\underline V_{f|(\underline v_m,\underline v_s)}(u),
\underline v_m,\underline v_s \bigr) \,du \biggr)
\\
&&{} +\sum_{k\in\mathcal K^f_\bullet} \underline\zeta^f_{\cdot k}
\int_0^t \lambda_{k}^{\mathrm{CR}}
\bigl(\underline V_{f|(\underline
v_m,\underline v_s)}(u),\underline v_m,\underline
v_s \bigr) \,du
\end{eqnarray*}
with a unique stationary probability measure $\mu_{(\underline
v_m,\underline v_s)}(d\underline z)$ on $\mathbb{R}_+^{|\mathcal
I^f|}$, such that
%
\begin{equation}
\label{eq:averates-middle} \quad\tilde\lambda_{k}^{\mathrm{CR}}(\underline
v_m,\underline v_s) =\int_{\mathbb{R}_+^{|\mathcal I^f|}}
\lambda_{k}^{\mathrm{CR}}(\underline z,\underline v_m,
\underline v_s) \mu_{(\underline v_m,\underline v_s)}(d\underline z)<\infty,\qquad k\in
\mathcal K^m.
\end{equation}
\item[(i-b)] For each $\underline v_s\in\mathbb{R}_+^{|\mathcal
I^s|}$ there exists a well-defined process that is the solution
of
\begin{eqnarray*}
\underline V_{m|\underline v_s}(t) &= &\underline V_{m|\underline
v_s}(0) + \sum
_{k\in\mathcal K^m_\circ} \underline\zeta^m_{\cdot k}
Y_{k} \biggl( \int_0^t \tilde
\lambda_{k}^{\mathrm{CR}} \bigl(\underline V_{m|\underline
v_s}(u),
\underline v_s \bigr) \,du \biggr)
\\
&&{}+\sum_{k\in\mathcal
K^m_\bullet} \underline\zeta^m_{\cdot k}
\int_0^t \tilde\lambda_{k}^{\mathrm{CR}}
\bigl(\underline V_{m|\underline
v_s}(u),\underline v_s \bigr) \,du,
\end{eqnarray*}
which has a unique stationary probability measure $\mu_{\underline
v_s}(d\underline v_m)$ on $\mathbb{R}^{|\mathcal I^m|}$, such
that
%
\begin{equation}
\label{eq:averates-final} \bar\lambda_{k}^{\mathrm{CR}}(\underline
v_s) =\int_{\mathbb
{R}_+^{|\mathcal I^m|}} \tilde\lambda_{k}^{\mathrm{CR}}(
\underline v_m,\underline v_s) \mu_{\underline v_s}(d
\underline v_m)<\infty,\qquad k\in\mathcal K^s.
\end{equation}
\item[(ii)] There exists a well-defined process that is the solution
of (\ref{eq:crnlimit-ts}) with $\bar\lambda_k^{\mathrm{CR}}$ given
by \eqref{eq:averates-final}.
\item[(iii)] see Assumption~\ref{ass:crnerg}(iii).
\end{longlist}
\end{assumption}

The extension of Lemma~\ref{lem:crnlimit-ts} then becomes the following.

\begin{lemma}[(Convergence of three-scale reaction networks)]\label{lem:crnlimit-ts3}
Let $\underline V^N$ be the vector
process of rescaled species amounts
for the reaction network which is the unique solution to
(\ref{eq:crn-ss}) [or \eqref{eq:crn-ss2}]. Assume that
$(\underline\alpha, \underline\beta, \gamma=0)$ satisfy the three
time scale system assumptions
\eqref{eq:three-scale} for some $\mathcal I^f, \mathcal I^m,
\mathcal I^s$ and $0<\varepsilon_1<\varepsilon_2 = 1$, and the
Assumptions \ref{ass:crnerg3} are satisfied. Then, if $\underline
V^N(0)\mathop{\Longrightarrow}\limits^{N\to\infty} \underline V(0)$, the process of
rescaled amounts of the slow species $\underline V^N_s$ converges
weakly to the solution $\underline V_s$ of (\ref{eq:crnlimit-ts})
with rates given by (\ref{eq:averates-final}) in the Skorokhod
topology.
\end{lemma}

 The proof of Lemma~\ref{lem:crnlimit-ts3} is along the same
lines as the proof of Lemma~\ref{lem:crnlimit-ts}, this time applying
the stochastic averaging twice; see \citet{Kang2012} for the same
approach.

\subsubsection*{Conserved quantities}
We turn now to the problem of \textit{conserved quantities}. Suppose we
have a two-scale reaction network with $\operatorname{dim}(\mathcal
N(({\underline{\underline\zeta}}^f)^{\textsc t} ))=:n^f>0$. Then
there exist
linearly independent $\mathbb{R}$-valued vectors $\underline
\theta^{c_i}=(\theta^{c_i}_1,\ldots,\break \theta^{c_i}_{|\mathcal I^f|})$,
$i=1,\ldots,n^f$ such that $t\mapsto\langle\underline\theta^{c_i},
\underline V_{f|\underline v_s}(t)\rangle$ where $\underline
V_{f|\underline v_s}$ from \eqref{eq:crnlimit-fast} is constant. In
other words, the change of $\langle\underline\theta^{c_i},
\underline V^N_f(t)\rangle$ on the time scale $N  \,dt$ goes to 0. We
set $\Theta^f := (\underline\theta^{c_i})_{i=1,\ldots,n^f}$, that is,
%
\begin{equation}
\label{eq:Thetaf} \mathcal N \bigl( \bigl({\underline{\underline
\zeta}}^f \bigr)^{\textsc t} \bigr) = \operatorname {span} \bigl(
\Theta^f \bigr)
\end{equation}
and note that the construction implies that $\underline\theta^{c_i}$
has a unique parameter $\alpha_i$ associated with it, which we denote
by $\alpha_{c_i}$, $i=1,\ldots,|\Theta^f|$ (all the species in the
support of $\underline\theta^{c_i}$ must have the same scaling
parameter $\alpha_{c_i}$, as any species with a greater value of the
scaling parameter does not effectively contribute to the conservation
law in the limit).

We assume that the effective changes for these combinations are on the
time scale $dt$, that is, $\sup_{k\in\mathcal K\dvtx \langle
\underline\theta^{c_i}, \underline\zeta_{\cdot,k}\rangle\neq
0}\beta_k \leq\alpha_{c_i}$. In other words, we exclude the
possibility that they create a new time scale, or that they
effectively remain constant as then we do not need to worry about
their dynamics. This will be all we need for our main results on the
compartment model of multi-scale reaction networks. If they change on
the time scale $dt$, we need to consider their behavior together with
that of the slow species.

We let $\underline V^N_c=(V_{c_i}^N)_{i=1,\ldots,|\Theta^f|}$ be the
vector of rescaled conserved quantities. For each
$i=1,\ldots,|\Theta^f|$, let $\mathcal K_{\underline\theta^{c_i}}$ be
the set of slow reactions such that $\beta_k=\alpha_{c_i}$ and
$\langle\underline\theta^{c_i}, \underline\zeta_{\cdot k}\rangle
\neq
0$, and
let $\mathcal K^c:=\bigcup_{i=1}^{|\Theta^f|}\mathcal
K_{\underline\theta^{c_i}}$. Note that $\mathcal K^c \cap\mathcal K^f
= \varnothing$ by construction. Let $\underline{\underline\zeta}^c$ be
the matrix with $|\Theta^f|$ rows and $|\mathcal K^c|$ columns defined
by
%
\begin{equation}
\label{eq:cons-zeta} \underline{\underline\zeta}^c_{\underline
\theta^{c_i}k}=\lim
_{N\to\infty}N^{-\alpha_{c_i}}N^{\beta
_k} \bigl\langle \underline
\theta^{c_i}, \underline \zeta_{\cdot k} \bigr\rangle,\qquad i =1,
\ldots,| \Theta^f|, k\in\mathcal K_i^c.
\end{equation}
Let $\Theta^f_\circ\subseteq\Theta^f$ be the subset of conserved
quantities for which $\alpha_{c_i}=0$, $\mathcal
K^c_\circ=\bigcup_{\underline\theta^{c_i}\in\Theta^f_\circ
}\mathcal
K_{\underline\theta^{c_i}}$. Let $\Theta^f_\bullet\subseteq\Theta^f$
be the subset of conserved quantities for which $\alpha_{c_i}>0$,
$\mathcal
K^c_\bullet=\bigcup_{\underline\theta^{c_i}\in\Theta^f_\bullet
}\mathcal
K_{\underline\theta^{c_i}}$. As before, $\mathcal K^c_\circ$ and
$\mathcal K^c_\bullet$ are disjoint.

We extend our results, under obvious modifications of our
earlier assumptions below. Note that the dynamics of
conserved quantities depends on that of the fast species in the same
way as the dynamics of the slow species does.

\begin{assumption}[(Dynamics of a two-scale reaction network with
conserved quantities)] The two-scale reaction network
(\ref{eq:two-scale}) with effective change (\ref{eq:fast-zeta}) on
time scale $N  \,dt$ and (\ref{eq:slow-zeta}) and
(\ref{eq:cons-zeta}) on time scale $dt$ satisfies the following
conditions:
\label{ass:crnerg-cons}
\begin{longlist}[(iii)]
\item[(i)] For each $(\underline v_s;\underline
v_c)\in\mathbb{R}_+^{|\mathcal I^s|+|\Theta^f|}$, $\underline v_c
:= (v_{c_i})_{i=1,\ldots,|\Theta^f|}$, there exists a well-defined
process that is the solution of
%
\begin{eqnarray*}
\underline V_{f|(\underline v_s;\underline v_c)}(t)& =& \underline V_{f|(\underline v_s;\underline v_c)}(0)\\
&&{} + \sum
_{k\in\mathcal K^f_\circ} \underline\zeta^f_{\cdot k}
Y_{k} \biggl( \int_0^t
\lambda_{k}^{\mathrm{CR}} \bigl(\underline V_{f|(\underline v_s;\underline v_c)}(u),
\underline v_s \bigr) \,du \biggr)
\\
&&{}+\sum_{k\in\mathcal K^f_\bullet} \underline\zeta^f_{\cdot k}
\int_0^t \lambda_{k}^{\mathrm{CR}}
\bigl(\underline V_{f|(\underline
v_s;\underline v_c)}(u),\underline v_s \bigr) \,du,
\end{eqnarray*}
which satisfies the constraints
%
\begin{equation}
\label{eq:constraints1} \bigl\langle\underline\theta^{c_i}, \underline
V_{f|(\underline
v_s; \underline v_c)} \bigr\rangle= v_{c_i},\qquad \theta^{c_i}\in
\Theta^f,
\end{equation}
and which has a unique stationary probability measure
$\mu_{(\underline v_s;\underline v_c)}(d\underline z)$ on
$\mathbb{R}_+^{|\mathcal I^f|}$ [concentrated on the linear
subspace such that \eqref{eq:constraints1} is satisfied], such
that
%
\begin{equation}
\label{eq:averates-cons} \bar\lambda_{k}^{\mathrm{CR}}(\underline
v_s, \underline v_c) = \int_{\mathbb{R}_+^{|\mathcal I^f|}}
\lambda_{k}^{\mathrm{CR}}( \underline z,\underline v_s,
\underline v_c) \mu_{(\underline v_s;\underline v_c)}(d\underline z)<\infty,\qquad k\in
\mathcal K^s.
\end{equation}
\item[(ii)] In addition to a well-defined solution of
(\ref{eq:crnlimit-ts}), there exists a well-defined process that
is the solution of
%
\begin{eqnarray}
\label{eq:crnlimit-cons} %
\underline V_c(t) & =&
\underline V_c(0) + \sum_{k\in\mathcal
K^c_\circ} \underline
\zeta^c_{\cdot k} Y_{k} \biggl( \int
_0^t \bar\lambda_{k}^{\mathrm{CR}}
\bigl(\underline V_s(u), \underline V_c(u) \bigr) \,du
\biggr)
\nonumber
\\[-8pt]
\\[-8pt]
\nonumber
&&{} +\sum_{k\in\mathcal
K^c_\bullet} \underline\zeta^c_{\cdot k}
\int_0^t \bar\lambda_{k}^{\mathrm{CR}}
\bigl(\underline V_s(u), \underline V_c(u) \bigr) \,du,
\end{eqnarray}
that is,
%
\begin{eqnarray}
\label{eq:crnlimit-cons2} %
V_{c_i}(t) & =&
V_{c_i}(0) + \sum_{k\in\mathcal K^c_\circ}
\zeta^c_{\underline\theta^{c_i} k} Y_{k} \biggl( \int
_0^t \bar\lambda_{k}^{\mathrm{CR}}
\bigl(\underline V_s(u), \underline V_c(u) \bigr) \,du
\biggr)
\nonumber
\\[-8pt]
\\[-8pt]
\nonumber
&&{} + \sum_{k\in\mathcal
K^c_\bullet} \zeta^c_{\underline\theta^{c_i} k}
\int_0^t \bar\lambda_{k}^{\mathrm{CR}}
\bigl(\underline V_s(u), \underline V_c(u) \bigr) \,du,
\end{eqnarray}
where the rates in both \eqref{eq:crnlimit-ts} and
\eqref{eq:crnlimit-cons}, \eqref{eq:crnlimit-cons2} are given by
(\ref{eq:averates-cons}).
\item[(iii)] Same as in Assumptions \ref{ass:crnerg}.
\end{longlist}
\end{assumption}

The following Lemma~\ref{lem:crnlimit-cons} is again in Theorem~5.1 of \citet{KangKurtz2013}.

\begin{lemma}[(Convergence of two-scale reaction networks with conserved fast
quantities)]
Let \label{lem:crnlimit-cons} $\underline V^N$ be the
process of rescaled species amounts \eqref{eq:crn-ss} for a
two-scale reaction network, with $\underline\alpha, \underline
\beta$ satisfying (\ref{eq:two-scale}), $\gamma=0,\varepsilon=1$,
and with conserved quantities $\Theta^f = (\underline
\theta^{c_i})_{i=1,\ldots,|\Theta^f|}$ [which is a basis of the null
space of $((\zeta_{ik})_{i\in\mathcal I^f, k\in\mathcal
K^f})^{\textsc t}$]
whose effective change is on time scale $dt$,
with Assumptions \ref{ass:crnerg-cons} satisfied. Then, if
$\underline V^N(0)\mathop{\Longrightarrow}\limits^{N\to\infty} \underline V(0)$, we have
joint convergence of the process of rescaled amounts of the slow and
conserved quantities $(\underline V^N_s (\cdot), \underline V^N_c
(\cdot))\mathop{\Longrightarrow}\limits^{N\to\infty} (\underline V_s(\cdot),
\underline
V_c(\cdot))$ in the Skorohod topology, with $\underline V_s$ the
solution of (\ref{eq:crnlimit-ts}) and $\underline V_c$ the solution
of (\ref{eq:crnlimit-cons}) with rates given by
(\ref{eq:averates-cons}).
\end{lemma}

It is clear that the result on the limiting dynamics of the conserved
quantities which change on the time scale $dt$ holds even if we do not
have any slow species on this time scale.
We then only have the dynamics of conserved quantities following
(\ref{eq:crnlimit-cons2}) with the rates $\bar\lambda_k^{\mathrm{CR}}$
obtained using the stationary probability measure for the fast species
$\mu_{\underline v_c}(\cdot)$ which depends on the conserved
quantities only. Analogously, it is possible that the dynamics of
conserved species on time scale $dt$ is trivial in which case we have
the dynamics of slow quantities following (\ref{eq:crnlimit-ts}) with
$\underline V_c(u)=\underline v_c, u>0$. Furthermore, if we have a
reaction network on three scales, it is obvious how to write the
analogous result for the conserved quantities on whichever slower time
scale their dynamics is. Both of these situations appear in the
dynamics of compartment reaction network models, and the above lemmas
provide all the tools we need for our results on models with movement
between compartments.

\begin{example}[(Michaelis--Menten kinetics)]
One\label{ex:michmen} of the simplest multi-scale reaction systems
with conserved quantities on the fast time scale leads to the
well-known Michaelis--Menten kinetics. A
{substrate} $S$ is transformed into a {product} $P$ with
the help of an {enzyme} $E$ via a {complex} $ES$ formed by
enzyme and substrate. The set of reactions is
%
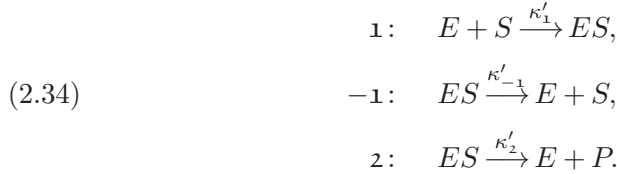
\begin{eqnarray}
\label{eq:michmen1} %
\mathfrak1\dvtx & &\quad E +S
\mathop{\xrightarrow}^{\kappa_{\mathfrak1}'} ES,\nonumber
\\
\mathfrak{-1}\dvtx & &\quad ES  \mathop{\xrightarrow}^{\kappa_{\mathfrak{-1}}'}
E+S,
\\
\mathfrak2\dvtx & &\quad ES  \mathop{\xrightarrow}^{\kappa_{\mathfrak2}'} E+P.\nonumber
\end{eqnarray}
The sum of numbers of free and bound enzymes $E,ES$ is a constant,
which will be denoted $m$ (also the sum of numbers $S, ES, P$
molecules is a constant but $S,P$ will be more abundant and they
will not effectively contribute to a conserved quantity on the fast
time scale). We denote molecules numbers by $x_S, x_E, x_{ES}$ and
$x_P$, let $\underline x = (x_S, x_E, x_{ES}, x_P)$ and let the
reaction rates be given by mass action kinetics,
%
\begin{equation}\quad
\label{eq:michmen2} %
 \Lambda_{\mathfrak1}^{\mathrm{CR}}(
\underline x)  = \kappa_{\mathfrak1}' x_S
x_E, \qquad \Lambda_{\mathfrak{-1}}^{\mathrm{CR}}(\underline x) =
\kappa_{\mathfrak{-1}}' x_{ES}, \qquad \Lambda_{\mathfrak
2}^{\mathrm{CR}}(
\underline x)  = \kappa_{\mathfrak2}' x_{ES}.
\end{equation}
For the scaled system, we use $\alpha_S = \alpha_{P}=1, \alpha_{E} =
\alpha_{ES} =0$. Setting the rescaled species counts $v_S =
N^{-1}x_S, v_E = x_E, v_{ES} = x_{ES}, v_P = N^{-1}x_P$, and
%
\begin{eqnarray}
\beta_{\mathfrak1}& =& 0,\qquad \beta_{\mathfrak{-1}} =
1,\qquad \beta_{\mathfrak2} = 1,\label{eq:gene2}
\nonumber
\\[-8pt]
\\[-8pt]
\nonumber
\kappa_{\mathfrak1}& =& \kappa_{\mathfrak1}',\qquad
\kappa_{\mathfrak{-1}} = N^{-1} \kappa_{\mathfrak{-1}}',\qquad
\kappa_{\mathfrak2} = N^{-1} \kappa_{\mathfrak2}',
\label{eq:gene31}
\end{eqnarray}
we write
\begin{eqnarray*}
\lambda_{\mathfrak1}^{\mathrm{CR}}(\underline v)  = \kappa_1
v_S v_E, \qquad\lambda_{\mathfrak{-1}}^{\mathrm{CR}}(
\underline v) = \kappa_{-1} v_{ES},\qquad \lambda_{\mathfrak
2}^{\mathrm{CR}}(
\underline v)  = \kappa_2 v_{ES}.
\end{eqnarray*}
Note that the rescaling gives that $V^N_S + V^N_P = O(1)$, such that
we only need to describe $V_S^N$. The process $\underline V^N =
(V_S^N, V_E^N, V_{ES}^N)$ is given by
\begin{eqnarray*}
V_S^N(t) & =& V_S^N(0) -
N^{-1} Y_{\mathfrak1} \biggl( N \int_0^t
\kappa_{\mathfrak1} V^N_S(u) V^N_E(u)
\,du \biggr)
\\
&&{} + N^{-1}Y_{\mathfrak{-1}} \biggl( N \int_0^t
\kappa_{\mathfrak{-1}} V^N_{ES}(u) \,du \biggr),
\\
V_E^N(t) & =& V_E^N(0) -
Y_{\mathfrak1} \biggl( N \int_0^t
\kappa_{\mathfrak1} V^N_S(u) V^N_E(u)
\,du \biggr) + Y_{\mathfrak{-1}} \biggl( N \int_0^t
\kappa_{\mathfrak{-1}} V^N_{ES}(u) \,du \biggr)
\\
& &{}+ Y_{\mathfrak{2}} \biggl( N \int_0^t
\kappa_{\mathfrak{2}} V^N_{ES}(u) \,du \biggr),
\\
V_{ES}^N(t) & =& V_{ES}^N(0) +
Y_{\mathfrak1} \biggl( N \int_0^t
\kappa_{\mathfrak1} V^N_S(u) V^N_E(u)
\,du \biggr) - Y_{\mathfrak{-1}} \biggl( N \int_0^t
\kappa_{\mathfrak{-1}} V^N_{ES}(u) \,du \biggr)
\\
& &{}- Y_{\mathfrak{2}} \biggl( N \int_0^t
\kappa_{\mathfrak{2}} V^N_{ES}(u) \,du \biggr).
\end{eqnarray*}
From this, we see $V_E, V_{ES}$ are fast while $V_S, V_P$ are slow
species. For $\gamma=0$, $\varepsilon=1$, we have $\mathcal K^f
=\mathcal K^s = \{\mathfrak1, \mathfrak{-1}, \mathfrak2\}$,
$\mathcal K_{S} = \{\mathfrak1, \mathfrak{-1}\}$, $\mathcal K_E =
\mathcal K_{ES} = \{\mathfrak1, \mathfrak{-1}, \mathfrak2\}$,
$\mathcal K_P = \{\mathfrak2\}$ and $\mathcal I_f = \mathcal
I_f^\circ= \{E, ES\}$, $\mathcal I_s =\mathcal I_s^\bullet= \{S,
P\}$. The matrices describing the reaction dynamics on both scales
are
%
\begin{eqnarray}
\label{eq:MM4} \underline{\underline\zeta} &= &%
\matrix{ S
\cr
E
\cr
ES
\cr
P} %
\pmatrix{ -1 & 1 & 0
\cr
-1 & 1 & 1
\cr
1 & -1 & -1
\cr
0 & 0 &
1 } %
, \qquad\underline{\underline\zeta}^f = %
\matrix{
E
\cr
ES } %
\pmatrix{ -1 & 1 & 1
\cr
1 & -1 & -1} %
,
\nonumber
\\[-8pt]
\\[-8pt]
\nonumber
\underline{\underline\zeta}^s &=& %
\matrix{ S
\cr
P }
\pmatrix{ -1 & 1 & 0
\cr
0 & 0 & 1},
\end{eqnarray}
where the three columns give reactions $\mathfrak1, \mathfrak{-1}$
and $\mathfrak2$. Note that $\mathcal
N((\underline{\underline\zeta}^f)^{t}) = \operatorname{span}((1,1))$,
indicating that $V_{c} := V_E+V_{ES}$ is constant (at least on the
fast time scale). In order to study the dynamics of the slow
species, $\underline V_s :=(V_S, V_P)$, we have to check
Assumption~\ref{ass:crnerg-cons} and apply
Lemma~\ref{lem:crnlimit-cons}. Here, for Poisson processes
$Y_{\mathfrak1}$ and $Y_{\mathfrak{-1}}$, and fixed $\underline V_s
= (V_S, V_P) = (v_S, v_P) = \underline v_s$ and $V_c = V_E+V_{ES} =:
m$, from \eqref{eq:crnlimit-fast},
\begin{eqnarray*}
&&V_{E|(\underline v_s, v_c)}(t) - V_{E|(\underline v_s,
v_c)}(0)
\\
&&\qquad = - Y_{\mathfrak1} \biggl( \int_0^t
\kappa _{\mathfrak1} v_S V_E(u) \,du \biggr) +
Y_{\mathfrak{-1}} \biggl( \int_0^t
\kappa_{\mathfrak{-1}} V_{ES}(u) \,du \biggr)\\
&&\qquad\quad{} + Y_{\mathfrak{2}} \biggl(
\int_0^t \kappa_{\mathfrak{2}}
V_{ES}(u) \,du \biggr)
\\
&&\qquad \stackrel d = - Y_{\mathfrak1} \biggl( \int_0^t
\kappa_{\mathfrak1} v_S V_E(u) \,du \biggr) +
Y_{\mathfrak{-1}+
\mathfrak2} \biggl( \int_0^t (
\kappa_{\mathfrak{-1}} + \kappa_{\mathfrak2})V_{ES}(u) \,du \biggr),
\\
&&V_{ES|(\underline v_s, v_c)}(t) - V_{ES|(\underline v_s,
v_c)}(0)
\\
& &\qquad= Y_{\mathfrak1} \biggl( \int_0^t
\kappa_{\mathfrak1} v_S V_E(u) \,du \biggr) +
Y_{\mathfrak{-1} +\mathfrak2} \biggl( \int_0^t (
\kappa_{\mathfrak{-1}} + \kappa_{\mathfrak2})V_{ES}(u) \,du \biggr)
\end{eqnarray*}
for some Poisson process $Y_{\mathfrak{-1}+ \mathfrak2}$
which is independent of $Y_{\mathfrak1}$. Note that
$(V_{E|(\underline v_s, v_c)}(\cdot), V_{ES|(\underline v_s,
v_c)}(\cdot))$ behaves like an Ehrenfest
urn with two compartments, where each $E$ turns to $ES$ at rate
$\kappa_{\mathfrak1} v_S$, and each $ES$ turns to $E$ at rate
$(\kappa_{\mathfrak{-1}} + \kappa_{\mathfrak2})$. It is well known
that $(V_{E|(\underline v_s, v_c)},
V_{ES|(\underline v_s, v_c)}) \stackrel d = (X, m-X)$ has an equilibrium
\[
X\sim\operatorname{Binom} \biggl(m,\frac{\kappa_{\mathfrak{-1}} +
\kappa_{\mathfrak2}}{\kappa_{\mathfrak{-1}} + \kappa_{\mathfrak
2} + \kappa_{\mathfrak1}v_S} \biggr),
\]
which gives the desired $\mu_{(\underline v_s, m)}(dv_{E}, dv_{ES})$
concentrated on $V_c=V_E + V_{ES}=m$. The
conserved species $V_c$ do not change even on the time scale $dt$
(this will no longer be the case in a heterogeneous compartment
model in the next section). Hence, \eqref{eq:averates-cons} gives
\begin{eqnarray*}
\bar\lambda_{\mathfrak1}(\underline v_s) & =& \mathbf E[
\kappa_{\mathfrak1} v_S X] = \kappa_{\mathfrak1}
v_S \frac{m
(\kappa_{\mathfrak{-1}} + \kappa_{\mathfrak
2})}{\kappa_{\mathfrak{-1}} +
\kappa_{\mathfrak2} + \kappa_{\mathfrak1}v_S},
\\
\bar\lambda_{\mathfrak{-1}}(\underline v_s) & =& \mathbf E \bigl[
\kappa_{\mathfrak{-1}} (m-X) \bigr] = \kappa_{\mathfrak{-1}} \frac{m
\kappa_{\mathfrak1}v_S}{\kappa_{\mathfrak{-1}} +
\kappa_{\mathfrak2} + \kappa_{\mathfrak1}v_S},
\\
\bar\lambda_{\mathfrak{2}}(\underline v_s) & = &\mathbf E \bigl[
\kappa_{\mathfrak{2}} (m-X) \bigr] = \kappa_{\mathfrak{2}} \frac{m
\kappa_{\mathfrak1}v_S}{\kappa_{\mathfrak{-1}} +
\kappa_{\mathfrak2} + \kappa_{\mathfrak1}v_S}.
\end{eqnarray*}
Lemma~\ref{lem:crnlimit-cons} implies that in the limit
$N\to\infty$, we obtain the dynamics
%
\begin{eqnarray} \label{eq:michmenVS}
V_S(t) & =& V_S(0) - \int
_0^t \bar\lambda_{\mathfrak1} \bigl(
\underline V_s(u) \bigr)\,du + \int_0^t
\bar\lambda_{\mathfrak{-1}} \bigl(\underline V_s(u) \bigr)\,du
\nonumber
\\[-8pt]
\\[-8pt]
\nonumber
& =& V_S(0) - \int_0^t
\frac{m\kappa_1 \kappa_2
V_S(u)}{\kappa_{\mathfrak{-1}} + \kappa_{\mathfrak2} +
\kappa_{\mathfrak1}V_S(u)}\,du, 
\end{eqnarray}
for $V_S$, 
which is the classical Michaelis--Menten kinetics.
\end{example}

\section{Chemical reactions in multiple compartments}
\label{S:3}
We now assume that the chemical system is separated into
a set of $\mathcal D$ \textit{compartments}, and chemical species can migrate
within these compartments. For species $i\in\mathcal I$, \textit{movement}
happens from compartment $d'$ to $d''$ at rate
$\Lambda_{i,d',d''}^{\mathrm{M}}$.

\subsection{The Markov chain model}
\label{S:31}
Denoting by $X_{id}(t)$ the number of mol\-ecules of species $i$ in
compartment $d$ at time $t$, we assume that $(\underline{\underline
X}(t))_{t\geq0}$ with $\underline{\underline X}(t) =
(X_{id}(t))_{i\in\mathcal I, d\in\mathcal D}$ is solution of
%
\begin{eqnarray}
\label{eq:crn0} %
X_{id}(t) & =&
X_{id}(0) + \sum_{k\in\mathcal K}
\zeta_{ik} Y_{kd} \biggl(\int_0^t
\Lambda_{kd}^{\mathrm{CR}} \bigl(\underline X_{\cdot
d}(u)
\bigr) \,du \biggr)
\nonumber
\\[-8pt]
\\[-8pt]
\nonumber
& &{}+ \sum_{d',d''\in\mathcal D} \bigl(\delta_{d''}(d) -
\delta_{d'}(d) \bigr) Y_{i,d',d''} \biggl( \int
_0^t \Lambda^{\mathrm{M}}_{i,d',d''}
X_{id'}(u)\,du \biggr),
\end{eqnarray}
where $\delta_{d}(\cdot)$ is a Dirac delta function, $\underline
X_{\cdot d} = (X_{id})_{i\in\mathcal I}$
and all the $Y_\cdot$'s are independent (rate 1) Poisson processes. We
assume the following.

\begin{assumption}[(Dynamics of un-scaled multi-compartment reaction
network)]The reaction network dynamics satisfies the following
conditions: \label{ass:31}
\begin{longlist}[(ii)]
\item[(i)] Same as Assumption~\ref{ass:21}(i) in each compartment
and for all $k$ there is at least one $d$ with
$\Lambda^{\mathrm{CR}}_{kd}\neq0$.
\item[(ii)] Given $(Y_{kd})_{k\in\mathcal K, d\in\mathcal
D}$, and $(Y_{i,d',d''})_{i\in\mathcal I, d',d''\in\mathcal D}$,
the time change equation~\eqref{eq:crn0} has a unique solution.
\end{longlist}
\end{assumption}

 For $X_i(t) := \sum_{d\in\mathcal D}
X_{id}(t)$,
\[
X_i(t)\stackrel d= X_i(0) + \sum
_{k\in\mathcal K} \zeta_{ik} Y_k \biggl(\int
_0^t \sum_{d\in\mathcal D}
\Lambda_{kd}^{\mathrm{CR}} \bigl(\underline X_{\cdot d}(u)
\bigr)\,du \biggr),
\]
%
for some independent (rate 1) Poisson processes $(Y_k)_{k\in\mathcal
K}$. However, since the rate $\sum_{d\in\mathcal
D}\Lambda_{kd}^{\mathrm{CR}}(\underline X_{\cdot d}(s))$ depends on
all entries in $\underline{\underline X}(s)$, the process
$((X_i(t))_{i\in\mathcal I})_{t\geq0}$ is not in general Markov.

\subsection{The rescaled system}
\label{S:32}
Consider the solution of \eqref{eq:crn0} with the chemical reaction
rates $\Lambda^{\mathrm{CR}}_{kd}$ and movement rates
$\Lambda^{\mathrm{M}}_{i,d',d''}$ replaced by
$\Lambda^{\mathrm{CR},N}_{kd}$ and $\Lambda^{\mathrm{M},N}_{i,d',d''}$,
respectively. For real-valued
\[
\underline\alpha= (\alpha_i)_{i\in\mathcal I},\qquad \underline
\beta= (\beta_k)_{k\in\mathcal K},\qquad \gamma,\qquad \underline\eta= (
\eta_i)_{i\in\mathcal I},
\]
with $\alpha_i\geq0, i\in\mathcal I$, we denote the $(\underline
\alpha, \underline\beta, \gamma, \underline\eta)$-rescaled system by
\begin{eqnarray*}
V_{id}^N(t) &:= &N^{-\alpha_i}
X_{id}^N \bigl(N^\gamma t \bigr),\qquad i\in\mathcal I,
d \in\mathcal D,
\\
\lambda^{\mathrm{CR},N}_{k}(\underline v)&:=&N^{-\beta_k}\Lambda
^{\mathrm{CR},N}_{k} \bigl( \bigl(N^{\alpha_i}v_i
\bigr)_{i\in\mathcal I} \bigr),\qquad k\in\mathcal K,
\\
\lambda^{\mathrm{M},N}_{i,d',d''}(\underline v)&:=&N^{-\eta_i}\Lambda
^{\mathrm{M},N}_{i,d',d''} \bigl( \bigl(N^{\alpha_i}v_i
\bigr)_{i\in\mathcal I} \bigr),\qquad i\in\mathcal I, d',d''
\in \mathcal D,
\end{eqnarray*}
where $\underline\alpha, \underline\beta, \gamma, \underline\eta
$ is
chosen so that $V_{id}^N = \mathcal O(1), \lambda^{\mathrm
{CR},N}_{k}=\mathcal O(1), \lambda^{\mathrm{M},N}_{i,d',d''}=\mathcal O(1)$
(reactions of the same type and species of the same type are scaled by
the same parameters in each compartment). Again, we will restrict to
the case $\gamma=0$.

\begin{assumption}[(Dynamics of scaled multiple compartment reaction
network)]
In addition to Assumption~\ref{ass:22} within each compartment,
there \label{ass:22sp} exist $\lambda^{\mathrm{M}}_{i,d',d''}$, $i \in
\mathcal I, d',d'' \in\mathcal D$ with
%
\begin{equation}
\label{eq:resc10sp} %
 N^{-\eta_i}\Lambda^{\mathrm{M},N}_{i,d',d''}
\mathop{\xrightarrow}^{N\to\infty} \lambda^{\mathrm{M}}_{i,d',d''}.
\end{equation}
Again, we will assume that this
convergence is actually an identity.
\end{assumption}

The $(\underline\alpha, \underline\beta, \gamma, \underline
\eta)$-rescaled system $\underline{\underline
V}^N(t)=N^{-\underline\alpha}\underline{\underline X}(N^\gamma t)$
is the unique solution to the system of stochastic equations
%
\begin{eqnarray}
\label{eq:crn1} %
V^N_{id}(t) &
= &V^N_{id}(0) + \sum_{k\in\mathcal K}
N^{-\alpha_i} \zeta_{ik}Y_{kd} \biggl(N^{\beta_k + \gamma}
\int_0^t \lambda_{kd}^{\mathrm{CR}}
\bigl(\underline V^N_{\cdot d}(u) \bigr) \,du \biggr)
\nonumber
\\[-8pt]
\\[-8pt]
\nonumber
&&{} + \sum_{d',d''\in\mathcal D} N^{-\alpha_i} \bigl(
\delta_{d''}(d) - \delta_{d'}(d) \bigr)Y_{i,d',d''}
\biggl( N^{\alpha_i+\eta_i+\gamma} \int_0^t
\lambda^{\mathrm{M}}_{i,d',d''} V^N_{id'}(s) \biggr).
\end{eqnarray}
In addition, define 
\begin{equation}
\label{eq:S} \underline S^N = \bigl(S_i^N
\bigr)_{i\in\mathcal I} \qquad\mbox{with }S_i^N := \sum
_{d\in\mathcal D}V_{id}^N,
\end{equation}
then $S_i^N$ solves
%
\begin{equation}
\label{eq:S2} S^N_{i}(t)  = S^N_{i}(0)
+ \sum_{k\in\mathcal K} N^{-\alpha_i} \zeta_{ik}
\sum_{d\in\mathcal D} Y_{kd} \biggl(N^{\beta_k + \gamma}
\int_0^t \lambda_{kd}^{\mathrm{CR}}
\bigl(\underline V^N_{\cdot d}(u) \bigr) \,du \biggr).
\end{equation}
%

\begin{remark}[(Heterogeneity of the reaction network)] Our set-up does
not preclude
the option that different compartments may have different reaction
networks all together.
We contain all possible reactions in the stoichiometric matrix~$\underline{\underline\zeta}$, then setting individual compartment
rates $\lambda^{\mathrm{CR}}_{kd}$
to zero in desired compartments can achieve this.
\end{remark}

\subsection{Spatial single-scale systems}
\label{Sec:33}

We can now examine the effect of heterogeneity on the chemical
reaction systems via compartmental models. We assume
that \eqref{eq:single-scale} holds within every compartment. The sets
$\mathcal I, \mathcal I_\circ, \mathcal I_\bullet$ and $\mathcal K,
\mathcal K_\circ^\ast, \mathcal K_\bullet^\ast$ and
$\underline{\underline\zeta}^\ast$ are used as in
Section~\ref{S:resultsSingle}. We assume that movement of species is
fast, $\eta_i>0, i\in\mathcal I$ and it has a unique equilibrium.

\begin{assumption}[(Equilibrium for movement)]\label{ass:migerg}
For each species $i\in\mathcal I$, the movement
Markov chain, given through the jump rates
$\lambda_{i,d',d''}^{\mathrm{M}}$ from $d'$ to $d''$, has a unique
stationary probability distribution 
denoted by
$(\pi_i(d))_{d\in\mathcal D}$.
\end{assumption}

\begin{lemma}[(Movement equilibrium)]
Let \label{l:spatial1} Assumption~\ref{ass:migerg} hold.
\begin{longlist}[(1)]
\item[(1)] Let $i\in\mathcal I$ be such that $\alpha_i=0$ (i.e.,
$i\in\mathcal I_\circ$). Consider the Markov chain of only the
movement of
molecules of species $i$, that is, the solution of
\begin{eqnarray*}
V_{id}(t) = V_{id}(0) + \sum_{d',d''\in\mathcal D}
\bigl(\delta_{d''}(d) - \delta_{d'}(d) \bigr)Y_{i,d',d''}
\biggl( \int_0^t \lambda^{\mathrm{M}}_{i,d',d''}
V_{id'}(u)\,du \biggr)
\end{eqnarray*}
started with $\sum_{d\in\mathcal D} V_{id}(0) = s_i$. Then, the
unique equilibrium probability distribution of this Markov chain is
given as
the multinomial distribution with parameters $(s_i;
(\pi_i(d))_{d\in\mathcal D})$.
\item[(2)] Let $i\in\mathcal I$ be such that $\alpha_i>0$ (i.e.,
$i\in\mathcal I_\bullet$). Consider the limiting deterministic
process of only the movement of molecules of species $i$, that is, the
solution of
\[
V_{id}(t) = V_{id}(0) + \sum_{d'\in\mathcal D}
\int_0^t \bigl(\lambda^{\mathrm{M}}_{i,d',d}
V_{id'}(u) - \lambda^{\mathrm{M}}_{i,d,d'}
V_{id}(u) \bigr)\,du
\]
started with $\sum_{d\in\mathcal D} V_{id}(0) = s_i$. Then the
unique equilibrium of this process is given by $(s_i
\pi_i(d))_{d\in\mathcal D}$.
\end{longlist}
We denote the equilibrium probability distribution of movement of all species,
started in $(s_i)_{i\in\mathcal I}$ by $\mathbf P_{\underline s}$
and by $\mathbf E_{\underline s}$ the corresponding expectation
operator. From the above, $\mathbf P_{\underline s}$ is a product of
multinomial and point mass distributions.
\end{lemma}

\begin{pf} In (1), we have an Ehrehfest urn model with
$|\mathcal D|$ boxes; due its reversibility it is easy to check we
have the correct equilibrium. In (2), we have a deterministic
system of $|\mathcal D|$ equations whose equilibrium is equally easy
to obtain.
\end{pf}

We start with the simplest results for chemical reaction networks
which are on a single scale, and describe the effect of mixing on the
heterogeneous chemical reaction system.

\begin{assumption}[(Dynamics of the spatial single-scale reaction
network)]
The \label{ass:crnSpa1} spatial single-scale reaction network on
time scale $dt$, where Assumption~\ref{ass:migerg} holds, satisfies
the following conditions:
\begin{longlist}[(ii)]
\item[(i)] Given $(Y_k)_{k\in\mathcal
K^\ast_\circ}$, the time change equation
%
\begin{eqnarray}
\label{eq:T1a} \underline S(t)  &= &\underline S(0) + \sum
_{k\in\mathcal
K^*_\circ} \underline\zeta^*_{\cdot k} Y_k
\biggl( \int_0^t \bar\lambda_k^{\mathrm{CR}}
\bigl(\underline S(u) \bigr) \,du \biggr)
\nonumber
\\[-8pt]
\\[-8pt]
\nonumber
&&{}+ \sum_{k\in\mathcal K^*_\bullet}
\underline\zeta^*_{\cdot k} \int_0^t
\bar \lambda_k^{\mathrm{CR}} \bigl(\underline S(u) \bigr) \,du
\end{eqnarray}
has a unique solution $\underline S := (\underline S(t))_{t\geq
0}$, where
%
\begin{eqnarray}
\label{eq:T1b} \bar\lambda_k^{\mathrm{CR}}(\underline s) :=
\mathbf E_{\underline
s} \biggl[\sum_{d\in\mathcal
D}
\lambda_{kd}^{\mathrm{CR}}(\underline V_{\cdot d}) \biggr].
\end{eqnarray}
\item[(ii)] Same as (iii) in Assumption~\ref{ass:crnerg} for
all $d\in\mathcal D$.
\end{longlist}
\end{assumption}

\begin{theorem}[(Heterogeneous single-scale system)] \label{T1}
Let $\underline{\underline V}^N$ be the vector
process of
rescaled species amounts for the reaction network which is the
unique solution to \eqref{eq:crn1}. Assume that $(\underline\alpha,
\underline\beta, \gamma=0)$ satisfy single scale assumption \eqref
{eq:single-scale} within compartments and
$\eta_i=\eta>0$, $i\in\mathcal I$. Let $\underline S^N(t) =
(S_i^N(t))_{i\in\mathcal I}$ be $S_i^N(t) := \sum_{d\in\mathcal D}
V_{id}^N(t)$, and suppose Assumptions \ref{ass:migerg}
and \ref{ass:crnSpa1} for the rescaled network hold. If $\underline
S^N(0) \mathop{\Longrightarrow}\limits^{N\to\infty}
\underline S(0)$, then
the process of rescaled sums $\underline S^N(\cdot)$
converges weakly to the unique solution $\underline S(\cdot)$ of
\eqref{eq:T1a} in the Skorohod topology.
\end{theorem}

\begin{pf}
In the heterogeneous reaction network, we have $|\mathcal
I|\times|\mathcal D|$ species; one for each type and each
compartment, with rescaled amounts $V^N_{id}$. Movement between
compartments can be viewed as (at most) $|\mathcal D|\times
|\mathcal D|$ first-order reactions involving only species of the
same type $i\in\mathcal I$ in different compartments, with net
change in compartment $d$ of
$(\delta_{d''}(d)-\delta_{d'}(d))_{(d',d'')\in\mathcal
D\times\mathcal D}$ at rate $\{\Lambda^{\mathrm{M}}_{i,d',d''},
d',d''\in\mathcal D\}$. This set of reactions together with the
original reactions within each compartment give an overall network
in which all the species $V^N_{id}$ with $(i,d)\in\mathcal
I\times\mathcal D$ are fast, whose conserved quantities is a vector
of sums over all the compartments for each species type, which are
given by $S_i^N := \sum_{d\in\mathcal D} V_{id}^N, i\in\mathcal
I$. Since $\eta_i>0, i\in\mathcal I$ the movement reactions change
all the species amounts on the time scale $N^{\eta}  \,dt$, and its
effective changes on this time scale are still
$(\delta_{d''}(d')-\delta_{d'}(d''))_{(d',d'')\in\mathcal
D\times\mathcal D}$ while the original within compartment
reactions effectively change only the conserved sum quantities on
the time scale $dt$ and its effective changes on this time scale are
given by $\underline{\underline\zeta}^*$.

In order to apply Lemma~\ref{lem:crnlimit-cons}, set $\varepsilon:=
\eta$ and we need to check Assumptions~\ref{ass:crnerg-cons}. In
this special case, there are no slow species only fast species and
conserved quantities. Condition (i) is simply the requirement that---in
the limit $N\to\infty$---for fixed given vector of sums of
species movement leads to a well-defined process on the species
amounts in different compartments, which for each value of the
vector of sums $\underline s$ has a unique stationary probability
measure $\mathbf P_{\underline s}$, which is concentrated on
$\sum_{d\in\mathcal D}v_{id}=s_i$. This is exactly implied by
Lemma~\ref{l:spatial1} under
Assumption~\ref{ass:migerg}. Conditions (ii) and (iii) in
Assumptions \ref{ass:crnerg-cons} is assumed in the statement of the
theorem.

Let us consider the dynamics of the conserved quantities. Here,
$\underline\theta^{c_{i}} =  (1_{j=i})_{j\in\mathcal I,
d\in\mathcal D}$ is the $i$th conserved quantity. On the time
scale $dt$, the reaction dynamics of these conserved sums is a Markov
chain whose effective change is given by the matrix
$\underline{\underline\zeta}^c = \underline{\underline\zeta}^*$
with overall rate equal to a sum of the individual compartment
rates.

Since the equilibrium for the movement dynamics $\mathbf
P_{\underline s}$ is given by
%
\begin{eqnarray}
\label{eq:migrlimiteq} %
 \mathbf P_{\underline s}(d
\underline{\underline v}) & =& \prod_{{i\in\mathcal I_\circ}}
\pmatrix{s_i
\cr
{v_{i1}\cdots v_{i|\mathcal D|}}}
\pi_i(1)^{v_{i1}} \cdots\pi_i\bigl(|\mathcal
D|\bigr)^{v_{i |\mathcal D|}}
\nonumber
\\[-8pt]
\\[-8pt]
\nonumber
&&{}\times\prod_{{i\in\mathcal
I_\bullet}}\delta_{\pi_i(1)s_i}(dv_{i1})
\cdots\delta_{\pi_i(\mathcal D)s_i}(dv_{i|\mathcal D|}),
\end{eqnarray}
the averaged rates for reaction dynamics in each compartment under
the equilibrium probability measure as considered in \eqref{eq:T1b}
are exactly
of the form \eqref{eq:averates-cons},
%
\begin{eqnarray}
\label{eq:358} %
\bar\lambda^{\mathrm{CR}}_{k}(\underline
s) & = &\sum_{\underline
v_{1 \cdot}\dvtx \sum_d v_{1d}= s_1} \cdots\sum
_{\underline
v_{|\mathcal I| \cdot}\dvtx \sum_d v_{|\mathcal I|d}= s_{|\mathcal
I|}} \sum_{d\in\mathcal D}
\lambda^{\mathrm{CR}}_{kd}(\underline v_{\cdot d})
\nonumber\\
&&{}\times \prod_{{i\in\mathcal I_\circ}} \pmatrix{s_i
\cr
v_{i1} \cdots v_{i|\mathcal D|}} \pi_i(1)^{v_{i1}}
\cdots\pi_i\bigl(|\mathcal D|\bigr)^{v_{i|\mathcal D|}}
\nonumber
\\[-8pt]
\\[-8pt]
\nonumber
&&{}\times \prod_{{i\in\mathcal
I_\bullet}}\delta_{\pi_i(1)s_i}(v_{i1})
\cdots\delta_{\pi
_i(|\mathcal
D|)s_i}(v_{i|\mathcal D|})
\\
& =& \sum_{d\in\mathcal D} \mathbf E_{\underline s} \bigl[
\lambda_{kd}^{\mathrm{CR}} (\underline V_{\cdot d}) \bigr].\nonumber
\end{eqnarray}
%
\upqed\end{pf}

\begin{corollary}[(Mass-action kinetics)]\label{cor:CT1}
Let $\underline\alpha, \underline\beta, \gamma, \underline\eta$ be
as in Theorem~\ref{T1}. If the reaction rates are given by mass
action kinetics for some $\kappa_{kd}, k\in\mathcal K, d\in\mathcal D$
%
\begin{equation}
\label{eq:360} \lambda_{kd}^{\mathrm{CR}}(\underline
v_{\cdot d}) = \kappa_{kd} \prod_{{i\in\mathcal I_\circ}}
\nu_{ik}!\pmatrix{v_{id}
\cr
\nu_{ik}} \cdot
\prod_{{i\in\mathcal I_\bullet}} v_{id}^{\nu_{ik}},
\end{equation}
then the limit of $\underline S^N(\cdot)$ in the Skorohod topology is the
solution of (\ref{eq:T1a}) with rates given by
\[
\bar\lambda^{\mathrm{CR}}_{k}(\underline s) = \sum
_{d\in\mathcal D} \kappa_{kd}\prod
_{{i\in\mathcal I_\circ}} \nu_{ik}!\pmatrix{s_i
\cr
\nu_{ik}} \pi_i(d)^{\nu_{ik}} \cdot \prod
_{{i\in\mathcal I_\bullet}} \bigl(\pi_i(d)s_i
\bigr)^{\nu_{ik}}.
\]
If $\alpha_i=0$ for all $i\in\mathcal I$, the limit process for the
sums is a Markov chain model for reaction networks with mass
action kinetics (\ref{eq:mak1}) whose rate parameters are
%
\begin{equation}
\label{eq:avemakrates} \bar\kappa_k = \sum_{d\in\mathcal D}
\kappa_{kd} \prod_{i\in
\mathcal
I}
\pi_i(d)^{\nu_{ik}}.
\end{equation}
If $\alpha_i>0$ for all $i\in\mathcal I$, the limit
process for the sums is the deterministic solution to an ordinary
differential equation
\[
d\underline S(t)=\sum_{k\in\mathcal K}\underline
\zeta^*_{\cdot
k}\bar\lambda^{\mathrm{CR}}_{k} \bigl(\underline
S(t) \bigr)\,dt
\]
with mass action kinetics (\ref{eq:mak1}) whose rate parameters are
(\ref{eq:avemakrates}).
\end{corollary}

\begin{remark}[(Different time scales for the movement)] From the point
of view of the limit on time scale $dt$, the parameters for time
scale of movement of different species types do not have to all be
equal $\eta_i=\eta$; as long as $\eta_i>0$ for all $i\in\mathcal I$,
it is easy to show that the limit dynamics of $\underline
S^N(\cdot)$ is as above.
\end{remark}

\begin{pf*}{Proof of Corollary~\ref{cor:CT1}}
We plug \eqref{eq:360} into \eqref{eq:358}. This gives
%
\begin{eqnarray*}
\bar\lambda^{\mathrm{CR}}_{k}(\underline s) & = &\mathop{\sum
_{{\underline
x_{1 \cdot}\dvtx \sum_d x_{1d}= s_1}}}_{\mathrm{if}\ \alpha_1=0} \cdots
\mathop{\sum_{{\underline x_{|\mathcal I| \cdot}\dvtx \sum_d
x_{|\mathcal I|d}= s_{|\mathcal I|}}}}_{{\mathrm{if
}\ \alpha_{|\mathcal I|=0}}} \sum
_{d\in\mathcal D} \kappa_{kd}
\\
&&{}\times \prod_{{i\in\mathcal
I_\circ}}\nu_{ik}!
\pmatrix{x_{id}
\cr
\nu_{ik}}\pmatrix{s_i
\cr
x_{i1} \cdots x_{i|\mathcal D|}} \pi_i(1)^{x_{i1}}
\cdots \pi_i\bigl(|\mathcal D|\bigr)^{x_{i|\mathcal D|}}
\\
&&{}\times \prod_{{i\in\mathcal I_\bullet}} \bigl(\pi_i(d)s_i
\bigr)^{\nu_{ik}}
\\
& =& \sum_{d\in\mathcal D} \kappa_{kd}\mathop{\sum
_{{x_{1d}=0,\ldots,s_1}}}_{\mathrm{if}\ \alpha_1=0}
\cdots\mathop{\sum_{{x_{|\mathcal I|d}= 0,\ldots,s_{|\mathcal
I|}}}}_{\mathrm{if}\ \alpha_{|\mathcal I|=0}}
\\
&&{}\times\prod_{{i\in\mathcal I_\circ}} \nu_{ik}!
\pmatrix{s_i
\cr
x_{id}}\pmatrix{x_{id}
\cr
\nu_{ik}}\pi_i(d)^{x_{id}} \bigl(1-
\pi_i(d) \bigr)^{s_i-x_{id}} \cdot\prod
_{{i\in\mathcal I_\bullet}} \bigl(\pi_i(d)s_i
\bigr)^{\nu_{ik}}
\\
& = &\sum_{d\in\mathcal D} \kappa_{kd}\prod
_{{i\in\mathcal I_\circ}} \nu_{ik}!\pmatrix{s_i
\cr
\nu_{ik}}\pi_i(d)^{\nu_{ik}} \mathop{\sum
_{{x_{1d}=0,\ldots,s_1}}}_{{\mathrm{if}\ \alpha_1=0}} \cdots
\mathop{ \sum_{{x_{|\mathcal I|d}= 0,\ldots,s_{|\mathcal I|}}}}
_{\mathrm{if}\ \alpha_{|\mathcal I|}=0}
\\
&&{}\times \prod_{{i\in\mathcal I_\circ}} \pmatrix{s_i -
\nu_{ik}
\cr
x_{id} - \nu_{ik}}
\pi_i(d)^{x_{id}-\nu_{ik}} \bigl(1-\pi_i(d)
\bigr)^{s_i-x_{id}} \cdot\prod_{{i\in\mathcal I_\bullet}} \bigl(
\pi_i(d)s_i \bigr)^{\nu_{ik}}
\\
& = &\sum_{d\in\mathcal D} \kappa_{kd}\prod
_{{i\in\mathcal
I_\circ}} \nu_{ik}!\pmatrix{s_i
\cr
\nu_{ik}} \pi_i(d)^{\nu_{ik}}\cdot\prod
_{{i\in\mathcal I_\bullet}} \bigl(\pi_i(d)s_i
\bigr)^{\nu_{ik}}.
\end{eqnarray*}
When $\alpha_i=0$ for all $i\in\mathcal I$ only the first sum in
(\ref{eq:T1a}) exists, whereas when $\alpha_i>0$ for all
$i\in\mathcal I$ only the second sum in \eqref{eq:T1a} exists.
\end{pf*}

\begin{example}[(Self-regulating gene in multiple compartments)]
We \label{ex:selfSp} place the reaction kinetics from
Example~\ref{ex:self} in a spatial multi-compartment setting. Let
the dynamics initiate with $S_G(0)+S_{G'}(0)=1$ active and inactive
genes and $S_P(0)=s_P$ proteins in the whole space. If
$\eta_{G},\eta_{G'},\eta_{P}>0$, the movement is faster than any
effective reaction dynamics, and the limiting dynamics of the
rescaled sums in the whole system solves
\begin{eqnarray*}
S_{G'}(t) & =& S_{G'}(0) + Y_{\mathfrak1} \biggl(\bar
\kappa_{\mathfrak1} \int_0^t
\bigl(1-S_{G'}(u) \bigr)S_P(u) \,du \biggr) -
Y_{\mathfrak
2} \biggl(\bar\kappa_{\mathfrak2} \int_0^t
S_{G'}(u)\,du \biggr),
\\
S_P(t) & =& S_P(0) + \bar\kappa_{\mathfrak3} \int
_0^t S_{G'}(u)\,du - \bar
\kappa_{\mathfrak4} \int_0^t
S_{P}(u)\,du,
\end{eqnarray*}
where $\bar\kappa_k$ are given by (\ref{eq:avemakrates}) and
$\pi_G,\pi_{G'},\pi_{P}$ are the equilibrium distributions of the
movement of $G, G'$ and $P$, respectively. Given the values of
system sums $S_{G'}(t),S_P(t)$ the molecules of $G',P$ will then be
distributed in compartments according to
\[
V_{G'd}(t)\sim\operatorname{ Multinom} \bigl(S_{G'}(t),
\pi_{G'}(d) \bigr), \qquad V_{P\,d}(t)\sim\delta_{S_P(t)\pi_P(d)}.
\]
\end{example}

\subsection{Spatial multi-scale systems}
\label{Sec:34}
We next consider heterogeneous reaction networks on multiple time
scales, with interplay between time scales on which the
reaction network dynamics evolves and time scales on which the species
move between compartments. We give results for chemical reactions on
two time scales, extensions to more are obvious.

We stick to our notation from Section~\ref{S:two-scale}. In
particular, we assume the reaction dynamics (within each compartment)
has a separation of time scales (\ref{eq:two-scale}) with
$\varepsilon=1, \gamma=0$. We set $\mathcal K^f$ and $\mathcal K^s$ as
in \eqref{eq:Kf} and \eqref{eq:Ks}, respectively, and $\mathcal I^f$
and $\mathcal I^s$ for the sets of fast and slow species, if only
chemical reactions within compartments are considered. The scaling
parameters for movement of all fast species is $\eta_i=\eta_f$ for
$i\in\mathcal I^f$ while for all slow species is $\eta_i=\eta_s,
i\in\mathcal I^s$. We assume both\vadjust{\goodbreak} $\eta_f,\eta_s>0$.
In order to assess the interplay of dynamics on
different time scales, we need to consider all possible orderings of
$\varepsilon=1, \eta_f$ and $\eta_s$. In the sequel, we assume that
$\eta_f, \eta_s \neq1$ for simplicity. Moreover, the cases
$\eta_s\leq\eta_f<1$ and $\eta_f < \eta_s<1$, as well as
$1<\eta_s\leq\eta_f$ and $1<\eta_f < \eta_s$ lead to the same
limiting behavior, because the movement processes occurring on the
time scale $N^{\eta_f}   \,dt$ and $N^{\eta_s}   \,dt$ are independent
(a movement of one species type on a time scale that is different from
that of reactions depends only on its own molecular counts and is
independent of other species types). Therefore, we are left with the
four cases
%
\begin{eqnarray}
\label{eq:9123} &&({1})\quad 1<\eta_s,\eta_f;\qquad
({2})\quad \eta_s<1<\eta _f;
\nonumber
\\[-8pt]
\\[-8pt]
\nonumber
&& (3)\quad
\eta_f<1<\eta_s;\qquad (4)\quad \eta_f,
\eta_s<1.
\end{eqnarray}
As in the nonspatial situation, we also need to distinguish the cases
when (i) there are no conserved quantities on the time scale of fast
species [meaning that $\mathcal N((\underline{\underline\zeta
}^f)^{\textsc t})$
is the null space], and (ii) when some quantities are conserved [i.e.,
$\mathcal N((\underline{\underline\zeta}^f)^{\textsc t}) =
\operatorname{span}(\Theta^f)$, where $\Theta^f = (\underline
\theta^{c_i})_{i=1,\ldots,n^f}$ is a linearly independent family of
$\mathbb R^{|\mathcal I^f|}$-valued vectors]. In the latter case, the
quantities $\langle\underline\theta^{c_i},
(V_{id}^N(\cdot))_{i\in\mathcal I^f}\rangle$ also change on the time
scale $N^{\eta_f}  \,dt$ for $d\in\mathcal D$ by movement of the fast
species, but $\langle\underline\theta^{c_i},
(S_{i}^N(\cdot))_{i\in\mathcal I^f}\rangle$ is constant on the time
scale $N^{\eta_f}  \,dt$. We start with the case of $\mathcal
N((\underline{\underline\zeta}^f)^{\textsc t}) =$ null space.

\subsubsection*{No conserved quantities on the fast time scale}
We need to consider different processes of possible effective reaction
dynamics for fast species and their sums, conditional on knowing the
values of the slow species. In each of the four cases above we need to
consider different intermediate processes and assumptions on them. We
write here, distinguishing fast and slow species,
$\underline{\underline v} = (\underline{\underline v}_f, \underline
{\underline v}_s)$ with $\underline{\underline v}_f =
(v_{id})_{i\in\mathcal I^f, d\in\mathcal D}$, $\underline{\underline
v}_s = (v_{id})_{i\in\mathcal I^s, d\in\mathcal D}$, as well as
$\underline s = (\underline s_f, \underline s_s)$, $\underline s_f =
(s_i)_{i\in\mathcal I^f}$, $\underline s_s = (s_i)_{i\in\mathcal
I^s}$.

\begin{assumption}[(Dynamics of the spatial multi-scale reaction network)]
In \label{ass:generg} each case (1)--(4), the spatial
two-scale reaction network on time scale $N  \,dt$, where
Assumption~\ref{ass:migerg} holds, satisfies the following
conditions:
\begin{longlist}[(iii)]
\item[(i)](1) Given $(Y_k)_{k\in\mathcal
K^f_\circ}$, the time-change equation of the dynamics of
$\underline S_f$ given the value of $\underline S_s={\underline
s}_{s}$
%
\begin{eqnarray}
\label{eq:crnlimit-1} %
{\underline S}_{f|{\underline s}_{s}}(t)& =& {
\underline S}_{f|{\underline s}_{s}}(0) + \sum_{k\in\mathcal
K^f_\circ}
\underline\zeta^f_{\cdot k} Y_{k} \biggl( \int
_0^t \tilde\lambda_{k}^{\mathrm{CR(1)}}
\bigl({\underline S}_{f|{\underline
s}_s}(u),{\underline s}_s \bigr) \,du
\biggr)
\nonumber
\\[-8pt]
\\[-8pt]
\nonumber
&&{}+\sum_{k\in\mathcal
K^f_\bullet} \underline\zeta^f_{\cdot k}
\int_0^t \tilde\lambda_{k}^{\mathrm{CR(1)}}
\bigl({\underline S}_{f|{\underline
s}_s}(u),{\underline s}_s \bigr) \,du
\end{eqnarray}
has a unique solution, where for all $\underline s_f, \underline
s_s$
%
\begin{eqnarray}
\label{eq:rates-1} \tilde\lambda_{k}^{\mathrm{CR(1)}}(\underline
s_f,\underline s_s)& = &
\mathop{
\int_{\mathbb{R}_+^{|\mathcal I^f|\times
|\mathcal D|}}}_{\times\mathbb{R}_+^{|\mathcal I^s|\times
|\mathcal D|}} \sum_{d\in\mathcal D}
\lambda_{kd}^{\mathrm
{CR}}(\underline{v}_{\cdot d,f},\underline{
v}_{\cdot
d,s})\mathbf P_{(\underline s_f,\underline
s_s)}(d\underline{\underline
v}_f,d\underline{\underline v}_s)
\nonumber
\\[-8pt]
\\[-8pt]
\nonumber
&<&\infty,
\end{eqnarray}
where $\underline{v}_{\cdot d,f}=(v_{id})_{i\in\mathcal I^f},
\underline{v}_{\cdot d,s} = (v_{id})_{i\in\mathcal I^s}$, for
$\mathbf P_{(\underline s_f,\underline s_s)}$ a product of
multinomial and point mass probability distributions\vspace*{-2pt} for both
$\underline{\underline v}_f$ and $\underline{\underline v}_s$
defined in (\ref{eq:migrlimiteq}). In addition, $\underline
S_{f|\underline s_s}(\cdot)$ has a unique stationary probability measure
$\mu_{\underline s_s}(d\underline s_f)$ on
$\mathbb{R}_+^{|\mathcal I^f|}$.\vspace*{1pt}

(2) Given $(Y_k)_{k\in\mathcal
K_\circ^f}$, the time-change equation of the dynamics of
$\underline S_f$ given the value of $\underline{\underline
V}_s={\underline{\underline v}}_{s}$
%
\begin{eqnarray}
\label{eq:crnlimit-2} %
 {\underline S}_{f|\underline{\underline v}_s}(t) &=& {
\underline S}_{f|\underline{\underline v}_s}(0) + \sum_{k\in\mathcal
K^f_\circ}
\underline\zeta^f_{\cdot k} Y_{k} \biggl( \int
_0^t \tilde\lambda_{k}^{\mathrm{CR(2)}}
\bigl({\underline S}_{f|\underline{\underline v}_s}(u),\underline{\underline v}_s
\bigr) \,du \biggr)
\nonumber
\\[-8pt]
\\[-8pt]
\nonumber
&&{}+\sum_{k\in\mathcal K^f_\bullet} \underline\zeta^f_{\cdot k}
\int_0^t \tilde\lambda_{k}^{\mathrm{CR(2)}}
\bigl({\underline S}_{f|\underline{\underline v}_s}(u),\underline{\underline v}_s
\bigr) \,du
\end{eqnarray}
has a unique solution, where for all $\underline s_f,
\underline{\underline v}_s$
%
\begin{eqnarray}
\label{eq:rates-2} \tilde\lambda_{k}^{\mathrm{CR(2)}}(\underline
s_f,\underline{\underline v}_s)=\int_{\mathbb{R}_+^{|\mathcal
I^f|\times|\mathcal D|}}
\sum_{d\in\mathcal
D}\lambda_{kd}^{\mathrm{CR}}(
\underline{ v}_{\cdot d,
f},\underline{v}_{\cdot d,s})\mathbf
P_{\underline
s_f}(d\underline{\underline v}_f)<\infty
\end{eqnarray}
for $\mathbf P_{{\underline s}_f}$ a product of multinomial and
point mass probability distributions as in~(\ref{eq:migrlimiteq}),
where $\mathcal I$ is replaced by $\mathcal I^f$, $\underline s$
by $\underline s_f$ and $\underline{\underline v}$ by
$\underline{\underline v}_f$. In addition, $\underline
S_{f|\underline{\underline v}_s}(\cdot)$ has a unique stationary probability
measure $\mu_{\underline{\underline v}_s}(d\underline s_f)$ on
$\mathbb{R}_+^{|\mathcal I^f|}$.

(3) Given $(Y_{kd})_{k\in\mathcal
K_\circ^f, d\in\mathcal D}$, the time-change equation of the
dynamics of $\underline{\underline V}_f$ given the value of
${\underline S}_s={\underline s}_{s}$
%
\begin{eqnarray}
\label{eq:crnlimit-3} %
\underline V_{\cdot d,f|{\underline s}_{s}}(t) &=&
\underline V_{\cdot
d,f|{\underline
s}_{s}}(0)  + \sum_{k\in\mathcal K^f_\circ}
\underline\zeta^f_{\cdot k} Y_{kd} \biggl( \int
_0^t \tilde\lambda_{kd}^{\mathrm{CR(3)}}
\bigl(\underline V_{\cdot d,f|{\underline s}_s}(u), {\underline s}_s \bigr) \,du
\biggr)
\nonumber
\\[-8pt]
\\[-8pt]
\nonumber
&&{}+\sum_{k\in\mathcal K^f_\bullet} \zeta^f_{ik}
\int_0^t \tilde\lambda_{kd}^{\mathrm{CR(3)}}
\bigl(\underline V_{\cdot d,f|{\underline s}_s}(u), {\underline s}_s \bigr) \,du
\end{eqnarray}
has a unique solution, where for all $\underline{\underline
v}_f,\underline s_s$
%
\begin{equation}
\label{eq:rates-3} \tilde\lambda_{kd}^{\mathrm{CR(3)}}(\underline
v_{\cdot d, f},\underline s_s)=\int_{\mathbb{R}_+^{|\mathcal
I^s|\times|\mathcal D|}}
\lambda_{kd}^{\mathrm
{CR}}(\underline{v}_{\cdot d, f},\underline{
v}_{\cdot d,
s})\mathbf P_{\underline s_s}(d\underline{\underline
v}_s)<\infty
\end{equation}
for $\mathbf P_{{\underline s}_s}$ a product of multinomial and
point mass probability distributions as in~(\ref{eq:migrlimiteq}),
where $\mathcal I$ is replaced by $\mathcal I^s$, $\underline s$
by $\underline s_s$ and $\underline{\underline v}$ by
$\underline{\underline v}_s$. In addition, $\underline{\underline
V}_{f|{\underline s}_s}(\cdot) = (V_{id,f|{\underline
s}_s}(\cdot))_{i\in\mathcal I^f, d\in\mathcal D}$ has a unique
stationary probability measure $\mu_{{\underline s}_s}(d\underline
{\underline
v}_f)$ on $\mathbb{R}_+^{|\mathcal I^f| \times|\mathcal D|}$.

(4) Given $(Y_{kd})_{k\in\mathcal
K_\circ^f, d\in\mathcal D}$, the time-change equation of the
dynamics of $\underline{\underline V}_f$ given the value of
$\underline{\underline V}_s=\underline{\underline v}_{s}$
%
\begin{eqnarray}
\label{eq:crnlimit-4} %
\underline V_{\cdot d,f|\underline{\underline v}_{s}}(t) &=&
\underline V_{\cdot d,f|\underline{\underline v}_{s}}(0) \nonumber\\
&&{} + \sum_{k\in\mathcal K^f_\circ}
\underline\zeta^f_{\cdot k} Y_{kd} \biggl( \int
_0^t \tilde \lambda_{kd}^{\mathrm{CR(4)}}
\bigl(\underline V_{\cdot d,
f|\underline{\underline
v}_s}(u),\underline v_{\cdot d, s} \bigr) \,du
\biggr)
\\
&&{}+\sum_{k\in\mathcal K^f_\bullet}\zeta^f_{ik}
\int_0^t \tilde \lambda_{kd}^{\mathrm{CR(4)}}
\bigl(\underline V_{\cdot d,
f|\underline{\underline v}_s}(u),\underline v_{\cdot d, s} \bigr) \,du\nonumber
\end{eqnarray}
has a unique solution with unique stationary probability measure
$\mu_{\underline{\underline v}_s} (d\underline{\underline
v}_f)$
on $\mathbb{R}_+^{|\mathcal I^f| \times|\mathcal D|}$. Here, we
set
%
\begin{equation}
\label{eq:lambdaCR4} \tilde\lambda_{kd}^{\mathrm{CR(4)}}:=
\lambda_{kd}^{\mathrm{CR}}.
\end{equation}
\item[(ii)] There exists a well-defined process $\underline
S_s(\cdot)$ that is the unique solution of
%
\begin{eqnarray}
\label{eq:T21} %
 \underline S_s(t) &=&
\underline S_s(0)  + \sum_{k\in\mathcal
K^s_\circ}
\underline\zeta^s_{\cdot k} Y_k \biggl( \int
_0^t \bar\lambda_k^{\mathrm{CR}(\ell)}
\bigl(\underline S_s(u) \bigr) \,du \biggr)
\nonumber
\\[-8pt]
\\[-8pt]
\nonumber
&&{} + \sum_{k\in\mathcal
K^s_\bullet} \underline\zeta^s_{\cdot k}
\int_0^t \bar\lambda_k^{\mathrm{CR}(\ell)}
\bigl(\underline S_s(u) \bigr) \,du,
\end{eqnarray}
where rates $(\bar\lambda_k^{\mathrm{CR}(\ell)})_{\ell=1,2,3,4}$ are
given from $(\tilde\lambda_k^{\mathrm{CR}(\ell)})_{\ell=1,2,3,4}$
in each case as 
\begin{eqnarray}\qquad
\label{eq:rates-final1} \bar\lambda_{k}^{\mathrm{CR(1)}}(\underline
s_s) &=&\int_{\mathbb{R}_+^{|\mathcal I^f|}} \tilde\lambda_{k}^{\mathrm{CR(1)}}(
\underline s_f,\underline s_s) \mu_{\underline s_s}(d
\underline s_f)
\nonumber
\\[-8pt]
\\[-8pt]
\nonumber
&=& \sum_{d\in\mathcal D}\iint
\lambda_{kd}^{\mathrm
{CR}}( \underline{v}_{\cdot d,f},\underline{
v}_{\cdot d,s}) \mathbf P_{(\underline s_f,\underline
s_s)}(d\underline{\underline
v}_f,d\underline{\underline v}_s) \mu_{\underline s_s}(d
\underline s_f) <\infty;
\\
\label{eq:rates-final2} \bar \lambda_{k}^{\mathrm{CR(2)}}(\underline
s_s) &=&\int_{\mathbb{R}_+^{|\mathcal I^s|\times|\mathcal
D|}}\int_{\mathbb{R}_+^{|\mathcal I^f|}}
\tilde\lambda_{k}^{\mathrm{CR(2)}}(\underline s_f,
\underline{\underline v}_s) \mu_{\underline{\underline
v}_s}(d\underline
s_f) \mathbf P_{\underline
s_s}(d\underline{\underline
v}_s)
\nonumber
\\[-8pt]
\\[-8pt]
\nonumber
&=& \sum_{d\in\mathcal D}\iiint
\lambda_{kd}^{\mathrm
{CR}}( \underline{ v}_{\cdot d, f},
\underline{v}_{\cdot d,s}) \mathbf P_{\underline s_f}(d\underline{\underline
v}_f) \mu_{\underline{\underline v}_s}(d \underline s_f) \mathbf
P_{\underline s_s}(d\underline{\underline v}_s) <\infty;
\\
\label{eq:rates-final3} \bar \lambda_{k}^{\mathrm{CR(3)}}(\underline
s_s) &=&\sum_{d\in\mathcal
D}\int
_{\mathbb{R}_+^{|\mathcal I^f|\times|\mathcal
D|}} \tilde\lambda_{kd}^{\mathrm{CR(3)}}(
\underline{\underline v}_f,\underline s_s)
\mu_{\underline
s_s}(d\underline{\underline v}_f)
\nonumber
\\[-8pt]
\\[-8pt]
\nonumber
&=& \sum_{d\in\mathcal D} \iint
\lambda_{kd}^{\mathrm
{CR}}( \underline{v}_{\cdot d, f},\underline{
v}_{\cdot d,
s})\mathbf P_{\underline s_s}(d\underline{\underline
v}_s)\mu_{\underline s_s}(d \underline{\underline v}_f) <
\infty;
\\
\label{eq:rates-final4} \bar \lambda_{k}^{\mathrm{CR(4)}}(\underline
s_s) &=&\sum_{d\in\mathcal
D}\int
_{\mathbb{R}_+^{|\mathcal I^s|\times|\mathcal
D|}}\int_{\mathbb{R}_+^{|\mathcal I^f|\times
|\mathcal D|}} \tilde
\lambda_{k}^{\mathrm{CR(4)}}(\underline{\underline v}_f,
\underline{\underline v}_s) \mu_{\underline{\underline
v}_s}(d\underline{\underline
v}_f) \mathbf P_{\underline
s_s}(d\underline{\underline
v}_s)
\nonumber
\\[-8pt]
\\[-8pt]
\nonumber
&=& \sum_{d\in\mathcal D}\iint
\lambda_{kd}^{\mathrm{CR}}( \underline v_{\cdot d,f}, \underline
v_{\cdot d,s}) \mu_{\underline{\underline
v}_s}(d \underline{\underline v}_f)
\mathbf P_{\underline
s_s}(d\underline{ \underline v}_s) <\infty.
\end{eqnarray}
%
%
\item[(iii)] Same as (iii) in Assumption~\ref{ass:crnerg} in each
compartment.
\end{longlist}
\end{assumption}

\begin{remark}[(Equivalent formulation)]
For the dynamics under the above assumption, the following is
immediate: In \label{rem:crnSpa2} each case (1)--(4), the
spatial two-scale reaction network on time scale $dt$, where
Assumption~\ref{ass:generg} holds, satisfies the following
condition: given $(Y_k)_{k\in\mathcal K^s_\circ}$,
the time-change equation \eqref{eq:T21} has a unique solution, where
for all $\underline s_s$
%
\begin{eqnarray}
\label{eq:T22b} \bar\lambda_k^{\mathrm{CR}(\ell)} ( \underline
s_s )  := \mathbf E_{\underline s_s} \biggl[\sum
_{d\in\mathcal D} \lambda_{kd}^{\mathrm{CR}}(\underline
V_{ \cdot d}) \biggr] <\infty
\end{eqnarray}
and the distribution of $(V_{id})_{i\in\mathcal I, d\in\mathcal D}$ in
\eqref{eq:T22b} \,depends on the parameters $\eta_s, \eta_f$ as
follows:
%
\begin{eqnarray}
\label{eq:mea-1}(\ell)=({1}) &&\quad \mathbf P_{\underline
s_s}(d
\underline{\underline v}_f, d\underline{\underline v}_s) =
\int_{\mathbb R_+^{|\mathcal I^f|}}\mu_{\underline s_s} (d\underline s_f)
\mathbf P_{(\underline s_f, \underline s_s)} (d\underline{\underline v}_f, d\underline{
\underline v}_s),
\\
\label{eq:mea-2} (\ell)=({2}) &&\quad \mathbf P_{\underline s_s}(d
\underline{\underline v}_f, d\underline{\underline v}_s) =
\mathbf P_{\underline
s_s}(d\underline{\underline v}_s) \int
_{\mathbb R_+^{|\mathcal
I^f|}}\mu_{\underline{\underline v}_s}(d \underline s_f)
\mathbf P_{\underline s_f}(d\underline{\underline v}_f),
\\
\label{eq:mea-3} (\ell)=({3}) &&\quad  \mathbf
P_{\underline s_s}(d\underline{\underline v}_f, d\underline{\underline
v}_s) = \mathbf P_{\underline
s_s}(d\underline{\underline
v}_s)\mu_{\underline s_s} (d\underline{\underline v}_f),
\\
\label{eq:mea-4} (\ell)=({4}) & &\quad \mathbf
P_{\underline s_s}(d\underline{\underline v}_f, d\underline{\underline
v}_s) = \mathbf P_{\underline
s_s}(d\underline{\underline
v}_s)\mu_{\underline{\underline
v}_s} (d\underline{\underline v}_f).
\end{eqnarray}
\end{remark}

We can now state our results for the limiting behavior of $\underline
S^N_f:=(S^N_i, i\in\mathcal I^f)$ and $\underline S^N_s:=(S^N_i, i\in
\mathcal I^s)$ on the time scales $N  \,dt$ and $dt$.

\begin{theorem}[(Two-scale system without conserved fast
quantities)] \label{thm:T2} Let $\underline{\underline V}^N$ be
the vector process
of rescaled species amounts for the reaction network which is the
unique solution to \eqref{eq:crn1}. Assume that $(\underline\alpha,
\underline\beta, \gamma=0)$ satisfy two-scale system
assumption \eqref
{eq:two-scale} for some
$\mathcal I^f, \mathcal I^s$ with $\varepsilon=1$ and $\mathcal
N((\underline{\underline\zeta}^f)^{\textsc t})=0$ [with
$\underline{\underline\zeta}^f$ from \eqref{eq:fast-zeta}] within
compartments without
conserved quantities on the fast time scale. In addition,
$\eta_i=\eta_f>0$ for all $i\in\mathcal I^f$ and $\eta_i=\eta_s>0$
for all $i\in\mathcal I^s$, one of the cases (1)--(4)
holds, and Assumption~\ref{ass:generg} holds. Then, if $\underline
S^N(0) \mathop{\Longrightarrow}\limits^{N\to\infty}
\underline S(0)$, the rescaled
sums of slow species $\underline S^N_s(\cdot)$ from \eqref{eq:S}
converges weakly to the unique solution $\underline S_s(\cdot)$
of \eqref{eq:T21} in the Skorohod topology.
\end{theorem}

\begin{remark}[(Interpretation)]
The rates in Theorem~\ref{thm:T2} have an intuitive interpretation.
In order to compute $\mathbf E_{\underline s} [
\lambda_{kd}^{\mathrm{CR}} (\underline V_{\cdot d}) ]$, we have to
know the distribution of $\underline{\underline V}$ given
$\underline S_s$. Consider case (1) as an example: since
movement of particles are the fastest reactions in the system, given
the value of $\underline S_s=\underline s_s$, (i)
$\underline{\underline V}_s$ are distributed according to $\mathbf
P_{\underline s_s}(d\underline{\underline v}_s)$ from
\eqref{eq:migrlimiteq}, and (ii) $\underline S_f$ is distributed
according to the probability measure $\mu_{\underline
s_s}(d\underline s_f)$ from Assumption~\ref{ass:generg}(i)(1); then, given the value of $\underline S_f=\underline s_f$, the
values of $\underline{\underline V}_f$ are distributed according to
$\mathbf P_{\underline s_f}(d\underline{\underline v}_f)$ from
\eqref{eq:migrlimiteq}. In case (3), fast reactions within
compartments intertwine with movement between them: given the value
of $\underline S_s=\underline s_s$, (i) $\underline{\underline V}_s$
are again distributed according to $\mathbf P_{\underline s_s}$ from
\eqref{eq:migrlimiteq}, but (ii) $\underline{\underline V}_f$ are
distributed according to $\mu_{{\underline
s_s}}(d\underline{\underline v}_f)$.
\end{remark}

\begin{pf*}{Proof of Theorem~\ref{thm:T2}}
The proof relies on use
of Lemmas \ref{lem:crnlimit-ts3} and \ref{lem:crnlimit-cons}. Let us
first consider the case (1): $\eta_f,\eta_s>1$. In this case
on the two fastest time scales $N^{\eta_f}   \,dt$ and $N^{\eta_s}
\,dt$, we have movement of fast and slow species, respectively, whose
sums are unchanged on any time scale faster than $N  \,dt$. We view
these two fastest time scales as one, since dynamics of movement of
fast and slow species are in this case independent of each other and
can be combined. Regarding all of the movement as a set of
first-order reactions as in proof of Theorem~\ref{T1}, we have a
three time scale dynamics: movement of all species is the fast
process on time scales $N^{\eta_f}   \,dt,N^{\eta_s}  \,dt$, effective
change of fast species is the medium process on time scale $N  \,dt$,
and effective change of slow species is the slow process on the time
scale $dt$. The fast process of movement of all species has a
stationary probability measure that is a product of multinomial and
point mass probability distributions $\mathbf P_{(\underline s_f,
\underline s_s)}$ from (\ref{eq:migrlimiteq}). Arguments from
Lemma~\ref{lem:crnlimit-cons} imply that on the time scale $N  \,dt$
all rates for the reaction network dynamics
$\tilde\lambda_k^{\mathrm{CR(1)}}$ are sums over compartments of rates
averaged with respect to $\mathbf P_{(\underline s_f,\underline
s_s)}$ as in (\ref{eq:rates-1}), and the medium process of the
sums of fast species $\underline S_f$ has effective change given by
$\underline{\underline\zeta}^f$. Condition (i)(1) of
Assumption~\ref{ass:generg} ensures that on time scale $N  \,dt$ the
medium process $\underline S_f(\cdot)$ is well defined and has a
unique stationary probability distribution $\mu_{\underline
s_s}(d\underline s_f)$. Condition (ii) of
Assumption~\ref{ass:generg} then ensures that, in addition to
conditions (i-a) and (i-b), also condition (ii) in the assumptions
for Lemma~\ref{lem:crnlimit-ts3} is met, and consequently the
limiting dynamics of the slow process $\underline S_s(\cdot)$ with
effective change given by $\underline{\underline\zeta}^f$ is
well defined and given by the solution of (\ref{eq:T21}) with rates
as in (\ref{eq:rates-final1}).

Next, consider the case (2): $\eta_f>1, \eta_s<1$. In this case
we have a four time scale dynamics: movement of fast species is the
fast process on time scale $N^{\eta_f}  \,dt$, effective change of
fast species is the medium-fast process on time scale $N  \,dt$,
movement of slow species is the medium-slow process on time scale
$N^{\eta_s}  \,dt$, and finally effective change of slow species is
the slow process on time scale $dt$. The fast process on time scale
$N^{\eta_f}  \,dt$ of movement of all species has a stationary
probability measure that is $\mathbf P_{\underline s_f}$
(\ref{eq:migrlimiteq}) (over $i\in\mathcal I^f$
only). Lemma~\ref{lem:crnlimit-cons} implies that on the next time
scale $N  \,dt$ rates $\tilde\lambda_k^{\mathrm{CR(2)}}$ are averaged
with respect to $\mathbf P_{\underline s_f}$ as in~(\ref{eq:rates-2}),
and the medium-fast process $\underline S_f$ has
an effective change given by $\underline{\underline\zeta}^f$. We
now have that this process is well defined and has a unique
stationary probability distribution $\mu_{\underline{\underline
v}_s}(d\underline s_f)$. Furthermore, on the next time scale
$N^{\eta_s}  \,dt$ we only have the movement of slow species, which
has a stationary probability measure that is $\mathbf P_{\underline
s_s}$ from (\ref{eq:migrlimiteq}) (over $i\in\mathcal I^s$
only). Finally, the limiting dynamics of the slow process $\underline
S_s(\cdot)$ on time scale $dt$, by an extension of
Lemma~\ref{lem:crnlimit-ts3} to four time scales, is well defined
and given by the solution of (\ref{eq:T21}) with rates as in
\eqref{eq:rates-final2}.

Let us next consider the case (3): $\eta_f<1,\eta_s>1$. We
again have a four time scale dynamics: movement of slow species is
the fast process on time scale $N^{\eta_s}  \,dt$, effective change
of fast species is the medium-fast process on time scale $N  \,dt$,
movement of fast species is the medium-slow process on time scale
$N^{\eta_f}  \,dt$, and finally effective change of slow species is
the slow process on time scale $dt$. Lemma~\ref{lem:crnlimit-cons}
implies that on the medium-fast time scale $N  \,dt$ rates
$\tilde\lambda_{kd}^{\mathrm{CR(3)}}$ are averaged with respect to
$\mathbf P_{\underline s_s}$ (\ref{eq:migrlimiteq}) (over
$i\in\mathcal I^{s}$ only) as in (\ref{eq:rates-3}), and the
medium-fast process $\underline{\underline V}_f$ has an effective
change given by $\underline{\underline\zeta}^f$ in each compartment
$d\in\mathcal D$. This process is well defined and has a unique
stationary probability distribution $\mu_{\underline
s_s}(d\underline{\underline v}_f)$. On the next time scale
$N^{\eta_s}  \,dt$ we only have the movement of slow species, which
has a stationary probability measure that is $\mathbf P_{\underline
s_f}$ (\ref{eq:migrlimiteq}) (over $i\in\mathcal I^f$
only). Finally, on time scale $dt$, Lemma~\ref{lem:crnlimit-ts3}
extended to four time-scales implies the limiting dynamics of the
slow process $\underline S_s(\cdot)$ is well defined and given by
the solution of (\ref{eq:T21}) with rates as in
(\ref{eq:rates-final3}).

Finally, we consider case (4): $\eta_f<1,\eta_s<1$. On the
fast time scale $N  \,dt$ in each compartment $d\in\mathcal D$,
independently we have reaction dynamics of the fast species, with a
unique equilibrium $\mu_{\underline{\underline
v}_s}(\underline{\underline v}_f)$ that must be a product
distribution over the different compartments. Similarly, to case
(1) when all the movement is fastest, now all the
movement is on the medium time scale, and the movement of
all molecules (fast and slow) is independent and can be viewed as
combined on one time scale with unique stationary
probability distribution $\mathbf P_{(\underline s_f, \underline
s_s)}$ (\ref{eq:migrlimiteq}) (over all $i\in\mathcal I$). This
implies that on the slow time scale $dt$, the effective change of
$\underline S_s$ is due to reaction dynamics with rates
$\tilde\lambda_{kd}^{\mathrm{CR,4}}(\underline{\underline
v}^N_f,\underline{\underline v}^N_s)$ that have been averaged over
$\mathbf P_{(\underline s_f, \underline s_s)}$, and is given by
$\underline{\underline\zeta}^s$. Lemma~\ref{lem:crnlimit-ts3}
implies that $\underline S_s(\cdot)$ is well defined and given by
the solution of (\ref{eq:T21}) with rates as in
(\ref{eq:rates-final4}).
\end{pf*}

If the movement of fast species is slower than fast reactions [i.e.,
we consider cases (3) or (4)], the equilibria for reactions
is always attained before movement of fast species can change this
equilibrium, as stated in the next corollary.

\begin{corollary}[(Irrelevance of movement of fast species)]
In \label{cor:indepSf} cases (3) or~(4) of Theorem~\ref{thm:T2},
the limiting dynamics of $\underline
S_s$ is independent of $\mathbf P_{\underline s_f}$.
\end{corollary}

\begin{pf}
The assertion can be seen directly from \eqref{eq:mea-3}
and \eqref{eq:mea-4}, since the right-hand sides do not depend on
$\mathbf P_{\underline s_f}$.
\end{pf}

If all slow species are continuous, the limiting dynamics for cases
(1), (2) and (3), (4) are equal. The key to
this observation is the following lemma.

\begin{lemma}\label{lem:equivofeq}
In the situation of Theorem~\ref{thm:T2}, assume all slow species are
continuous,
$\mathcal I^s_\circ= \varnothing$, and let
${\underline{\underline{\pi s}}}_s := (\pi_i(d)s_i)_{i\in\mathcal
I^s, d\in\mathcal D}$.
\begin{longlist}[(ii)]
\item[(i)] For stationary probability measures $\mu_{\underline
s_s} (d\underline s_f)$ of $\underline
S_{f|\underline s_s}$ from
Assumption~\ref{ass:generg}$\mathrm{(i)({1})}$ and
$\mu_{\underline{\underline v}_s}(d\underline s_f)$ of $\underline
S_{f|\underline{\underline v}_s}$ from $\mathrm{(i)({2})}$, we
have
\[
\mu_{\underline s_s}(d\underline s_f) =
\mu_{\underline{\underline
{\pi
s}}_s}(d\underline s_f).
\]
\item[(ii)] Likewise, for stationary probability measures
$\mu_{{\underline s}_{s}}(d\underline{\underline v}_f)$ of
$\underline{\underline V}_{f|{\underline s}_{s}}$ from
$\mathrm{(i)({3})}$ and $\mu_{\underline{\underline
v}_{s}}(d\underline{\underline v}_f)$ of
$\underline{\underline V}_{f|\underline{\underline v}_{s}}$ from
$\mathrm{(i)({4})}$, we have
\[
\mu_{{\underline s}_{s}}(d\underline{\underline v}_f) =
\mu_{\underline{\underline{\pi s}}_{s}}(d\underline{\underline v}_f).
\]
\end{longlist}
\end{lemma}

\begin{pf}
(i) It suffices to show that $\mu_{\underline{\underline{\pi s}}_s}$
is a stationary probability measure for the process $\underline
S_{f|\underline s_s}(\cdot)$ from \eqref{eq:crnlimit-1}, since we
assumed that this process has a unique stationary probability
distribution.
Note that, by independence of the movement of fast and slow species,
for $k\in\mathcal K^f$,
%
\begin{eqnarray*}
\tilde\lambda_{k}^{\mathrm{CR(1)}}(\underline s_f,
\underline s_s) & =& \int\sum_{d\in\mathcal D}
\lambda_{kd}^{\mathrm{CR}}(\underline{v}_{\cdot d,f},
\underline{v}_{\cdot d,s}) \mathbf P_{(\underline s_f,
\underline s_s)}(d\underline{\underline
v}_f, d\underline{\underline v}_s)
\\
& =& \int\sum_{d\in\mathcal D} \lambda_{kd}^{\mathrm{CR}}(
\underline{v}_{\cdot d,f}, \underline{v}_{\cdot d,s}) \mathbf
P_{\underline
s_f}(d\underline{\underline v}_f) \mathbf
P_{\underline
s_s}(d\underline{\underline v}_s)
\\
& =& \int\tilde\lambda_{k}^{\mathrm{CR(2)}}(\underline s_f,
\underline{\underline v}_s)\mathbf P_{\underline
s_s}(d\underline{
\underline v}_s)
\\
& =& \tilde\lambda_{k}^{\mathrm{CR(2)}}(\underline s_f,
\underline{\underline{\pi s}}_s).
\end{eqnarray*}
Since these rates are equal, the corresponding equilibrium
distributions must be equal as well. In other words, the equilibrium
distribution $\mu_{\underline s_s}(d\underline s_f)$ from
Assumption~\ref{ass:generg}(i)({1}) must equal the equilibrium
distribution $\mu_{\underline{\underline\pi s}_s}(d\underline s_f)$
from Assumption~\ref{ass:generg}(i)({2}).
(ii) follows along similar lines.
\end{pf}

\begin{corollary}\label{cor:T2}
Suppose in Theorem~\ref{thm:T2} all slow species are continuous,
$\mathcal I_\circ^s = \varnothing$. Then, the dynamics
of \eqref{eq:T21} is the same among the first two cases
(1) and~(2), and among the last two cases
(3) and (4).
\end{corollary}

\begin{pf}
Since\vspace*{2pt} $\mathbf P_{\underline s_s} (d\underline{\underline v}_s)$ is
the delta-measure on $\underline{\underline{\pi s}}_s$, all
assertions can be read directly
from \eqref{eq:mea-1}--\eqref{eq:mea-4} together with
Lemma~\ref{lem:equivofeq}.
\end{pf}

We note that the case (1) (where all
species move faster than the fast reactions occur) plays a special
role under mass action kinetics.

\begin{corollary}[(Homogeneous mass action kinetics)] \label{cor:hom}
Suppose that in Theorem~\ref{thm:T2} the reaction rates are given
by mass action kinetics with constants satisfying the homogeneity
condition
\[
\kappa_{k}:=|\mathcal D| \kappa_{kd}\prod
_{i\in\mathcal
I}\pi_{i}(d)^{\nu_{ik}}.
\]
Then the dynamics of $\underline S_s$ in case (1) is
the same as for the system regarded as a single compartment.
\end{corollary}

\begin{pf} For (1) from \eqref{eq:rates-1}, we only need to
calculate the average with respect to equilibrium of the movement
dynamics for both slow and fast species. The same calculation as for
mass action kinetics in Corollary~\ref{cor:CT1} we get the
first equality in
%
\begin{eqnarray*}
\tilde\lambda^{\mathrm{CR(1)}}_{k}(\underline s_f,
\underline s_s) & =& \sum_{d\in\mathcal D}\int
\lambda_{kd}^{\mathrm
{CR}}(\underline{v}_{\cdot d,f},\underline{
v}_{\cdot
d,s})\mathbf P_{(\underline s_f,\underline
s_s)}(d\underline{\underline
v}_f,d\underline{\underline v}_s)
\\
& = &\sum_{d\in\mathcal D} \kappa_{kd}\prod
_{{i\in\mathcal
I_\circ}} \nu_{ik}! \pmatrix{s_i
\cr
\nu_{ik}} \pi_i(d)^{\nu_{ik}}\cdot\prod
_{{i\in\mathcal I_\bullet}} \bigl(\pi_i(d)s_i
\bigr)^{\nu_{ik}}
\\
& = &\kappa_k \prod_{{i\in\mathcal
I_\circ}}
\nu_{ik}!\pmatrix{s_i
\cr
\nu_{ik}} \prod
_{{i\in\mathcal I_\bullet}}s_i^{\nu_{ik}}.
\end{eqnarray*}
Since the right-hand side gives the reaction rates for mass action
kinetics within a single compartment, as given through $\underline
V_{f|\underline v_s}$ from \eqref{eq:crnlimit-fast}, the equilibrium
$\mu_{\underline v_s}(d\underline z)$ from
Assumption~\ref{ass:crnerg}(i) and $\mu_{\underline s_s}(d\underline
s_f)$ from Assumption~\ref{ass:generg}(i)(1) must be the same
and the assertion follows.
\end{pf}

\begin{example}[(Production from fluctuating source in multiple compartments)]
We \label{ex:productionSp} consider reaction kinetics from
Example~\ref
{ex:production}
and extend it to a spatial multi-compartment setting. Recall that the
chemical reaction network is given (within compartments) by the set
of reactions
%
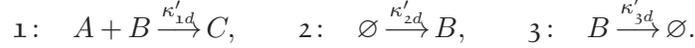
\begin{eqnarray*}
\mathfrak1\dvtx\quad A + B \mathop{\xrightarrow}^{\kappa_{\mathfrak1d}'} C,\qquad
\mathfrak2\dvtx\quad \varnothing\mathop{\xrightarrow}^{\kappa_{\mathfrak
2d}'} B,\qquad
\mathfrak3\dvtx\quad B \mathop{\xrightarrow}^{\kappa_{\mathfrak3d}'} \varnothing.
\end{eqnarray*}
We consider $\Lambda_{k}^{\mathrm{CR}}(\underline x)$ as
in \eqref{eq:second14} with $\kappa_{k}'$ replaced by
$\kappa_{kd}'$, $k\in\{\mathfrak1, \mathfrak2, \mathfrak3\}$. We
have $\underline{\underline x} = (\underline{x}_{\cdot
d})_{d\in\mathcal D}$, $\underline{x}_{\cdot d} = (x_{Ad},
x_{Bd})$, and the dynamics are given by
%
\begin{eqnarray*}
\Lambda_{\mathfrak1d}^{\mathrm{CR}}(\underline x_{\cdot d}) =
\kappa_{\mathfrak1d}' x_{Ad} x_{Bd},\qquad
\Lambda_{\mathfrak2d}^{\mathrm{CR}}(\underline x_{\cdot d}) =
\kappa_{\mathfrak2d}',\qquad \Lambda_{\mathfrak
3d}^{\mathrm{CR}}(
\underline x_{\cdot d}) = \kappa_{\mathfrak
3d}'
x_{Bd}.
\end{eqnarray*}
Movement of species is given as in \eqref{eq:crn1}. Scaling in
each compartment is as in the nonspatial setting
\eqref{eq:second14b}, \eqref{eq:second15} and \eqref{eq:second15b},
so rescaled species counts are
\[
v_{Ad} = N^{-1}x_{Ad},\qquad v_{Bd} =
x_{Bd},
\]
%
and rates are
\[
\lambda_{\mathfrak1d}^{\mathrm{CR}}(\underline v_{\cdot d}) =
\kappa_{\mathfrak1d} v_{Ad} v_{Bd},\qquad \lambda_{\mathfrak
2d}^{\mathrm{CR}}(
\underline v_{\cdot d}) = \kappa_{\mathfrak2d},\qquad \lambda_{\mathfrak3}^{\mathrm{CR}}(
\underline v_{\cdot d}) = \kappa_{3d} v_{Bd}.
\]
%
The process $\underline{\underline V}^N = (V_{Ad}^N, V_{Bd}^N)$ is
given as in \eqref{eq:1gene312b} and additional movement terms. We
set $\eta_s = \eta_A$ for movement of slow species and $\eta_f =
\eta_B$ for movement of fast species. We assume (as in
Assumption~\ref{ass:migerg}) that movement of species $A,B$ have
stationary probability distributions $(\pi_A(d))_{d\in\mathcal D}$
and $(\pi_B(d))_{d\in\mathcal D}$. We derive the dynamics of $S_A$
as
\[
S_A(t) = S_A(0) - \int
_0^t \bar\lambda^{\mathrm{CR}}
\bigl(S_A(u) \bigr)\,du
\]
for appropriate $\lambda$. Since the slow species $A$ are
continuous, we are in the regime of Corollary~\ref{cor:T2} and we
distinguish the following two cases:

\textit{Dynamics in the cases} (1)${}+{}$(2). We have
\begin{eqnarray*}
&&S_{B|s_A}(t) - S_{B|s_A}(0)\\
&&\qquad = -Y_{\mathfrak1} \biggl( \int
_0^t \int \sum_{d\in\mathcal D}
\kappa_{\mathfrak1d} v_{Ad}v_{Bd} \mathbf
P_{(s_A, S_{B|s_A}(u))} (d\underline v_A, d\underline v_B)
\,du \biggr)
\\
&&\qquad\quad{} + Y_{\mathfrak2} \biggl(\sum_{d\in\mathcal D}
\kappa_{\mathfrak2d} t \biggr) -Y_{\mathfrak3} \biggl( \int
_0^t \int\sum_{d\in\mathcal D}
\kappa_{\mathfrak3d} v_{Bd} \mathbf P_{(s_A, S_{B|s_A}(u))} (d\underline
v_A, d\underline v_B) \,du \biggr)
\\
&&\qquad\mathop{=}\limits^{d} -Y_{\mathfrak1+ \mathfrak3} \biggl( \int_0^t
\biggl(\underbrace{\sum_{d\in\mathcal D} \kappa_{\mathfrak
1d}
\pi_A(d) \pi_B(d)}_{=: \bar\kappa_{\mathfrak1}} s_A
+ \underbrace{\sum_{d\in\mathcal D} \kappa_{\mathfrak
3d}
\pi_B(d)}_{=: \bar\kappa_{\mathfrak3}} \biggr) S_{B|s_A}(u) \,du \biggr)
\\
&&\qquad\quad{} + Y_{\mathfrak2} \biggl( \underbrace{\sum_{d\in\mathcal
D}
\kappa_{\mathfrak2d}}_{=: \bar\kappa_{\mathfrak2}} t \biggr).
\end{eqnarray*}
Hence, the equilibrium of the above process is as in
Example~\ref{ex:production} given by
\[
X\sim\mu_{s_A}(ds_B)= \operatorname{Pois}
\biggl(\frac{\bar\kappa
_{\mathfrak
2}}{\bar\kappa_{\mathfrak3} + \bar\kappa_{\mathfrak
1}s_A} \biggr).
\]
We can now compute $\bar\lambda_{\mathfrak1}^{\mathrm{CR(1)+(2)}}$
from (\ref{eq:rates-final1}) as
%
\begin{eqnarray}
\label{eq:1876} %
 \bar\lambda^{\mathrm{CR(1)+(2)}}_{\mathfrak1}
(s_A) & =&- \int\sum_{d\in\mathcal D}
\kappa_{\mathfrak1d} \pi_A(d)s_A \pi
_B(d) s_{B} \mu_{s_A}(ds_B)
\nonumber
\\[-8pt]
\\[-8pt]
\nonumber
& =& -\sum_{d\in\mathcal D} \kappa_{\mathfrak1d}
\pi_A(d)s_A\pi_B(d) \frac{\bar\kappa_{\mathfrak2}
}{\bar\kappa_{\mathfrak3} + \bar\kappa_{\mathfrak1}s_A} =
-\frac{ \bar\kappa_{\mathfrak1}\bar\kappa_{\mathfrak2} s_A}{
\bar\kappa_{\mathfrak3} + \bar\kappa_{\mathfrak1}s_A}.
\end{eqnarray}

\textit{Dynamics in the cases} (3)${}+{}$(4). We have in each
compartment $d\in\mathcal D$
\begin{eqnarray*}
&&V_{B_d|s_A}(t) - V_{B_d|s_A}(0) \\
&&\qquad=-Y_{\mathfrak1} \biggl( \int
_0^t \int\kappa_{\mathfrak1d}
v_{Ad}V_{B_d|s_A}(u) \mathbf P_{s_A} (d\underline
v_{A_d}) \,du \biggr) + Y_{\mathfrak2} ( \kappa _{\mathfrak2d} t )
\\
&&\qquad\quad{} -Y_{\mathfrak3} \biggl( \int_0^t \int
\kappa_{\mathfrak
3d} V_{B_d|s_A}(u) \mathbf P_{s_A} (d\underline
v_{A_d}) \,du \biggr)
\\
&&\qquad \mathop{=} \limits^{d} -Y_{\mathfrak1+ \mathfrak3} \biggl( \int
_0^t \bigl( \kappa_{\mathfrak1d}
\pi_A(d) s_A + \kappa_{\mathfrak3d} \bigr)
V_{B_d|s_A}(u) \,du \biggr) + Y_{\mathfrak2} (\kappa _{\mathfrak
2d} t ).
\end{eqnarray*}
Hence, the equilibrium of the above process is
\[
X\sim\mu_{s_A}(dv_{B_d})= \operatorname{Pois}
\biggl(\frac{\kappa_{\mathfrak
2d}}{\kappa_{\mathfrak3d} + \kappa_{\mathfrak
1d}\pi_A(d)s_A} \biggr)
\]
and for $\bar\lambda^{\mathrm{CR(3)+(4)}}_{\mathfrak1}$ from
(\ref{eq:rates-final3}) we have
%
\begin{eqnarray}
\label{eq:1877} %
\bar\lambda^{\mathrm{CR(3)+(4)}}_{\mathfrak1}(s_A)
&=&-
\sum_{d\in\mathcal D} \int \kappa_{\mathfrak1d}
v_{Ad}v_{Bd} \mu_{s_A}(dv_{B_d})\mathbf
P_{s_A}(dv_{A_d})
\nonumber
\\[-8pt]
\\[-8pt]
\nonumber
&= &- \sum_{d\in\mathcal D} \frac{\kappa_{\mathfrak
1d}\kappa_{\mathfrak2d}\pi_A(d)s_A}{\kappa_{\mathfrak3d} +
\kappa_{\mathfrak1d}\pi_A(d)s_A}.
\end{eqnarray}
Note that we are in the regime of Corollary~\ref{cor:indepSf}, which
shows that $\bar\lambda^{\mathrm{CR(3)+(4)}}_{\mathfrak1}$ is
independent of $\pi_B$.

\textit{Comparison of dynamics in cases} (1)${}+{}$(2) \textit{and}
(3)${}+{}$(4). Let us compare the case $({1})+({2})$, when the turnover rate of $A$ is given by (\ref{eq:1876}),
and $({3})+({4})$, when the rate is given by
(\ref{eq:1877}). First note that even when the network is spatially
homogeneous (the chemical constants satisfy assumption in Corollary~\ref{cor:hom}) there is a marked difference between the dynamics of
cases $({1})+({2})$ (as in single compartment case) and
cases $({3})+({4})$ depending on the movement equilibria
$\pi_A$ and~$\pi_B$. However, if we additionally suppose the slow
species $A$ are equidistributed $\pi_A(d)=1/|\mathcal D|$ then all
four cases have the same dynamics.
\end{example}

\subsubsection*{Conserved quantities on the fast time scale}
Now, we include \emph{conserved quantities} in our two-scale system
in multiple compartments, that is, we have a two-scale reaction network
with $\operatorname{dim}(\mathcal N(({\underline{\underline\zeta
}}^f)^{\textsc t}
))=:n^f>0$. We will use the same notation as in
Section~\ref{Sec:24}. In particular, $\Theta^f := (\underline
\theta^{c_j})_{j=1,\ldots,n^f}$ are linearly independent vectors which
span the null space of $({\underline{\underline\zeta}}^f)^{\textsc
t}$. Every
$\underline\theta^{c_j}$ has a unique parameter $\alpha_i$ associated
with it, $j=1,\ldots,|\Theta^f|=n^f$. Here, $\Theta^f_\circ$ is the
subset of conserved quantities for which $\alpha_{c_j}=0$, and
$\Theta^f_\bullet$ is the subset of conserved quantities for which
$\alpha_{c_j}>0$. Conservation means that $t\mapsto\langle\underline
\theta^{c_j}, \underline S_{f|\underline s_s}(t)\rangle$ with
$\underline S_{f|\underline s_s}$ from \eqref{eq:crnlimit-1} is
constant, $j=1,\ldots,|\Theta^f|$. We let $S_{c_j}^N = \langle
\underline\theta^{c_j}, \underline S_f^N\rangle$ and $\underline
S^N_c=(S_{c_j}^N)_{i=1,\ldots,|\Theta^f|}$ be the vector of rescaled
conserved quantities. Again, $\mathcal K_{\underline\theta^{c_j}}$ is
the set of reactions such that $\beta_k=\alpha_{c_j}$ and $\langle
\underline\theta^{c_j}, \underline\zeta_{\cdot k}\rangle\neq0$, and
let $\mathcal K^c:=\bigcup_{j=1}^{|\Theta^f|}\mathcal
K_{\underline\theta^{c_i}}$, $\mathcal
K^c_\circ:=\bigcup_{j=1}^{|\Theta^f_\circ|}\mathcal
K_{\underline\theta^{c_j}}$ and $\mathcal
K^c_\bullet:=\bigcup_{j=1}^{|\Theta^f_\bullet|}\mathcal
K_{\underline\theta^{c_j}}$. We still
let $\underline{\underline\zeta}^c$ be the matrix defined
by \eqref{eq:cons-zeta}.

Again, we consider the four cases as given in \eqref{eq:9123}. In
addition, we assume that $\langle\underline\theta^{c_j}, \underline
S^N\rangle$ changes on the time scale $dt$. We write here,
distinguishing fast species, conserved quantities and slow species,
$\underline{\underline v} = (\underline{\underline v}_f, \underline
{\underline v}_c, \underline{\underline v}_s)$ with
$\underline{\underline v}_f = (v_{id})_{i\in\mathcal I^f, d\in
\mathcal
D}$, $\underline{\underline v}_c = (\langle\underline
\theta^{c_j},\underline v_{\cdot d, f}\rangle)_{j=1,\ldots,|\Theta^f|,
d\in\mathcal D}$, $\underline{\underline v}_s =
(v_{id})_{i\in\mathcal I^s, d\in\mathcal D}$, as well as $\underline s
= (\underline s_f, \underline s_c, \underline s_s)$, $\underline s_f =
(s_i)_{i\in\mathcal I^f}$, $\underline s_c = (\langle\underline
\theta^{c_j},\underline
s_f\rangle)_{j=1,\ldots,|\Theta^f|}$, and $\underline s_s =
(s_i)_{i\in\mathcal I^s}$.

\begin{remark}[(Conserved quantities as new species)]
In light of the previous results, one would guess that conserved
quantities on the fast time scale can be handled as if they are new
chemical species, evolving on the slow time scale. However, an
important distinction between slow species and conserved quantities
does exist: movement of conserved quantities occurs on the time
scale $N^{\eta_f}  \,dt$ rather than on the time scale $N^{\eta_s}
\,dt$, on which it occurs for the slow species. This implies \textit{an
important distinction between slow species and conserved
quantities occurs in cases} (2) \textit{and} (3); the
averaging measures over the intermediate time scales treat conserved
and slow species differently.
\end{remark}

Although what follows resembles our previous results, in order to be
able to use them,
we do have to state the assumptions and results for systems
with conserved species explicitly. We omit all the proofs as they
follow analogous steps
to those for systems without conserved species.

\begin{assumption}[(Dynamics of the spatial multi-scale reaction
network with conserved quantities)]
In \label{ass:generg2} each case (1)--(4), the spatial
two-scale reaction network on time scale $N  \,dt$, where
Assumption~\ref{ass:migerg} holds, satisfies the following
conditions:
\begin{longlist}[(iii)]
\item[(i)](1) Given
$(Y_k)_{k\in\mathcal K^f_\circ}$, the time-change equation of the
dynamics of $\underline S_f$ given the values of $\underline
S_s={\underline s}_{s}$ and $\underline S_c={\underline s}_{c}$,
denoted $({\underline S}_{f|({\underline s}_{s}, {\underline
s}_{c})}(t))_{t\geq0}$, given by \eqref{eq:crnlimit-1} with
${\underline S}_{f|{\underline s}_{s}}$ replaced by ${\underline
S}_{f|({\underline s}_{s}, {\underline s}_{c})}$, has a unique
solution, where $\tilde\lambda_{k}^{\mathrm{CR(1)}}(\underline
s_f,\underline s_s)$ is given by \eqref{eq:rates-1}. In addition,
$\underline S_{f|(\underline s_s, \underline s_c)}(\cdot)$ has a
unique stationary probability measure $\mu_{(\underline s_s,
\underline s_c)}(d\underline s_f)$ on $\mathbb{R}_+^{|\mathcal
I^f|}$ with $\langle\underline\theta^{c_j}, \underline
s_f\rangle= s_{c_j}$, $\mu_{(\underline s_s, \underline
s_c)}$-almost surely, $j=1,\ldots,|\Theta^f|$.

(2) Given $(Y_k)_{k\in\mathcal
K_\circ^f}$, the time-change equation of the dynamics of
$\underline S_f$ given the value of $\underline{\underline
V}_s={\underline{\underline v}}_{s}$ and $\underline
S_c={\underline s}_{c}$, denoted $({\underline
S}_{f|(\underline{\underline v}_s, \underline s_c)}(t))_{t\geq
0}$, given by \eqref{eq:crnlimit-2} with ${\underline
S}_{f|\underline{\underline v}_s}$ replaced by ${\underline
S}_{f|(\underline{\underline v}_s, \underline s_c)}$, has a
unique solution, where $\tilde\lambda_{k}^{\mathrm
{CR(2)}}(\underline s_f,\underline{\underline v}_s)$ is given
by \eqref{eq:rates-2}. In addition, $\underline
S_{f|(\underline{\underline v}_s, \underline s_c)}(\cdot)$ has a
unique stationary probability measure $\mu_{(\underline{\underline
v}_s, \underline s_c)}(d\underline s_f)$ on
$\mathbb{R}_+^{|\mathcal I^f|}$ with $\langle\underline
\theta^{c_j}, \underline s_f\rangle= s_{c_j}$,
$\mu_{(\underline{\underline v}_s, \underline s_c)}$-almost
surely, $j=1,\ldots,|\Theta^f|$.

(3) Given $(Y_{kd})_{k\in\mathcal
K_\circ^f, d\in\mathcal D}$, the time-change equation of the
dynamics of $\underline{\underline V}_f$ given the values of
${\underline S}_s={\underline s}_{s}$ and ${\underline{\underline
V}}_c={\underline{\underline v}}_{c}$, denoted by
$(\underline{\underline V}_{f|({\underline
s}_{s},{\underline{\underline v}}_{c})}(t))_{t\geq0}$, given
by \eqref{eq:crnlimit-3} with $\underline V_{\cdot d,f|{\underline
s}_{s}}$ replaced by $\underline V_{\cdot d,f|({\underline
s}_{s},{\underline{\underline v}}_{c})}$, has a unique
solution, where $\tilde\lambda_{kd}^{\mathrm
{CR(3)}}(\underline{\underline v}_f,\underline s_s)$ is given by
\eqref{eq:rates-3}. In addition, $\underline{\underline
V}_{f|({\underline s}_s, {\underline{\underline
v}}_{c})}(\cdot)$ has a unique stationary probability
measure $\mu_{({\underline s}_s, {\underline{\underline
v}}_{c})}(d\underline{\underline v}_f)$ on
$\mathbb{R}_+^{|\mathcal I^f| \times|\mathcal D|}$ with $\langle
\underline\theta^{c_j}, \underline{\underline v}_f\rangle=
\underline v_{c_j}$, $\mu_{(\underline{s}_s, \underline{\underline
v}_c)}$-almost surely, $j=1,\ldots,|\Theta^f|$.

Moreover, given $\underline s_s$ and $\underline s_c$, the
movement dynamics of $\underline{\underline V}_{c|(\underline s_s,
\underline s_c)}$ is a unique solution $\underline{\underline
V}_{c|(\underline s_s, \underline s_c)}(\cdot) = (\underline
V_{\cdot d, c|(\underline s_s, \underline s_c)})_{d\in\mathcal D}$
of the time-change equations
%
\begin{eqnarray}
\label{eq:Vc1}  &&\bigl\langle\underline\theta^{c_j},
\underline V_{\cdot d,
f} \bigr\rangle_{|(\underline s_s, \underline s_c)}(t) - \bigl\langle
\underline \theta^{c_j},\underline V_{\cdot d, f} \bigr\rangle
_{|(\underline s_s, \underline s_c)}(0)
\nonumber
\\
&&\qquad =  \sum_{i\in\mathcal I^f_\circ} \theta^{c_j}_i
\sum_{d', d''\in\mathcal D} \bigl(\delta_{d''}(d)-
\delta_{d'}(d) \bigr)\nonumber
\\
&&\hspace*{65pt}\qquad\quad{} \times Y_{i,d',d''} \biggl( \int_0^t
\lambda^{\mathrm{M}}_{i,d',d''} \int v_{id'} \mu_{(\underline
s_s, \underline{\underline V}_{c|(\underline s_s, \underline
s_c)}(u))}
(d\underline{\underline v}_f)\,du \biggr),\nonumber\\
&&\hspace*{308pt}\theta^{c_j} \in\Theta^f_\circ,
\\
&&\bigl\langle\underline\theta^{c_j}, \underline
V_{\cdot d,
f} \bigr\rangle_{|(\underline s_s, \underline s_c)}(t) - \bigl\langle\underline
\theta^{c_j},\underline V_{\cdot d, f} \bigr\rangle
_{|(\underline s_s, \underline s_c)}(0)
\nonumber
\\
&&\qquad = \sum_{i\in\mathcal I^f_\bullet} \theta^{c_j}_i
\sum_{d'\in\mathcal D} \int_0^t
\int \bigl(\lambda^{\mathrm{M}}_{i,d',d} v_{id'} -
\lambda^{\mathrm{M}}_{i,d,d'} v_{id} \bigr)
\nonumber
\\
&&\hspace*{80pt}\qquad\quad{} \times\mu_{(\underline s_s, \underline{\underline
V}_{c|(\underline s_s, \underline s_c)}(u))}
(d\underline{\underline v}_f) \,du,\qquad
\theta^{c_j} \in\Theta^f_\bullet,
\nonumber
\end{eqnarray}
with an equilibrium probability distribution of movement $\mathbf
P_{(\underline s_s, \underline s_c)}(d\underline{\underline v}_c)$
with $\sum_{d\in\mathcal D} \underline v_{\cdot d, c} = \underline
s_c$, $\mathbf P_{(\underline s_s, \underline
s_c)}(d\underline{\underline v}_c)$-almost surely.

(4) Given $(Y_{kd})_{k\in\mathcal
K_\circ^f, d\in\mathcal D}$, the time-change equation of the
dynamics of $\underline{\underline V}_f$ given the values of
$\underline{\underline V}_s=\underline{\underline v}_{s}$ and
${\underline{\underline V}}_{c} = {\underline{\underline v}}_{c}$,
denoted by $(\underline{\underline V}_{f|({\underline{\underline
v}}_{s},{\underline{\underline v}}_{c})}(t))_{t\geq0}$,
given by \eqref{eq:crnlimit-4} with $\underline V_{\cdot
d,f|{\underline{\underline v}}_{s}}$ replaced by $\underline
V_{\cdot d,f|({\underline{\underline
v}}_{s},{\underline{\underline v}}_{c})}$, has a unique
solution, where $\tilde\lambda_{kd}^{\mathrm
{CR(4)}}(\underline{\underline v}_f,\underline{\underline v}_s)$
is given by \eqref{eq:lambdaCR4}. In addition,
$\underline{\underline V}_{c|(\underline{\underline v}_{s},
{\underline{\underline v}}_{c})}$ has a unique stationary
probability measure $\mu_{(\underline{\underline v}_s,
{\underline{\underline v}}_{c})}(d\underline{\underline v}_f)$
with $\langle\underline\theta^{c_j}, \underline{\underline
v}_f\rangle= \underline v_{c_j}$, $\mu_{(\underline{\underline
{v}}_s, \underline{\underline v}_c)}$-almost surely,
$j=1,\ldots,|\Theta^f|$.

Moreover, given $\underline{\underline v}_s$ and
$\underline s_c$, the movement dynamics of $\underline{\underline
V}_{c|(\underline{\underline v}_s, \underline s_c)}$ is a unique
solution $\underline{\underline V}_{c|(\underline{\underline v}_s,
\underline s_c)} = (\underline V_{\cdot d,
c|(\underline{\underline v}_s, \underline s_c)})_{d\in\mathcal
D}$ of the time-change equations \eqref{eq:Vc1} with
$\underline{\underline V}_{c|(\underline s_s, \underline s_c)}$
replaced by $\underline{\underline V}_{c|(\underline{\underline
v}_s, \underline s_c)}$ with an equilibrium probability
distribution of movement $\mathbf P_{(\underline{\underline v}_s,
\underline s_c)}(d\underline{\underline v}_c)$ with
$\sum_{d\in\mathcal D} \underline v_{\cdot d, c}= \underline s_c$,
$\mathbf P_{(\underline{\underline v}_s, \underline s_c)}$-almost
surely.
\item[(ii)] From
$\{\tilde\lambda_k^{\mathrm{CR}(\ell)}\}_{\ell=1,2,3,4}$, we set
in each case
%
\begin{eqnarray}
\label{eq:rates-final-cons1} \bar\lambda_{k}^{\mathrm{CR(1)}}(\underline
s_s, \underline s_c) &=&\int_{\mathbb{R}_+^{|\mathcal I^f|}}
\tilde\lambda_{k}^{\mathrm{CR(1)}}(\underline s_f,\underline
s_s) \mu_{(\underline s_s, \underline s_c)}(d\underline s_f);
\\
\label{eq:rates-final-cons2} \bar \lambda_{k}^{\mathrm{CR(2)}}(\underline
s_s, \underline s_c) &=&\int_{\mathbb{R}_+^{|\mathcal I^s|\times|\mathcal
D|}}\int
_{\mathbb{R}_+^{|\mathcal I^f|}} \tilde\lambda_{k}^{\mathrm{CR(2)}}(\underline
s_f, \underline{\underline v}_s) \mu_{(\underline{\underline v}_s,
\underline s_c)}(d
\underline s_f) \mathbf P_{\underline
s_s}(d\underline{\underline
v}_s) ;
\\
\label{eq:rates-final-cons3} \bar \lambda_{k}^{\mathrm{CR(3)}}(\underline
s_s, \underline s_c) &=& \int_{\mathbb{R}_+^{|\Theta^f|\times|\mathcal D|}}
\int_{\mathbb{R}_+^{|\mathcal I^f|\times|\mathcal D|}} \tilde\lambda_{k}^{\mathrm{CR(3)}}(
\underline{\underline v}_f,\underline s_s)
\mu_{(\underline s_s,
\underline{\underline v}_c)}(d\underline{\underline v}_f)
\nonumber
\\[-8pt]
\\[-8pt]
\nonumber
&&\hspace*{84pt}{}\times \mathbf
P_{(\underline s_s,\underline
s_c)}(d\underline{\underline v}_c);
\\
\label{eq:rates-final-cons4} \quad\qquad\bar \lambda_{k}^{\mathrm{CR(4)}}(\underline
s_s, \underline s_c) &=&\int_{\mathbb{R}_+^{|\mathcal I^s|\times|\mathcal
D|}}\int
_{\mathbb{R}_+^{|\Theta^f|\times|\mathcal
D|}}\int_{\mathbb{R}_+^{|\mathcal I^f|\times
|\mathcal D|}} \tilde
\lambda_{k}^{\mathrm{CR(4)}}(\underline{\underline v}_f,
\underline{\underline v}_s) \mu_{(\underline{\underline v}_s,
\underline{\underline v}_c)}(d\underline{\underline
v}_f)
 \nonumber
 \\[-8pt]
 \\[-8pt]
 \nonumber
 &&\hspace*{123pt}{}\times\mathbf P_{(\underline{\underline v}_s,\underline
s_c)}(d\underline{\underline
v}_c)\mathbf P_{\underline
s_s}(d\underline{\underline
v}_s).
\end{eqnarray}
For $j=1,2,3,4$, there exists a well-defined process $(\underline
S_s(\cdot), \underline S_c(\cdot))$ that is the unique solution of
%
\begin{eqnarray}
\label{eq:T21c} %
 \underline S_s(t) & =&
\underline S_s(0) + \sum_{k\in\mathcal
K^s_\circ} \underline
\zeta^s_{\cdot k} Y_k \biggl( \int
_0^t \bar\lambda_k^{\mathrm{CR}(\ell)}
\bigl(\underline S_s(u),\underline S_c(u) \bigr) \,du
\biggr)
\nonumber
\\[-8pt]
\\[-8pt]
\nonumber
&&{} + \sum_{k\in\mathcal K^s_\bullet} \underline\zeta^s_{\cdot k}
\int_0^t \bar\lambda_k^{\mathrm{CR}(\ell)}
\bigl(\underline S_s(u),\underline S_c(\cdot) \bigr) \,du,
\end{eqnarray}
and for $j=1,\ldots,|\Theta^f|$,
\begin{eqnarray}
\label{eq:T21d} S_{c_j}(t) & =& S_{c_j}(0)\nonumber
\\
&&{} +
\sum_{k\in\mathcal K^c_\circ} \sum
_{i\in\mathcal I^f} \theta_{c_j}^i \underline
\zeta^s_{ik} Y_k \biggl( \int
_0^t
\bar\lambda_k^{\mathrm{CR}(\ell)}
\bigl(\underline S_s(u),\underline S_c(u) \bigr) \,du
\biggr)
\\
&&{} +
\sum_{k\in\mathcal K^c_\bullet} \sum
_{i\in\mathcal I^f} \theta_{c_j}^i \underline
\zeta^s_{ik} \int_0^t
\bar\lambda_k^{\mathrm{CR}(\ell)} \bigl(\underline S_s(u),
\underline S_c(u) \bigr) \,du.\nonumber
\end{eqnarray}
\item[(iii)] Same as (iii) in Assumption~\ref{ass:crnerg} in each
compartment.
\end{longlist}
\end{assumption}

\begin{remark}[(Equivalent formulation)]
For the dynamics under the above assumption, the following is
immediate: In \label{rem:crnSpa3} each case (1)--(4), the
spatial two-scale reaction network on time scale $dt$, where
Assumption~\ref{ass:generg2} holds, satisfies the following
condition: given $(Y_k)_{k\in\mathcal K^s_\circ\cup\mathcal
K^c_\circ}$, the time change equations \eqref{eq:T21c} and
\eqref{eq:T21d} have a unique solution, with
%
\begin{equation}
\label{eq:T22f} \bar\lambda_k^{\mathrm{CR}(\ell)} ( \underline
s_s ) := \mathbf E_{(\underline s_s, \underline
s_c)} \biggl[\sum
_{d\in\mathcal D} \lambda_{kd}^{\mathrm{CR}}(\underline
V_{ \cdot d}) \biggr] <\infty.
\end{equation}
The distribution of $(V_{id})_{i\in\mathcal I, d\in\mathcal D}$ in
\eqref{eq:T22f} depends on the parameters $\eta_s, \eta_f$ as
follows:
\begin{eqnarray*}
(1) &&\quad \mathbf P_{(\underline s_s,\underline s_c)} (d\underline{
\underline v}_f, d\underline{\underline v}_s) = \mathbf
P_{\underline s_s} (d\underline{\underline v}_s) \int
_{\mathbb R_+^{|\mathcal I^f|}}\mathbf P_{\underline s_f} (d\underline{\underline
v}_f)\mu_{(\underline s_s,
\underline s_c)} (d\underline s_f) ,
\\
(2) &&\quad \mathbf P_{(\underline s_s,\underline s_c)}(d\underline{
\underline v}_f, d\underline{\underline v}_s) = \mathbf
P_{\underline
s_s}(d\underline{\underline v}_s) \int
_{\mathbb R_+^{|\mathcal
I^f|}}\mathbf P_{\underline s_f}(d\underline{\underline
v}_f) \mu_{(\underline{\underline v}_s, \underline s_c)}(d \underline s_f) ,
\\
(3)&& \quad\mathbf P_{(\underline s_s,\underline s_c)}(d\underline{
\underline v}_f, d\underline{\underline v}_s) = \mathbf
P_{\underline
s_s}(d\underline{\underline v}_s) \int
_{\mathbb
R_+^{|\Theta_f|\times|\mathcal D|}} \mu_{(\underline s_s,
\underline{\underline v}_c)} (d\underline{\underline
v}_f)\mathbf P_{(\underline s_s,\underline
s_c)}(d\underline{\underline
v}_c) ,
\\
(4) &&\quad \mathbf P_{(\underline s_s,\underline s_c)}(d\underline{
\underline v}_f, d\underline{\underline v}_s) = \mathbf
P_{\underline
s_s}(d\underline{\underline v}_s) \int
_{\mathbb
R_+^{|\Theta_f|\times|\mathcal D|}} \mu_{(\underline{\underline
v}_s, \underline{\underline v}_c)} (d\underline{\underline
v}_f)\mathbf P_{(\underline{\underline v}_s,\underline
s_c)}(d\underline{\underline
v}_c).
\end{eqnarray*}
\end{remark}

\begin{theorem}[(Heterogeneous two-scale system with conserved fast
quantities)] Let \label{thm:T3} $\underline{\underline V}^N$ be the
vector process of rescaled species amounts for the reaction network
which is the unique solution to \eqref{eq:crn1}. Assume that
$(\underline\alpha, \underline\beta, \gamma=0)$ satisfy two-scale
system assumptions \eqref{eq:two-scale} for some $\mathcal I^f,
\mathcal I^s$ with $\varepsilon=1$ and $\mathcal
N((\underline{\underline\zeta}^f)^{\textsc t})=\operatorname{span}(\Theta
^f)$ [with
$\underline{\underline\zeta}^f$ from \eqref{eq:fast-zeta} and
$\Theta^f$ from \eqref{eq:Thetaf}] within compartments with
conserved quantities $(\theta^{c_j})_{i=1,\ldots,|\Theta^f|}$ on the
fast time scale. In addition, $\eta_i=\eta_f>0$, $i\in\mathcal I^f$,
$\eta_i=\eta_s>0$, $i\in\mathcal I^s$, one of the cases
(1)--(4) holds and Assumption~\ref{ass:generg2} holds.
Then, if $(\underline S^N_s (0), \underline S^N_c
(0))\mathop{\Longrightarrow}\limits^{N\to\infty} (\underline S_s(0), \underline
S_c(0))$, we have joint convergence of the process of rescaled
amounts of slow and conserved quantities $(\underline S^N_s (\cdot),
\underline S^N_c (\cdot))$ to $(\underline S_s(\cdot), \underline
S_c(\cdot))$ in the Skorohod topology, with $\underline S_s$ the
solution of (\ref{eq:T21c}) and $\underline S_c$ the solution of
(\ref{eq:T21d}) with rates given by
(\ref{eq:rates-final-cons1})--(\ref{eq:rates-final-cons4}).
\end{theorem}


Results analogous to Corollary~\ref{cor:indepSf} stating the
irrelevance of the
movement of fast species in cases (3) and (4) does not
carry over to the
case with conserved quantities, since on the time scale $N^{\eta_f} \,dt$
conserved quantities are still preserved and their movement equilibria
affects the end result.

%

\begin{lemma}
In the situation of Theorem~\ref{thm:T3}, assume $\mathcal I^s_\circ
= \varnothing$, that is, all slow species are continuous, and let
${\underline{\underline{\pi s}}}_s := (\pi_i(d)s_i)_{i\in\mathcal
I^s, d\in\mathcal D}$.
\label{lem:equivofeq2}
\begin{longlist}[(ii)]
\item[(i)] For stationary probability measures $\mu_{(\underline
s_s, \underline s_c)}(d\underline s_f)$ of $\underline
S_{f|(\underline s_s, \underline s_c)}$ from
Assumption~\ref{ass:generg2}$\mathrm{(i)}$(1) and
$\mu_{(\underline{\underline v}_s, \underline s_c)}(d\underline
s_f)$ of $\underline S_{f|(\underline{\underline v}_s, \underline
s_c)}$ from $\mathrm{(i)({2})}$, we have
\[
\mu_{(\underline s_s, \underline s_c)}(d\underline s_f) =
\mu_{(\underline{\underline{\pi s}}_s, \underline s_c)}(d\underline s_f).
\]
\item[(ii)] Likewise, for stationary probability measures
$\mu_{({\underline s}_{s}, \underline{\underline
v}_c)}(d\underline{\underline v}_f)$ of $\underline{\underline
V}_{f|({\underline s}_{s}, \underline{\underline v}_c)}$ from
$\mathrm{(i)({3})}$ and $\mu_{(\underline{\underline v}_{s},
\underline{\underline v}_c)}(d\underline{\underline v}_f)$ of
$\underline{\underline V}_{f|(\underline{\underline v}_{s},
\underline{\underline v}_c)}$ from $\mathrm{(i)({4})}$, we
have
\[
\mu_{({\underline s}_{s}, \underline{\underline v}_c)}(d\underline {\underline v}_f) =
\mu_{(\underline{\underline{\pi s}}_{s}, \underline{\underline v}_c)} (d\underline{\underline v}_f).
\]
\end{longlist}
\end{lemma}


\begin{corollary}\label{cor:T3}
Suppose in Theorem~\ref{thm:T3} all slow species are continuous,
$\mathcal I_\circ^s = \varnothing$.
Then dynamics of \eqref{eq:T21c} is the same in
cases (1), (2) and also in cases (3),
(4).\vadjust{\goodbreak}
\end{corollary}

\begin{corollary}[(Homogeneous mass action kinetics)] \label{cor:hom2}
Corollary~\ref{cor:hom} carries over to the same situation as in
Theorem~\ref{thm:T3}.
\end{corollary}


%
\begin{example}[(Michaelis--Menten kinetics in multiple compartments)]
We \label{ex:michmenSp} place Michaelis--Menten reaction kinetics
from Example~\ref{ex:michmen} in a spatial multi-compartment
setting. The chemical reaction network is given (within
compartments) by the set of reactions from \eqref{eq:michmen1}, with
$\kappa_{k}'$ replaced by $\kappa_{kd}'$ in compartment $d$. We
have $\underline{\underline x} = (\underline{x}_{\cdot
d})_{d\in\mathcal D}$, $\underline{x}_{\cdot d} = (x_{Sd}, x_{Ed},
x_{ESd}, x_{Pd})$, and the dynamics in each compartment $d$ is given
by rates \eqref{eq:michmen2} with $\kappa_{k}'$ replaced by
$\kappa_{kd}'$. Movement of species is given as
in \eqref{eq:crn1}. Again, we set $\alpha_S = \alpha_{P}=1,
\alpha_{E} = \alpha_{ES} =0$, and $\kappa_{\mathfrak1d} =
\kappa_{\mathfrak1d}'$, $\kappa_{-\mathfrak1d} =
N^{-1}\kappa_{-\mathfrak1d}'$ and $\kappa_{\mathfrak2d} =
N^{-1}\kappa_{\mathfrak2d}'$ as in \eqref{eq:gene31} so, setting
the rescaled species counts
\[
v_{Sd} = N^{-1}x_{Sd}, \qquad v_{Ed} =
x_{Ed},\qquad  v_{ESd} = x_{ESd}, \qquad v_{Pd} =
N^{-1}x_{Pd},
\]
and $\beta_{\mathfrak1} = 1, \beta_{\mathfrak{-1}} = 1,
\beta_{\mathfrak2} = 1$ as in \eqref{eq:gene2}. We write
%
\[
\lambda_{\mathfrak1d}^{\mathrm{CR}}(\underline v_{\cdot d}) =
\kappa_{\mathfrak1d} v_{Sd} v_{Ed},\qquad \lambda_{\mathfrak{-1}d}^{\mathrm{CR}}(
\underline v_{\cdot d}) = \kappa_{\mathfrak-1d} v_{ESd},\qquad
\lambda_{\mathfrak
2d}^{\mathrm{CR}}(\underline v_{\cdot d}) =
\kappa_{2d} v_{ESd}.
\]
The process $\underline{\underline V}^N = (V_{Sd}^N, V_{Ed}^N,
V_{ESd}^N, V_{Pd}^N)$ is given as in Example \eqref{ex:michmen} plus
additional movement terms. We set $\eta_s = \eta_S = \eta_P$ for
movement of slow species and $\eta_f = \eta_E = \eta_{ES}$ for
movement of fast species. We assume (as in
Assumption~\ref{ass:migerg}) that movement of species $i$ has a
stationary probability distribution $(\pi_i(d))_{d\in\mathcal
D}$. We have $\mathcal I^f = \mathcal I^f_\circ= \{E, ES\}$ and
$\mathcal I^s = \mathcal I^s_\bullet= \{S, P\}$ and $\mathcal K^f =
\mathcal K^s = \{\mathfrak1, \mathfrak{-1}, \mathfrak2\}$,
$\mathcal K_S = \{\mathfrak1, \mathfrak{-1}\}$, $\mathcal K_E =
\mathcal K_{ES} = \mathcal K$ and $\underline{\underline\zeta},
\underline{\underline\zeta}^f, \underline{\underline\zeta}^s$ as in
\eqref{eq:MM4}. For conserved quantities within compartments, we set
$V_{Cd} := V_{Ed} + V_{ESd}$ and note that while movement changes
the values of $V_{Cd}$ the overall sum $S_C= \sum_{d\in\mathcal D}
V_{Cd}:= m$ is a conserved quantity for all times, and thus, the
dynamics of $S_C$ is trivial.

We derive the dynamics of $S_S$ (as in Example~\ref{ex:michmen},
$S_S+S_P$ is a conserved quantity on the slow time scale). We have
from \eqref{eq:T21} that
\[
S_S(t) = S_S(0) - \int
_0^t \bar\lambda^{\mathrm{CR}}
\bigl(S_S(u) \bigr)\,du
\]
for appropriate $\bar\lambda$. Since all slow species are
continuous, we are in the regime of Corollary~\ref{cor:T3} and we
only need to distinguish the following two cases:

\textit{Dynamics in cases} (1)${}+{}$(2).
From \eqref{eq:crnlimit-1} with ${\underline S}_{f|{\underline
s}_{s}}$ replaced by ${\underline S}_{f|({\underline s}_{s},
{\underline s}_{c})}$, we have
\begin{eqnarray*}
&&S_{E|s_S}(t)  - S_{E|s_S}(0)\\
&&\qquad = -Y_{\mathfrak1} \biggl( \int
_0^t \int\sum_{d\in\mathcal D}
\kappa_{\mathfrak1d} v_{Sd}v_{Ed} \mathbf
P_{(s_S, S_{E|s_S}(u))} (d\underline v_S, d\underline v_E)
\,du \biggr)
\\
&&\qquad\quad{} + Y_{\mathfrak{-1} + \mathfrak2} \biggl( \int_0^t \int
\sum_{d\in\mathcal D} (\kappa_{\mathfrak{-1}d} +
\kappa_{\mathfrak{2}d}) v_{ESd} \mathbf P_{(m-S_{E|s_S}(u))} (d\underline
v_{ES}) \,du \biggr)
\\
&&\qquad = -Y_{\mathfrak1} \biggl( \int_0^t
\biggl(\underbrace{\sum_{d\in\mathcal D} \kappa_{\mathfrak1d}
\pi_S(d) \pi_E(d)}_{=: \bar\kappa_{\mathfrak1}} \biggr)
s_S S_{E|s_S}(u) \,du \biggr)
\\
&&\qquad\quad{} + Y_{\mathfrak{-1} + \mathfrak2} \biggl( \int_0^t
\biggl(\underbrace{\sum_{d\in\mathcal D} (\kappa_{\mathfrak{-1}d}
+ \kappa_{\mathfrak{2}d}) \pi_{ES}(d)}_{=: \bar\kappa_{\mathfrak{-1}} +
\bar\kappa_{\mathfrak2}} \biggr) \bigl(m
- S_{E|s_S}(u) \bigr) \,du \biggr).
\end{eqnarray*}
Hence, the equilibrium of the above process is as in
Example~\ref{ex:michmen} given by $X\sim\mu_{(s_S,m)}(ds_E)$ where
\[
X \sim\operatorname{Binom} \biggl(m,\frac{\bar\kappa_{\mathfrak{-1}} +
\bar\kappa_{\mathfrak2}}{\bar\kappa_{\mathfrak{-1}} +
\bar\kappa_{\mathfrak2} + \bar\kappa_{\mathfrak1}s_S} \biggr)
\]
and $S_{ES|s_S}$ has equilibrium $m-X$. We next
compute $\bar\lambda^{\mathrm{CR(1)+(2)}}$ from
(\ref{eq:rates-final-cons1}) as
\begin{eqnarray*}
\bar\lambda^{\mathrm{CR(1)+(2)}} (s_S)& =& \bar
\lambda^{\mathrm{CR(1)+(2)}}_{\mathfrak1} (s_S) - \bar
\lambda^{\mathrm{CR(1)+(2)}}_{\mathfrak{-1}}(s_S)
\\
& = &\sum_{d\in\mathcal D} \int \bigl(\kappa_{\mathfrak1d}
\pi_S(d) \pi_E(d) s_S s_E
\\
&&\hspace*{28pt}{}- \kappa_{\mathfrak{-1}d} \pi_{ES}(d) (m-s_E) \bigr)
\mu_{(s_S,m)}(ds_E)
\\
& =& \sum_{d\in\mathcal D} \kappa_{\mathfrak1d}
\pi_S(d) \pi_E(d) s_S m
\frac{\bar\kappa_{\mathfrak{-1}}+
\bar\kappa_{\mathfrak2}}{\bar\kappa_{\mathfrak{-1}} +
\bar\kappa_{\mathfrak2} + \bar\kappa
_{\mathfrak
1}s_S}\\
&&\hspace*{16pt}{}- \kappa_{\mathfrak{-1}d} \pi_{ES}(d)m
\frac{\bar\kappa_{\mathfrak
1}s_S}{\bar\kappa_{\mathfrak{-1}} + \bar
\kappa
_{\mathfrak2} +
\bar\kappa_{\mathfrak1}s_S}
\\
& = &\frac{
m\bar\kappa_{\mathfrak1}\bar\kappa_{\mathfrak2}s_S
}{\bar\kappa_{\mathfrak{-1}} + \bar\kappa_{\mathfrak2} +
\bar\kappa_{\mathfrak1}s_S}.
\end{eqnarray*}
Comparing this with \eqref{eq:michmenVS}, we see that in cases
(1)${}+{}$(2) Michaelis--Menten kinetics in multiple compartments
equals the same kinetics in a single compartment, when
$\kappa_{\mathfrak i}$ is exchanged by $\bar\kappa_{\mathfrak i}$,
$i={-\mathfrak1}, {\mathfrak1}, {\mathfrak2}$; compare also with
Corollary~\ref{cor:T3}.

\textit{Dynamics in cases} (3)${}+{}$(4). For simplicity, we
assume that $\lambda^{\mathrm{M}}_{d,d'} :=
\lambda^{\mathrm{M}}_{E,d,d'} = \lambda^{\mathrm{M}}_{ES, d, d'}$,
that is,
movement of $E$ and $ES$ is the same, and hence
$\pi_E(d)=\pi_{ES}(d), d\in\mathcal D$ [we will use this property
for deriving $\mathbf P_{(s_S, m)}(d\underline v_C)$ and $\mathbf
P_{(\underline v_S, m)}(d\underline v_C)$ below]. We will treat the
cases (3) and (4) separately and show the result of
Corollary~\ref{cor:T3} which states that these two cases lead to the
same
limiting dynamics.

(3) From Assumption~\ref{ass:generg2}(i)(3), for
$\underline v_C$ with $\sum_{d\in\mathcal D}v_{Cd}=m$, we have
\begin{eqnarray*}
&&V_{Ed|(s_S, \underline v_C)}(t) - V_{Ed|(s_S, \underline
v_C)}(0)
\\
&&\qquad= - Y_{\mathfrak1d} \biggl(\int_0^t \int
\kappa_{\mathfrak1d}V_{Ed|(s_S, \underline v_C)}(u)v_{Sd} \mathbf
P_{s_S}(dv_{Sd})\,du \biggr)
\\
&&\qquad\quad{} + Y_{(\mathfrak{-1} + \mathfrak2)d} \biggl(\int_0^t (
\kappa_{\mathfrak{-1}d} + \kappa_{\mathfrak{2}d}) \bigl(v_{Cd} -
V_{Ed|(s_S, \underline v_C)}(u) \bigr) \,du \biggr)
\\
&&\qquad = - Y_{\mathfrak
1d} \biggl(\int_0^t
\kappa_{\mathfrak1d}V_{Ed|(s_S, \underline
v_C)}(u)s_{S}\pi_S(d)
\,du \biggr)
\\
&&\qquad\quad{} + Y_{(\mathfrak{-1} + \mathfrak2)d} \biggl(\int_0^t (
\kappa_{\mathfrak{-1}d} + \kappa_{\mathfrak{2}d}) \bigl(v_{Cd} -
V_{Ed|(s_S, \underline v_C)}(u) \bigr) \,du \biggr).
\end{eqnarray*}
Hence, the equilibrium of $V_{Ed|(s_S, \underline v_C)}$ is as in
Example~\ref{ex:michmen} given by
\begin{eqnarray*}
X_d\sim\mu_{(v_{Sd}, v_{Cd})}(ds_{Ed}) =
\operatorname{Binom} \biggl(v_{Cd},\frac{\kappa_{\mathfrak{-1}d} +
\kappa_{\mathfrak2d}}{\kappa_{\mathfrak{-1}d} +
\kappa_{\mathfrak2d} + \kappa_{\mathfrak
1d}s_S\pi_S(d)} \biggr)
(s_{Ed})
\end{eqnarray*}
and $V_{ESd|(s_S, \underline v_C)}$ has equilibrium $v_{Cd}-X_d$. We
compute $\bar\lambda^{\mathrm{CR(3)}}$ from
(\ref{eq:rates-final-cons3}) as
\begin{eqnarray*}
 \bar\lambda^{\mathrm{CR(3)}}(s_S) &=& \bar\lambda^{\mathrm{CR(3)}}_{\mathfrak1}
(s_S) - \bar\lambda^{\mathrm{CR(3)}}_{\mathfrak{-1}}(s_S)
\\
& =& \sum_{d\in\mathcal D} \int\kappa_{\mathfrak1d}
v_{Ed}s_S\pi_S(d)\\
&&\hspace*{25pt}{} - \kappa_{\mathfrak{-1}d}(v_{Cd}-v_{Ed})
\mu_{(v_{Sd}, v_{Cd})}(dv_{Ed}) \mathbf P_{(s_S,m)}(dv_{Cd})
\\
& = &\sum_{d\in\mathcal D} \int \biggl(\kappa_{\mathfrak1d}
v_{Cd} \frac{\kappa_{\mathfrak{-1}d} + \kappa_{\mathfrak
2d}}{\kappa_{\mathfrak{-1}d} + \kappa_{\mathfrak2d} +
\kappa_{\mathfrak1d}s_S\pi_S(d)} s_S \pi_S(d)
\\
&&\hspace*{49pt}{} -\kappa_{\mathfrak{-1}d}v_{Cd} \frac{\kappa_{\mathfrak1d}s_S\pi_S(d)}{\kappa_{\mathfrak{-1}d} +
\kappa_{\mathfrak2d} + \kappa_{\mathfrak1d}s_S\pi_S(d)} \biggr)\mathbf
P_{(s_S,m)}(dv_{Cd})
\\
& =& \sum_{d\in\mathcal
D} \int\frac{ v_{Cd}\kappa_{\mathfrak1d} \kappa_{\mathfrak
2d}s_S \pi_S(d) }{\kappa_{\mathfrak{-1}d} + \kappa_{\mathfrak
2d} + \kappa_{\mathfrak1d}s_S\pi_S(d)}\mathbf
P_{(s_S,m)}(dv_{Cd}).
\end{eqnarray*}
Consider the equilibrium $\mathbf P_{(s_S,m)}(dv_{Cd})$ of movement
dynamics for conserved species $V_{Cd}=V_{Ed}+V_{ESd}$. Since we
assume the same migration dynamics for $E$ and $ES$, the
equilibrium $\mathbf P_{(\underline s_S, m)}(d\underline v_C)$
is given by a multinomial distribution with parameters $m,
(\pi_E(d))_{d\in\mathcal D}$ and $\int v_{Cd} \mathbf
P_{(s_S,m)}(dv_{Cd})=m\pi_E(d), d\in\mathcal D$.

For (4), the overall rate $\bar\lambda^{\mathrm{CR(4)}}$ from
(\ref{eq:rates-final-cons4}) has the same form as
$\bar\lambda^{\mathrm{CR(3)}}$ except that $\mathbf
P_{(s_S,m)}(dv_{Cd})$ is replaced by $\mathbf
P_{(v_{Sd},m)}(dv_{Cd})$. We first derive the equilibrium
probability distribution $\mu_{(v_{Sd}, v_{Cd})}(dv_{Ed})$
as above. Here, we find that $s_S\pi_S(d)$ is replaced by $v_{Sd}$,
leading to
\[
X_d\sim\mu_{(v_{Sd}, v_{Cd})}(dv_{Ed}) =
\operatorname{Binom} \biggl(v_{Cd},\frac{\kappa_{\mathfrak{-1}d} +
\kappa_{\mathfrak2d}}{\kappa_{\mathfrak{-1}d} + \kappa_{\mathfrak
2d} + \kappa_{\mathfrak1d}v_{Sd}} \biggr)
(v_{Ed}).
\]
The conserved quantities $V_{Cd}(\cdot)$ follow the same dynamics
as in case (3) except that $s_S\pi_S(d)$ is replaced by
$v_{Sd}$ and, therefore,
$\mathbf P_{(\underline v_S, m)}(d\underline v_C)$ is a multinomial
distribution with parameters $m, (\pi_E(d))_{d\in\mathcal D}$ as in
case (3). Hence, in the equation for the rates we have
$\int v_{Cd} \mathbf P_{(v_{Sd},m)}(dv_{Cd}) =m\pi_E(d),
d\in\mathcal D$, and the limiting dynamics in cases $({3})$ and
$({4})$ is given by
\[
\bar\lambda^{\mathrm{CR(3)+(4)}}(s_S)  = \sum
_{d\in\mathcal D} \frac{m\pi_E(d)\kappa_{\mathfrak{1}d} \kappa_{\mathfrak
2d}s_S\pi_S(d)}{\kappa_{\mathfrak{-1}d} +
\kappa_{\mathfrak2d} + \kappa_{\mathfrak1d}s_S\pi_S(d)}.
\]
Although we have seen that the dynamics for cases (1)${}+{}$(2),
as well as for (3)${}+{}$(4) is the same, in general
they are quite different from each other, unless some very special
relationships between the chemical constants and movement equilibria
in different compartments are assumed. See the results in
\citet{PfafPop14} on notions of dynamical homogeneity that allow one
to make some interesting conclusions.
\end{example}

\section{Discussion}
\label{Sec:4}

\textit{Specific features and extensions of spatial chemical
reaction models.}

(a) \emph{Heterogeneous reaction and migration
rates.} The reaction rates $\Lambda_{kd}^{\mathrm{CR}}$ in general
depend on the compartment $d$. For the same reason, the outflow of
species $i$ from compartment $d'$, $\sum_{d''\in\mathcal
D}\Lambda^{\mathrm{M}}_{i,d',d''}$ might depend on $i$ and $d'$.
Moreover, it is possible that $\Lambda^{\mathrm{CR}}_{kd}(\underline
x_{\cdot d})$ is zero for some compartments, that is, our model is
flexible enough to restrict some reactions to a subset of
compartments. Analogously, movement of certain species types can be
restricted to only a subset of compartments, that is,
$\Lambda^{\mathrm{M}}_{i,d,d'}$ can also be set to zero for some
$i,d,d'$. The only thing which is required is that every reaction $k$
happens within at least one compartment.

(b) \emph{Geometry of space.} The geometry of the spatial
system has not been explicitly relevant for our results. The reason is
that movement dynamics is assumed to happen at a different time scale
(either faster or slower) than the effective reaction dynamics of
either the slow or fast species. This implies that only the
equilibrium of the movement is relevant for any dynamics occurring on
the respectively slower scale.

(c) \emph{Chemical conformations.} Our model can be extended in
order to model different chemical conformations of chemical species
instead of spatial compartments. For this, let $\mathcal D_i$ be the
set of possible conformations of species $i$. Then any molecule of
species $i$ performs a Markov chain on $\mathcal D_i$ due to changes
in conformation. Moreover, in this case for each type of reaction $k$
its reaction rate $\Lambda^{\mathrm{CR}}_{k,\underline d, \underline
d'}$ might then depend on all conformations of reacting and produced
molecules $\underline d = (d_i)_{i\in\mathcal I}$ and $\underline d' =
(d'_i)_{i\in\mathcal I}$, respectively. For example, our results can
be applied to Michaelis--Menten kinetics with multiple conformations of
the enzyme and of the enzyme-substrate complex [see \citet{Kou}].

(d) \emph{Other density dependent processes.} The model can
also be applied to other density dependent Markov chain models, such
as epidemic or ecological models. Analogous results can also be made
for density dependent stochastic differential models of stochastic
population growth in spatially heterogeneous environments [see
\citet{EvansSchreiber}].

\textit{Conclusions.}
The main conclusion of our paper is the following algorithm for
determining the dynamics of a spatial chemical reaction network:
assume we are given a network of the form \eqref{eq:CRN} in a spatial
context, that is, \eqref{eq:crn0} holds with reaction rates as in
Assumption~\ref{ass:31}; introduce a (large) scaling constant $N$ and
rewrite the dynamics of all species in the form \eqref{eq:crn1} (for
some $\alpha_i$'s, $\eta_i$'s and $\beta_k$'s) assuming
 \eqref{eq:resc10} and  \eqref{eq:resc10sp} hold (admittedly, the
choice of $N$, $\alpha_i$'s and $\beta_k$'s is rather an art than a
science---for simplicity, we are assuming here that this step has been
done already); in addition, suppose every species moves between
compartments as in Assumption~\ref{ass:migerg}; the goal is to
understand the dynamics of overall normalized sums of species over
compartments as given in \eqref{eq:S2}.

There are two cases: either the system is on a single-scale, that is,
\eqref{eq:single-scale} holds, or the system is two-scale, that is,
\eqref{eq:two-scale} holds. (We do not treat higher order scales in
this paper.)
\begin{longlist}[(ii-b)]
\item[(i)] In the single-scale case Theorem~\ref{T1} applies. Essentially,
one has to average all reaction rates of reactions affecting slow
species over the equilibrium distribution of movement of all
species. If reaction rates are given by mass action kinetics,
Corollary~\ref{cor:CT1} applies.

\item[(ii)] The two-scale case is considerably more complicated. Here, every
species is either fast or slow and we have to consider all orders of
the time scale of fast reactions and movement of fast and slow
species. We call $S_f$ the overall sum of normalized fast species and
$S_s$ the overall sum of normalized slow species. Consider the
submatrices of slow and fast reactions, $\underline{\underline
\zeta}^f$ and $\underline{\underline\zeta}^s$ from
\eqref{eq:fast-zeta} and \eqref{eq:slow-zeta}, respectively. A
conserved quantity for the fast reaction subnetwork is a nontrivial
element of the null-space of $(\underline{\underline\zeta
}^f)^{\textsc t}$.

\item[(ii-a)] If there is no conserved quantity, we can use
Theorem~\ref{thm:T2}. Here, there are up to four time scales to
consider, movement of fast and slow species, the time scale of the fast
reactions and the time scale of the slow species.
In all cases, in order to determine the effective rate on $S_s$ on a
slower time scale, one has to average over the equilibrium of all
higher time scales. Interestingly, if all slow species are continuous
(i.e., have a deterministic process as a limit), it only matters if
the fast species move faster or slower than fast reactions. The speed
of the movement of slow species does not matter (see
Corollary~\ref{cor:T2}).

\item[(ii-b)] If there are conserved quantities for the fast reaction
subnetwork, these conserved quantities can still change on a slower
time scale. Here, we are assuming that this time scale is the same as
the time scale of the slow species.
The main difference from the case without conserved quantities is that
on the fast time scale, the equilibria we need to consider for
averaging are concentrated on a fixed conserved quantity. Then,
basically, the conserved quantity can be treated as new species with
its own dynamics (which changes on the timescale of slow species by
assumption). Again, there are four cases to consider; see
Theorem~\ref{thm:T3}. Also, if all slow quantities are continuous, it
only matters if the fast species move faster or slower than the fast
reactions; see Corollary~\ref{cor:T3}.
\end{longlist}
%






\printaddresses
\end{document}